\documentclass[11pt,reqno,letterpaper]{amsart}
\usepackage[utf8]{inputenc}
\usepackage[shortlabels]{enumitem}
\setlist[itemize]{
    itemsep=0.4em,
    topsep=0.7em
}
\setlist[enumerate]{
    itemsep=0.4em,
    topsep=0.7em
}

\usepackage[english]{babel}
\usepackage{xcolor}
\usepackage[T1]{fontenc}
\usepackage[normalem]{ulem}
\usepackage{caption}

\usepackage{etoolbox}

\makeatletter
\patchcmd{\abstract}
  {\normalfont\Small}
  {\normalfont\normalsize}
  {}{}
\makeatother

\makeatletter

\def\subsection{\@startsection{subsection}{2}%
  \z@{.9\baselineskip\@plus.2\baselineskip}%
  {.35\baselineskip}%
  {\normalfont\bfseries}}

\def\subsubsection{\@startsection{subsubsection}{3}%
  \z@{.75\baselineskip\@plus.2\baselineskip}%
  {.25\baselineskip}%
  {\normalfont\bfseries}}

\def\paragraph{\@startsection{paragraph}{4}%
  \z@{.6\baselineskip\@plus.2\baselineskip}%
  {-.6em}%
  {\normalfont\bfseries}}

\makeatother

\usepackage[margin=1in]{geometry}



\usepackage{amssymb}
\usepackage{amsfonts}
\usepackage[tikz]{bclogo}
\usetikzlibrary{arrows, arrows.meta, positioning, backgrounds, calc, fit}
\usepackage{lmodern}
\usepackage{textcomp}
\usepackage{mathtools, bm}
\usepackage{amssymb, bm}
\usepackage{amsmath}
\usepackage[unq]{unique}
\usepackage{comment}
\usepackage{xcolor}

\definecolor{linkGreen}{RGB}{0,110,65}

\definecolor{linkblue}{RGB}{0, 102, 204} 
\usepackage[
colorlinks=true,
linkcolor=linkblue,     
citecolor=linkGreen,     
urlcolor=linkblue       
]{hyperref}

\usepackage[numbers,sort&compress]{natbib}
\usepackage{graphicx} 

\usepackage{amsthm}
\usepackage{tikz}
\usepackage{caption}
\usepackage{subcaption}
\usepackage{varioref}

\usepackage{thm-restate}
\usepackage{cleveref}

\newtheorem{theorem}{Theorem}[section]
\newtheorem{thm}{Theorem}
\newtheorem{cor}[thm]{Corollary}
\newtheorem{lemma}[theorem]{Lemma}
\newtheorem{proposition}[theorem]{Proposition}
\newtheorem{corollary}[theorem]{Corollary}

\newtheorem{claim}[theorem]{Claim}

\theoremstyle{definition}
\newtheorem{remark}[theorem]{Remark}
\newtheorem{conjecture}[theorem]{Conjecture}
\newtheorem*{conjecture*}{Conjecture}

\newtheorem{definition}[theorem]{Definition}
\newtheorem{algorithm}[theorem]{Algorithm}

\newcommand{\eps}{\varepsilon}
\newcommand{\imax}{\Gamma}
\newcommand{\e}{\mathrm{e}}

\newcommand{\cA}{\mathcal{A}}
\newcommand{\cB}{\mathcal{B}}

\newcommand{\cE}{\mathcal{E}}
\newcommand{\cF}{\mathcal{F}}
\newcommand{\cG}{\mathcal{G}}

\newcommand{\cL}{\mathcal{L}}

\newcommand{\cP}{\mathcal{P}}

\newcommand{\cT}{\mathcal{T}}

\newcommand{\bP}{\mathbb{P}}
\newcommand{\bE}{\mathbb{E}}

\title{Universality in random graphs via optimal linking systems: trees and beyond}
\date{}

\author[Cohen Antonir]{Asaf Cohen Antonir}

\address{School of Mathematical Sciences, Tel Aviv University, Tel Aviv 6997801, Israel}
\email{asafc1@tauex.tau.ac.il}

\author[Lichev]{Lyuben Lichev}

\address{Institute of Statistics and Mathematical Methods in Economics, TU Wien, A-1040 Vienna, Austria}
\email{lyuben.lichev@tuwien.ac.at}

\author[Zhukovskii]{Maksim Zhukovskii}

\address{School of Computer Science, University of Sheffield, UK}
\email{m.zhukovskii@sheffield.ac.uk}

\makeatletter
\renewcommand{\@setauthors}{}
\makeatother

\begin{document}

\begin{abstract}
We develop a framework for proving universality results in sparse random graphs. As a first application, we show that there exists an absolute constant $C>1$ such that, with high probability, for every fixed constant $\Delta$, the binomial random graph $G(n,C\ln n/n)$ contains every $n$-vertex tree with maximum degree at most $\Delta$.
This answers a question of Montgomery ({\it Advances in Mathematics}, 2019). We also determine, for every $p$ satisfying $C\ln n/n\leq p=n^{-1+o(1)}$, the minimum girth $\ell$ (up to an absolute multiplicative constant) for which with high probability $G(n,p)$ contains all cycle factors of girth at least $\Omega(\ell)$. In particular, with high probability $G(n,C\ln n/n)$ contains all cycle factors of girth at least $C\ln n/\ln\ln n$, which is optimal up to a constant factor. This extends an earlier result of Ferber, Kronenberg, and Luh ({\it Transactions of the American Mathematical Society}, 2019) and significantly improves a corollary of a deep result of Kahn, Lubetzky, and Wormald ({\it Communications on Pure and Applied Mathematics}, 2017). 
One of the key ingredients in the proofs is establishing the optimal depth of linking systems in sparse random graphs.
\end{abstract}

\maketitle

\thispagestyle{empty}

\begin{center}
\small

\textbf{Asaf Cohen Antonir}\\
School of Mathematical Sciences, Tel Aviv University,\\
Tel Aviv 6997801, Israel\\
\texttt{asafc1@tauex.tau.ac.il}

\vspace{0.8em}

\textbf{Lyuben Lichev}\\
Institute of Statistics and Mathematical Methods in Economics,\\
TU Wien, A-1040 Vienna, Austria\\
\texttt{lyuben.lichev@tuwien.ac.at}

\vspace{0.8em}

\textbf{Maksim Zhukovskii}\\
School of Computer Science, University of Sheffield, UK\\
\texttt{m.zhukovskii@sheffield.ac.uk}

\end{center}

\clearpage
\setcounter{page}{1}

\section{Introduction}

A graph $G$ is said to be {\it  universal for a family of graphs $\mathcal{F}$} (or simply {\it $\mathcal{F}$-universal}) if it contains an isomorphic copy of every graph in $\mathcal{F}$ as a subgraph. 
Constructions of sparse universal graphs arise naturally in the study of VLSI (Very-Large-Scale Integration) circuit design. 
These graphs, together with related notions such as universal sets and induced universality, have found numerous applications in mathematics, computer science, and engineering~\cite{Alon-implicit,ABS,ABZZ,ADK,chips,BDSZZ,circuits,KNR}.
 
In response to the mentioned developments, universal graphs have attracted considerable attention in both probabilistic~\cite{ADILS25,BHKMPP19,DKRR,FKL19,FerberNenadov,FNP,JKS13,KimLee,Mon19} and deterministic (typically pseudorandom)~\cite{Alon-survey,AlonCapalbo,AKS07,BCPS10,BCLR,CapalboR,Capalbo,ChungG,ChungG2,ChungG3,FP87,HMMP-S25,KimKim,JKS13} settings. This theoretical interest is also motivated by the challenging nature of universality. 
In particular, while the property of a random graph to contain a copy of a fixed graph has been a central theme in probabilistic combinatorics for decades, understanding universality for large families of graphs $\cF$ is much more demanding and known results are usually suboptimal. 
In particular, a union bound over all elements in $\cF$ often fails catastrophically and fine understanding of structural properties which are common for all elements in the family $\cF$ is crucial.

In this paper, we develop a framework for studying universality in random graphs and show that it leads to essentially optimal results for several graph families. In particular, we answer a question of Montgomery~\cite{Mon19} concerning tree universality and settle a longstanding gap in the theory of cycle-factor universality. 
One of the key ingredients of our proof is a construction of optimal-depth linking systems, which was a central missing ingredient for obtaining tight universality results in sparse random graphs. While linking systems have played a central role in numerous combinatorial embedding problems~\cite{DMMPS24,FKL19,HMMP-S25,Montgomery-old}, previously known techniques did not allow them to be exploited efficiently in random graphs at densities just above the connectivity threshold.

The remainder of the introduction is organised as follows. Section~\ref{sc:intro-trees} surveys previous work on tree universality in random graphs, while Section~\ref{sc:intro-cycles} reviews results on cycle factors in random graphs. In Section~\ref{sc:intro-linking}, we discuss linking systems and their applications to graph embedding. Our main results are stated in Section~\ref{sc:intro-results}, and Section~\ref{sc:intro-proofs} provides a brief outline of the proofs and the main new ideas behind them. Finally, we conclude the introduction with a collection of open problems and conjectures in Section~\ref{sc:intro-open}.

\subsection{Tree universality}
\label{sc:intro-trees}

For a positive integer $\Delta$, define $\mathcal{T}_{n,\Delta}$ to be the set of all $n$-vertex trees with maximum degree at most $\Delta$. In their foundational work, Friedman and Pippenger~\cite{FP87} proved that there exists a $\mathcal{T}_{n,\Delta}$-universal graph on at most $Cn$ edges, where the constant $C=C(\Delta)$ depends only on $\Delta$ and not on $n$ (and that actually all graphs with good enough expansion properties and large enough vertex sets are such).
Note that the number of vertices in universal graphs considered in~\cite{FP87} is larger than $n$.
The result of Friedman and Pippenger was later refined in a series of works (see, e.g.,~\cite{BCPS10,BCLR,DK08,DKN22,Hax01}). 
In particular, Bhatt, Chung, Leighton, and Rosenberg~\cite{BCLR} constructed a $\mathcal{T}_{n,\Delta}$-universal graph {\it on $n$ vertices} with $O(n)$ edges (where the implied constant depends only on $\Delta$, as above). The number of edges is clearly optimal, up to a constant factor.

Clearly, for any constant $C$ and large $n$, a typical graph on $Cn$ edges is not $\mathcal{T}_{n,\Delta}$-universal. In particular, the random graph  $G(n,C/n)$\footnote{In the binomial random graph $G(n,p)$ edges between every pair of vertices from $\{1,\ldots,n\}$ appear independently with probability $p=p(n)$.} contains isolated vertices whp\footnote{With high probability, that is, with probability tending to 1 as $n\to\infty$.}. Nevertheless, the Friedman--Pippenger theorem was used by Alon, Krivelevich and Sudakov~\cite{AKS07} to prove that $G(n,C/n)$ is universal for slightly smaller trees: for every $\varepsilon>0$, there exists $C>0$ such that whp $G(n,C/n)$ is $\mathcal{T}_{(1-\varepsilon)n,\Delta}$-universal. 
 This raises the natural question: at which density $p$ does the random graph become $\mathcal{T}_{n,\Delta}$-universal? 

On the one hand, universality cannot occur below the connectivity threshold $p=\ln n/n$ established in the celebrated paper of Erd\H{o}s and R\'{e}nyi~\cite{ER}. Nevertheless, as $p$ gets slightly larger, namely when $pn=\ln n+\ln\ln n+\omega(1)$, the random graph whp contains the unique member of $\mathcal{T}_{n,2}$ --- a Hamilton path~\cite{Ham1,Ham2,Ham3}. 
About thirty years ago, Kahn formulated the natural conjecture (which was stated in print later in~\cite{KLW16}) that, for every fixed $\Delta$, there exists a constant $C$ such that, for any tree $T\in\mathcal{T}_{n,\Delta}$, the random graph $G(n, C\ln n/n)$ contains a copy of $T$ whp. Kahn, Lubetzky, and Wormald \cite{KLW16} verified this conjecture for a particular family of trees known as combs. The conjecture also follows readily for trees with $\Theta(n)$ leaves from the results of Alon, Krivelevich, and Sudakov, as was observed in~\cite{AKS07,Kri10}. In fact, for such trees, Hefetz, Krivelevich, and Szab\'o~\cite{HKS12} showed that the threshold $p=\ln n/n$ is sharp. 
The same paper established an analogous result for a somewhat opposite case --- bounded degree trees containing a bare path of length $\Theta(n)$. A subsequent work by Glebov, Johannsen, and Krivelevich achieved even a hitting time result for such trees (see the dissertation~\cite{GJK13}, although a paper with this result has never been published). Earlier, Krivelevich~\cite{Kri10} proved that every tree $T$ on $n$ vertices with maximum degree at most $\Delta=\Delta(n)$ appears in $G(n,p)$ whp whenever $np \ge C\max\{\Delta\ln n,n^{\eps}\}$ for sufficiently large constant $C>0$ and any fixed $\eps>0$. 
Although this bound is optimal when $\Delta\in [n^{\eps},n/\ln n]$, it provides only rather crude estimates of the threshold density for trees with smaller maximum degree.

To the best of our knowledge, the first result on the $\mathcal{T}_{n,\Delta}$-universality of $G(n,p)$ was obtained as a consequence of a more general theorem of Dellamonica, Kohayakawa, R\"{o}dl, and Ruci\'{n}ski~\cite{DKRRSODA}, 
which established universality for the family of {\it all}\, graphs with maximum degree at most $\Delta$. When specialised to $\mathcal{T}_{n,\Delta}$, however, their result yields rather suboptimal bounds. 
A few years later, Johannsen, Krivelevich, and Samotij~\cite{JKS13} significantly improved these bounds by focusing specifically on the family $\mathcal{T}_{n,\Delta}$. 
They proved that whp $G(n,p)$ is $\mathcal{T}_{n,\Delta}$-universal when $\Delta=\Delta(n)\geq\ln n$ and $p \ge C \Delta n^{-1/3} \ln n$ for large enough constant $C>0$. 
In particular, this immediately implies that whp $G(n,C(\ln n)^2n^{-1/3})$ is $\mathcal{T}_{n,\Delta}$-universal for every constant $\Delta$. 
 Subsequently, Ferber, Nenadov, and Peter~\cite{FNP} showed that, for every $\Delta=\Delta(n)\ge 2$ and $p=\omega(\Delta^{12}n^{-1/2}(\ln n)^3)$, the random graph $G(n, p)$ is $\mathcal{T}_{n,\Delta}$-universal whp. In particular, for constant $\Delta$, this improved upon the result of Johannsen, Krivelevich, and Samotij, although a substantial gap to the conjectured threshold $p=\Theta(\ln n/n)$ still remained.
Finally, Montgomery~\cite{Mon19} proved that, for every constant $\Delta\ge 2$, there exists a constant $C=C(\Delta)$ such that $G(n, C\ln n/n)$ is $\mathcal{T}_{n,\Delta}$-universal whp.
This established a strong form of the conjecture of Kahn and represents a landmark result in the area. It subsequently inspired further research on universality for bounded-degree trees in a variety of settings, including randomly perturbed graphs~\cite{BHKMPP19}, pseudo-random graphs~\cite{HMMP-S25}, and random geometric graphs~\cite{ADILS25}.

A major open problem highlighted by Montgomery~\cite{Mon19} was to determine whether the dependence of the constant $C$ on $\Delta$ in the bound on the probability threshold can be eliminated. 
This dependence is far from being a mere artefact of the proof. 
Instead, it reflects a fundamental difficulty arising from the need to bridge the gap between trees with many leaves and trees with many long bare paths --- a dichotomy that emerges naturally in the argument.

\subsection{Universality of cycle factors}
\label{sc:intro-cycles}

A {\it cycle factor} in a graph $G$ is a spanning subgraph of $G$ that is isomorphic to a union of vertex-disjoint cycles. In particular, a $C_{\ell}$-factor is a disjoint union of $\ell$-cycles.
The study of threshold phenomena for cycle factors in $G(n,p)$ originated with a series of papers on Hamilton cycles~\cite{Ham1,Ham2,Ham3,Korshunov,Posa}.
In this line of research, the threshold for Hamiltonicity
$$
p_{\rm Ham}=\frac{\ln n+\ln\ln n+O(1)}{n}
$$ 
was identified very precisely and  a corresponding hitting-time result was proved. In particular, before the first Hamilton cycle appears in the random graph process, whp there are no cycle factors at all. 
In contrast, determining the threshold for a triangle factor proved to be a much more difficult problem, requiring several decades to achieve the same level of precision as in the case of Hamilton cycles, thus answering a question of Erd\H{o}s and Spencer~\cite[§5.4]{hitting-question}. After a series of works locating the threshold~\cite{Alon-Yuster,Heckel,JKV,Kahn-Shamir,Krivelevich-factors,Rucinski-factors}, Heckel, Kaufmann, M\"{u}ller, and Pasch~\cite{HKMP24} showed that the hitting time for the emergence of a triangle factor coincides with the hitting time at which every vertex lies in a triangle, with the corresponding density being around $p_{\rm tri} = (2(\ln n)/n^2)^{1/3}$.
We note that~\cite{HKMP24} makes a substantial use of a work of Kahn~\cite{K22} on the hitting time of perfect matchings in the random $k$-uniform hypergraph process, also known as \emph{Shamir's problem}.

For the intermediate values $p_{\rm Ham}<p<p_{\rm tri}$, the cycle factors in $G(n,p)$ are not so well understood. 
It is natural to expect that the minimum $\ell$ such that $G(n,p)$ has a $C_{\ell}$-factor decreases as $p$ increases; note that this is the case when, for some fixed $\eps>0$, $p\ge n^{-1+\eps}$ by the main result of~\cite{FKL19}.  In particular, from the union bound, it is immediately clear that $G(n,\Theta(\ln n/n))$ does not have a $C_{\ell}$-factor for $\ell<(1-\varepsilon)\frac{\ln n}{\ln\ln n}$. 
On the other hand, as a corollary of a deep result on the minimum $k=k(n)$ such that the random $3$-regular graph contains a $k$-cycle factor, Kahn, Lubetzky, and Wormald~\cite{KLW17} proved that whp $G(n,(2+\varepsilon)p_{\rm Ham})$ has a $C_{\ell}$-factor where $\ell\ge c\log_2 n$, with $c\approx 4.82$, which is a $(\ln\ln n)$-factor far from the trivial lower bound.   In the dense case, namely when $p=n^{-1+\Theta(1)}$, the minimum such girth is $\ell=O(1)$. 
In particular, from the recent result of Burghart, Heckel, Kaufmann, M\"{u}ller, and Pasch~\cite{BHKMP} establishing sharp thresholds for $F$-factors for all strictly 1-balanced graphs $F$ and thus confirming the 30-year-old conjecture of Ruci\'{n}ski~\cite{Rucinski-factors}, it follows that $p= (2(\ln n)/n^{\ell-1})^{1/\ell}$ is a sharp threshold for containing a $C_{\ell}$-factor for every fixed $\ell\mid n$.

As in the case of trees, universality results for cycle factors are considerably more difficult to obtain. 
In this direction, 
for every fixed integer $\ell\ge 3$,
Ferber, Kronenberg, and Luh~\cite{FKL19} established a coarse threshold 
$p=\Theta\big((\ln n)^{1/\ell}/n^{1-1/\ell}\big)$ for universality of the family $\mathcal{H}_{n,2,\ell}$, consisting of all $n$-vertex graphs with maximum degree at most $2$ and girth at least $\ell$, which in particular contains $\mathcal{C}_{n,\ell}$, the family of cycle factors of girth at least $\ell$. This improves an earlier suboptimal result of Kim and Lee~\cite{KimLee}.
The proof of Ferber, Kronenberg, and Luh relies heavily on the known thresholds for $C_{\ell}$-factors. 
This poses serious challenges to extend the result to growing values of $\ell$ and, consequently, to the regime $p=n^{-1+o(1)}$.

In particular, it follows from~\cite{FKL19} that $p=\Theta\big(\sqrt[3]{\ln n/n^2}\big)$ is the threshold for $\mathcal{H}_{n,2}$-universality, where $\mathcal{H}_{n,\Delta}$ is the family of all $n$-vertex graphs with maximum degree at most $\Delta$ (so that $\mathcal{H}_{n,2}=\mathcal{H}_{n,2,3}$).
For larger values of $\Delta$, however, much less is known. To the best of our knowledge, the best currently available upper bound on the threshold for $\mathcal{H}_{n,\Delta}$-universality is $O_{\Delta}\big(((\ln n)^{3}/n)^{1/(\Delta-1/2)}\big)$, proved in~\cite{FerberNenadov}. This bound remains very far from the conjectured threshold,
which is believed to match the threshold for the appearance of a $K_{\Delta+1}$-factor~\cite{FKL19}.
We further remark that, for all fixed $\Delta\ge 1$, a uniform upper bound for the appearance of an arbitrary  $n$-vertex graph $H$ of maximum degree at most $\Delta$ is given in~\cite[Theorem~7.2]{FraKahNarPar2021}.

\subsection{Linking systems}
\label{sc:intro-linking}

Fix integers $k,m\ge 1$. Let $\mathbf{a}=(a_1,\ldots,a_m)$ and $\mathbf{b}=(b_1,\ldots,b_m)$ be two disjoint tuples of distinct vertices. 

\begin{definition}[Linking systems]
A graph $G$ on $(k+1)m$ vertices is an \emph{$(\mathbf{a},\mathbf{b};k)$-linking system} if $G$ is a union of forests $F_{\sigma}$, one for each permutation $\sigma\in S_m$, where each forest $F_{\sigma}$ consists of vertex-disjoint paths of length $k$ connecting $a_i$ with $b_{\sigma(i)}$, for every $i$.
Here, $k$ is called the \emph{depth} of the linking system.

With some abuse of terminology, $G$ is called a linking system if it contains a spanning linking system.
\end{definition}
In other words, a linking system provides, for every possible matching between $\mathbf{a}$ and $\mathbf{b}$, a collection of disjoint $k$-paths realising that matching. Linking systems have proved to be a powerful tool for embedding large spanning structures in random graphs and expanders. In particular, Montgomery~\cite{Montgomery-old} proved that $G(n,\Omega((\ln n)^{5}/(n\ln\ln n)))$ contains a spanning linking system of depth $k=\Theta((\ln n)^2)$ whp, and used this result to prove that for every $T\in \mathcal{T}_{n,\Delta}$, whp $G(n,\Omega((\ln n)^5/n))$  contains $T$. The existence of a linking system in~\cite{Montgomery-old} is derived entirely from expansion properties of $G(n,p)$ and the embedding procedure in the host graph relies on a version of the rollback method~\cite{DKN22,Hax01}, which requires fairly strong expansion properties that are not achievable in $G(n,\Theta(\ln n/n))$. 
More recently, \cite{HMMP-S25} and subsequently~\cite{DMMPS24} employed linking systems of depth $k=\Omega((\ln n)^3)$ to establish $\mathcal{T}_{n,\Delta}$-universality and Hamiltonicity, respectively, 
for expanders satisfying substantially weaker expansion conditions than those required in~\cite{Montgomery-old}. We emphasise that these works rely on the partially explicit constructions of linking systems based on the sorting networks of Ajtai, Koml\'{o}s, and Szemer\'{e}di~\cite{AKS}, as explained in more detail below.

A sorting network is a parallel algorithm for sorting $n$ numbers using a prescribed sequence of comparison operations~\cite{AKS,Knuth}. The depth of a sorting network is the number of parallel rounds of comparisons it performs. It is well known that every sorting network has depth $\Omega(\ln n)$, while the Ajtai--Koml\'{o}s--Szemer\'{e}di (AKS) sorting network is asymptotically optimal, having depth $O(\ln n)$. Such a sorting network naturally gives rise to an AKS graph, which serves as the underlying framework for the linking system constructions in~\cite{DMMPS24,HMMP-S25}.  To obtain embeddable linking systems,~\cite{DMMPS24,HMMP-S25} worked with  the AKS graph where every $C_4$ corresponding to a comparator is replaced with an embeddable gadget of size $\Theta((\ln n)^{2})$.
Crucially, because the depth of the AKS sorting network is already asymptotically optimal, this approach {\it cannot} yield linking systems of depth $o(\ln n)$. 

Therefore, existing purely combinatorial techniques --- that do not rely on randomness of the random graph --- fall short of proving the existence of a spanning linking system of presumably optimal depth $\Theta(\ln n/\ln\ln n)$ in $G(n,p)$, just above the connectivity threshold $p=\ln n/n$. This remains one of the main obstacles to obtaining tight tree-universality results in this regime.

\subsection{Main results}
\label{sc:intro-results}

Everywhere in this paper, we assume that the random graph $G\sim G(n,p)$ is sampled on the vertex set $[n]:=\{1,\ldots,n\}$.

Our first result answers positively the question of Montgomery~\cite{Mon19} about the existence of a universal constant factor $C$ in the threshold probability for tree universality. Recall $\cT:=\cT_{n,\Delta}$ denotes the family of all trees on $n$ vertices with maximum degree at most $\Delta$.

\begin{thm}\label{thm:univ_trees}
There exists a constant $C > 1$ such that, for every constant $\Delta\ge 2$, whp $G(n,C\ln n/n)$ is $\cT$-universal. 
\end{thm}

\begin{remark}
\label{rk:constant-factors}
We have not attempted to optimise the constant factor $C$.
We also believe that the proof extends Theorem~\ref{thm:univ_trees} by optimising all dependencies carefully (which we decided to omit for the sake of transparency of presentation) to the regime $\Delta =o((\ln\ln n)^{1/5})$ while keeping the edge probability $p=C\ln n/n$ unchanged. We believe that pushing our methods to the limit should improve this bound to $\Delta=(\ln\ln n)^{1-o(1)}$.
\end{remark}

One of the main components in the proof of Theorem~\ref{thm:univ_trees} is the embedding of linking systems of optimal depth in $G(n,C\ln n/n)$.
 Below, we state the respective result for all densities $C \ln n/n\leq p=n^{-1+o(1)}$.

\begin{thm}\label{thm:linking_system}
Let $C$ be a large enough absolute constant. Consider $\psi=\psi(n)\in\left[1, \frac{\ln n}{4\ln\ln n}\right]$ and an integer 
\[k\ge \frac{50\ln n}{\psi \ln\ln n} \qquad \text{such that}\qquad m = \frac{n}{k+1}\in \mathbb N.\]
Fix disjoint $m$-tuples of different vertices $\mathbf{a},\mathbf{b}\in [n]^m$. 
Then, whp $G\sim G(n,C(\ln n)^{\psi}/n)$ contains an $(\mathbf{a},\mathbf{b};k)$-linking system.
\end{thm}

We note that the depth of a linking system in this theorem is tight up to a constant factor. 
Indeed, by the union bound, it can be easily checked that whp $G(n,C(\ln n)^{\psi}/n)$ does not contain a disjoint union of paths of length $\frac{\ln n}{(1+\eps)\psi\ln\ln n}$  linking the $i$-th element of $\mathbf{a}$ with the $i$-th element of $\mathbf{b}$ for every $i\in[m]$.

The regime $\psi=1$ is both the most challenging and the most relevant case of Theorem~\ref{thm:linking_system}, as it is an important part of the proof of Theorem~\ref{thm:univ_trees}. For convenience and clarity, we state this special case separately.

\begin{cor}\label{cor:linking_system}
Let $k \geq  \frac{50\ln n}{\ln \ln n}$ and $m \coloneq \frac{n}{k+1}$ be integers. Fix disjoint $m$-tuples of different vertices $\mathbf{a},\mathbf{b}\in [n]^m$. 
Then, there is a constant $C>0$ such that whp the random graph $G\sim G(n,C\ln n/n)$ contains an $(\mathbf{a},\mathbf{b};k)$-linking system.
\end{cor}

Again, we did not try to optimise the constant in the above corollary, but we believe that the current proof strategy allows to take $C<2\cdot 10^5$.

Although Corollary~\ref{cor:linking_system} is a key ingredient in the proof of Theorem~\ref{thm:univ_trees}, by itself it only enables us to handle trees containing a collection of disjoint bare paths of length $\Omega(\ln n/\ln\ln n)$ which covers a constant fraction of the vertices. The complementary case requires a delicate analysis and several new ideas; see the proof outline below for further details. 

In fact, Corollary~\ref{cor:linking_system} has applications beyond tree universality. 
In particular, it yields a more general universality theorem for other graph families; see Theorem~\ref{thm:univ_adapted} in Section~\ref{sec:long bare paths}. 
As a special case, we~obtain the following theorem, which extends the result of Ferber, Kronenberg, and Luh~\cite{FKL19} discussed in Section~\ref{sc:intro-cycles} to sparse random graphs. Recall $\mathcal{C}:=\mathcal{C}_{n,\ell}$ is the family of all cycle factors on $[n]$ of girth at least $\ell$.

\begin{thm}\label{thm:main_cycles}
Let $K$ be a large enough constant.
There exists a constant $C>1$ satisfying that, for every $\psi=\psi(n)\in\left[1, \frac{\ln n}{4\ln\ln n}\right]$ and $\ell=\ell(n)\ge \frac{{K}\ln n}{\psi \ln\ln n}$, whp $G\sim G(n,C(\ln n)^{\psi}/n)$ is $\mathcal{C}$-universal.
\end{thm}

\begin{remark}
We have not attempted to optimise the constant factor $K$. However, by optimising all dependencies very carefully, we believe that the current proof strategy could be implemented with some $K<100$. Moreover, we wish to emphasise that the bound on the girth in Theorem \ref{thm:main_cycles} is tight up to a constant factor. Indeed, by the union bound, whp $G(n,C(\ln n)^{\psi}/n)$ does not contain a $C_{\ell}$-factor for all integers $\ell\leq \frac{\ln n}{(1+\varepsilon)\psi\ln\ln n}$.
\end{remark}

The general Theorem~\ref{thm:univ_adapted} has other interesting universality consequences. For example, let $\mathcal{M}_{n,\Delta}$ be the family of all graphs on $[n]$ of maximum degree at most $\Delta$ obtained by subdividing every edge of some graph by at least $D\ln n/\ln\ln n$ vertices, where $D$ is a sufficiently large absolute constant.
Then, for every fixed $\Delta>0$, whp $G(n,C\ln n/n)$ is $\mathcal{M}_{n,\Delta}$-universal, where the constant $C$ does not depend on~$\Delta$. As a consequence, we recover, in a stronger form a recent result of Krivelevich and Nenadov~\cite{KriNen} specialised for random graphs of density $p\geq C\ln n/ n$. 
Their result, for the particular case of random graphs at such densities, claims that whp every graph with at most $\eps n \ln\ln n/\ln n$ vertices and at most $\eps n \ln\ln n/\ln n$ edges appears in $G(n,p)$ as a minor.
Our result yields considerably more: First, for every such graph with at least one edge and maximum degree at most $\Delta$, $G(n,p)$ contains a \emph{spanning subdivision} of it. Second, every graph with at most $\eps n \ln\ln n/\ln n$ vertices and at most $\eps n \ln\ln n/\ln n$ edges can be realised as a minor in $G(n,p)$ by splitting its high-degree vertices into paths so that the overall graph has maximum degree at most $3$ (and this minor can be realised as a spanning subgraph by subdividing edges in the way described above). An analogous minor-universality of $G(n,p)$ for larger graphs follows from our results at higher densities.
Nevertheless, we note that~\cite{KriNen} proves this minor-universality result for all (vertex) expanders at all large enough densities.

\begin{remark}
The proof of Theorem~\ref{thm:univ_trees} extends almost verbatim to yield the following common generalisation of Theorems~\ref{thm:univ_trees}~and~\ref{thm:main_cycles}. 
Call a connected graph {\it complex} if it contains at least two cycles. Let $\mathcal{S}_{n,\Delta}$ denote the family of all $n$-vertex graphs of maximum degree at most $\Delta$ which contain no complex subgraph and no cycle of length less than $\frac{K\ln n}{\ln\ln n}$, where $K$ is as given in Theorem \ref{thm:main_cycles}. 
Then there exists $C>0$ such that for any fixed $\Delta$, whp $G(n,C\ln n/n)$ is $\mathcal{S}_{n,\Delta}$-universal. Nevertheless, we give proofs for the special cases of $\mathcal{T}_{n,\Delta}$ and $\mathcal{C}_{n,\ell}$. This is done both for historical reasons and to keep the exposition as clear and transparent as possible.
\end{remark}

\subsection{Proof strategy}
\label{sc:intro-proofs}

We start by presenting the main ideas in the proof of Theorem~\ref{thm:univ_trees}: we first split the family of trees into two families, where the first family consists of trees with many long bare paths. Then we present the proof strategy for the complement family. Next, we explain how to reduce universality for trees with many long bare paths as well as $\mathcal{C}$-universality in Theorem~\ref{thm:main_cycles} to the existence of a linking system
given by Theorem~\ref{thm:linking_system}. Finally we describe the proof strategy of Theorem~\ref{thm:linking_system}.
  A more detailed overview is presented in Section~\ref{sc:overview}.

Write $G\sim G(n,C\ln n/n)$ for a large absolute constant $C$ and write $\alpha=\Theta(n\ln\ln n/\ln n)$. We recall that a path in a graph $G$ is called \emph{bare} if each of its internal vertices has degree 2 in $G$.
We consider separately a subfamily $\mathcal{T}_1\subset \mathcal{T}_{n,\Delta}$ consisting of all trees with $\Theta(\alpha)$ vertex-disjoint bare paths of length $\Theta(n/\alpha)$ (so that these bare paths cover a small but constant proportion of all vertices). We stress that trees in the complement family $\mathcal{T}\setminus\mathcal{T}_1$ have $\Omega(\alpha)$ leaves.

\paragraph{Trees with $\Omega(\alpha)$  leaves.} We start from trees which do not belong to $\mathcal{T}_1$. The case when the number of leaves is at least $\exp(\Delta^5)\alpha$ is relatively easy\footnote{The choice of the bound $\exp(\Delta^5)\alpha$ is somehow arbitrary, we believe that  $C\Delta\alpha$ suffices, for large enough constant $C$. }, and the proof applies a combination of existing ideas; we explain the strategy in this regime first.
A key tool for embedding such trees is the embeddability/rollback technique developed by Glebov, Johannsen, and Krivelevich~\cite{GJK13,GJK-unpublished}, by Montgomery~\cite{Mon19} and by Dragani\'c, Krivelevich, and Nenadov~\cite{DKN22} on the basis of the foundational work of Friedman and Pippenger~\cite{FP87} and ideas of Johannsen, see also~\cite{DK08,Hax01}.
Following Montgomery~\cite{Mon19}, we also make use of a version of the notion of matchmaking sets, originally introduced by Glebov, Johannsen, and Krivelevich~\cite{GJK13,GJK-unpublished}. 

In essence, for $T\notin\mathcal{T}_1$, we first find a pair of matchmaking sets $U_L$ and $U_P$, each of size $|U_L|=|U_P|=\Omega(\Delta\alpha)$, which are used to embed certain leaves of $T$ and their parents. The matchmaking property of these sets enables us to `absorb'    small sets of vertices while avoiding a set $\tilde U$ of size $o(\alpha)$ consisting of bad vertices. More precisely, for any two disjoint sets $S_L,S_P\subseteq [n]\setminus(\tilde U\cup U_L\cup U_P)$ of equal size, provided this size differs from $|U_L|=|U_P|$ by at most a constant factor (say, a factor of $9$), there exists a perfect matching in $G$ between $U_L\cup S_L$ and $U_P\cup S_P$.

Let $L$ be a subset of the set of leaves of $T$ of size $10|U_L|$ say (we need these leaves to be well-separated, see details in Section~\ref{sc:overview}) and let $P$ be the set of their parents in $T$. Since the number of leaves is sufficiently large, a carefully executed rollback procedure, applied over several rounds, enables us to embed the entire tree $T\setminus L$ into $G\setminus U_L$, leaving at most $10|U_L|$ pairwise well-separated leaves of $T$ unembedded. At the same time, we ensure that the sets $U_P$ and $\tilde U$ are completely covered and that $U_P$ lies entirely within the image of $P$. It then remains to find a perfect matching between the set of 
not yet occupied vertices in $G$, which contains $U_L$, and the image of $P$, which is a superset of $U_P$. This is possible precisely because $U_L$ and $U_P$ form a pair of matchmaking sets.

The embedding of trees $T\notin\cT_1$ with fewer than $\exp(\Delta^5)\alpha$ leaves is considerably more challenging, constitutes the most technical part of this work, and requires several new ideas. 
The reason why this regime is more delicate is that a constant number of sets with the matchmaking property cannot absorb the buffer necessary for the rollback embedding.
To handle this case, we introduce a vertex-layering $L_1,\ldots,L_{\Gamma+1}$ of part of the tree where every vertex has a single neighbour on an upper layer.
The purpose of this construction is to increase the absorption capacity: namely, we identify disjoint sets $U_{L_1},\ldots,U_{L_{\imax}}\subseteq[n]$ of suitable size, one for each layer instead of just $U_L$ as in the previous case, which will be covered by the images of $L_1,\ldots,L_{\imax}$, respectively.

This construction is related to many technical problems. 
For example, while the natural way to construct the layers $L_1,\ldots,L_{\imax+1}$ is via leaf-cutting, a universality statement would require to treat all $n^{\imax+o(1)}$ choices of layer sizes simultaneously.
We cannot achieve this with a union bound since a single vertex is isolated with probability $\approx n^{-C}\gg n^{-\Gamma+O(1)}$.
To resolve this problem, we shift some of the trees in the layering to reduce the number of possible choices of layers to a constant.
Another problem is the fact that the degrees of the parents of a given layer may vary between $1$ and $\Delta-1$, so constructing perfect matchings is not enough.
We overcome this difficulty by developing more flexible tools for constructing many-to-one matchings with certain localisation properties.

Furthermore, the leaf-cutting procedure may fail, that is, it may terminate before reaching round $\imax$. 
In this case, we obtain a tree $T'$ with few leaves but many long bare paths, and hence $T'$ essentially belongs to the class $\mathcal{T}_1$ (although it has fewer than $n$ vertices). The embedding strategy then combines elements of the approaches used for trees with many long bare paths and for trees with many leaves when the leaf-cutting procedure succeeds. A further complication arises from the fact that the forests generated during the leaf-cutting process may extend from long bare paths of $T'$ that are intended to be embedded via a linking system. Since the linking system provides no control over the specific roles played by the vertices involved, it is not possible to prescribe in advance which vertices will serve as attachment points for the layered forests. 
To overcome this difficulty, we carefully interlace the two approaches.
Namely, we split the long paths in two groups: some are embedded via the linking system (see next part), and others cover a set $U_P$ as above by parents of the last successfully constructed layer.

\paragraph{Trees with many long bare paths and cycle factors.} 
In order to reduce $\mathcal{T}_1$-universality to the existence of a linking system, assume $G=G_1\cup G_2\cup G_3$ where all $G_i\sim G(n,p')$ are sampled independently with $p'>p/3$. A na\"{i}ve strategy would be as follows.

Partition $[n]$ into disjoint sets $V_1,V_2,V_3$, each of size $\Theta(n)$. Find two disjoint sets $A_1,A_2\subseteq V_3$, each of size $\alpha$, and partition $V_2$ into sets $U_1,\ldots,U_N$ of size slightly smaller than $\alpha$ in such a way that all sets $A_1$, $A_2$, and $U_i$ are matchmaking in $G_2$ whp. This can be achieved in a relatively straightforward manner using the celebrated Lov\'{a}sz Local Lemma. Since $C$ is sufficiently large, Corollary~\ref{cor:linking_system} implies that whp $G_3[V_3]$ forms a linking system of depth $50\ln n/\ln\ln n$ linking the sets $A_1$ and $A_2$. 
Now, fix a tree $T\in\mathcal{T}_1$. Remove from $T$ a collection of $\alpha$ bare paths of length $\Theta(n/\alpha)$, and let $F$ denote the resulting forest. Embed $F$ into $G_1[V_1]$, leaving at least $\varepsilon n$ vertices of $V_1$ unused. 
Such an embedding exists whp by the almost-spanning universality 
  result of Alon, Krivelevich, and Sudakov~\cite{AKS07}. To complete the embedding of $T$, 
it then remains to connect corresponding vertices from two disjoint sets $X_1,X_2\subseteq V_1$ of size $\alpha$ by pairwise vertex-disjoint paths of length $\Theta(n/\alpha)$.

If the sets $X_1$ and $X_2$ possessed absorption properties comparable to those of $A_1$, $A_2$, and the sets $U_i$, the task would be straightforward. Indeed, one could distribute the remaining vertices of $V_1$ among the sets $U_i$ in an arbitrary balanced way and then obtain perfect matchings between the consecutive pairs $(X_1,U_1)$, $(U_1,U_2),$ $\ldots$, $(U_{N/2},A_1)$, and
$(A_2,U_{N/2+1})$, $\ldots$, $(U_N,X_2)$. These matchings would naturally yield the required collection of disjoint paths.
The difficulty is that we have essentially no control over the locations of $X_1$ and $X_2$ inside $V_1$. In particular, we cannot even guarantee that every vertex of $U_1$ has a neighbour in $X_1$. To overcome this obstacle, we augment $X_1$ and $X_2$ with additional matchmaking sets selected in advance from $V_2$, alongside the sets $U_i$.

The reduction of $\mathcal{C}_{n,\ell}$-universality for $\ell=\frac{K\ln n}{\ln\ln n}$ is verbatim since, after deleting some long bare paths from any cycle-factor of girth at least $\ell$, 
 we get a linear forest $F$ which can be similarly embedded into $G_1[V_1]$. 
For all $\psi>1$, we just need to apply Theorem~\ref{thm:linking_system} instead of Corollary~\ref{cor:linking_system}.

\paragraph{Linking systems.}  Let $G\sim G(n,p)$. Our starting point is the following observation, which is a simple consequence of Markov's inequality. Let $\mathbf{a}=(a_1,\ldots,a_m)$, $\mathbf{b}=(b_1,\ldots,b_m)$, and $\mathbf{c}=(c_1,\ldots,c_m)$ be pairwise disjoint tuples of different vertices, and suppose that $V_1$, $V_2$, and $\mathbf{b}$ form a partition of $[n]$, with $|V_1|=|V_2|$, $\mathbf{a}\subseteq V_1$, and $\mathbf{c}\subseteq V_2$. Assume that whp there exists a spanning linear forest in $G[V_1\cup \mathbf{b}]$ consisting of $m$ paths such that, for every $i\in[m]$, the $i$th path has length $k$ and endpoints $a_i$ and $b_i$. Similarly, assume that whp there exists a spanning linear forest in $G[V_2\cup \mathbf{b}]$ consisting of $m$ paths such that, for every $i\in[m]$, the $i$th path has length $k$ and endpoints $b_i$ and $c_i$. Then, whp $G(n,p)$ is an $(\mathbf{a},\mathbf{c};2k)$-linking system.
Thus, the problem reduces to embedding a rooted linear forest.

When $\psi$ is large, this can be achieved via a relatively straightforward application of the fractional expectation-threshold theorem of Frankston, Kahn, Narayanan, and Park~\cite{FraKahNarPar2021}. However, for $\psi$ close to $1$, the family of all rooted linear forests fails to satisfy the spread conditions required for the application of this theorem.
To overcome this difficulty, we show that whp the collection of prefixes of the desired rooted linear forests is sufficiently rich to contain a large subfamily whose complements satisfy the necessary spread conditions.
This allows us to apply the expectation-threshold theorem to the family of complements of these prefixes.
 We believe that this new spread-improvement technique may be useful in other settings as well and is therefore of independent interest.

\subsection{Further challenges}
\label{sc:intro-open}

Our results leave several natural questions open. Perhaps the most immediate one concerns the location of a sharp threshold (whose existence is ensured by~\cite{Fri99}) for bounded-degree tree universality. We show that, for large enough $C$ and every fixed $\Delta$, whp $G(n,C\ln n/n)$ is $\mathcal{T}_{n,\Delta}$-universal. 
We~believe that $p=\ln n/n$ is a sharp threshold for this property. We also note that the possibility of this sharp threshold was already suggested by Montgomery~\cite{Mon19}.

\begin{conjecture}
\label{cj:trees}
For every $\varepsilon>0$ and every integer $\Delta\geq 2$, whp $G(n,(1+\varepsilon)\ln n/n)$ is $\mathcal{T}_{n,\Delta}$-universal.
\end{conjecture}
\noindent
We believe that our framework can be applied in the setting of (suitably dense) randomly perturbed graphs to provide a corresponding sharp threshold result but leave this for future work.

A stronger question concerns the hitting time for $\cT_{n,\Delta}$-universality. Let $\tau$ denote the hitting time at which the random graph process becomes $\mathcal{T}_{n,\Delta}$-universal. 
Glebov, Johannsen, and Krivelevich~\cite{GJK13,GJK-unpublished} conjectured that whp $\tau$ coincides with the hitting time for containing a Hamilton path $\tau_{\rm Ham}$ (or equivalently for having at most two vertices of degree at most 1). Even the simpler question whether, for every $T\in\mathcal{T}_{n,\Delta}$, whp $G_{\tau_{\rm Ham}}$ contains an isomorphic copy of $T$, remains open for trees with a small number of leaves and without a bare path of linear length.
Another interesting direction is to investigate the threshold for $\mathcal{T}_{n,\Delta}$-universality for $\Delta=\Delta(n)$. According to Remark~\ref{rk:constant-factors}, our proof gives the answer for $\Delta=o((\ln\ln n)^{1/5})$, while larger values remain open. 

Another direction concerns cycle factors. 
Determining  thresholds for the existence of $C_{\ell}$-factors for {\it all} $\ell=\omega(1)$  remains an open question. 
Recalling the family $\mathcal{C}_{n,\ell}$ of $n$-vertex cycle factors of girth at least $\ell$,
Theorem~\ref{thm:main_cycles} gives the answer for $\ell\geq\frac{K\ln n}{\ln\ln n}$ by exhibiting $\mathcal{C}_{n,\ell}$-universality. 
We believe that actually $\ln n/n$ is a sharp threshold and that the constant factor in the lower bound on $\ell$ can be reduced to 1. 

\begin{conjecture}
\label{cj:cycles}
For every $\varepsilon>0$ and $\ell\geq(1+\varepsilon)\frac{\ln n}{\ln\ln n}$, whp $G(n,(1+\varepsilon)\ln n/n)$ is $\mathcal{C}_{n,\ell}$-universal.
\end{conjecture}

Similarly, as a next step in this direction, we find it particularly interesting to understand the corresponding hitting times. Let $\tau_{\ell}$ denote the hitting time for the appearance of all cycle factors with girth at least $\ell$. Is it true that, for sufficiently large constants $C$ and $\ell=C\ln n/\ln\ln n$, the hitting time $\tau_{\ell}$ coincides with the hitting time for Hamiltonicity whp? Does this remain true when $\ell=(1+\varepsilon)\ln n/\ln\ln n$ for an arbitrary fixed $\varepsilon>0$? Such a result would provide a striking strengthening of the classical hitting-time theorem for Hamilton cycles.

Finally, our work studies linking systems, and the bounds we obtain on their depth are only optimal up to a constant factor. 
It would be very valuable to determine the correct relationship between the edge probability and the attainable depth of linking systems. We believe that establishing the following conjecture might be an intermediate step in resolving Conjecture~\ref{cj:trees} and Conjecture~\ref{cj:cycles}. 

\begin{conjecture}
Fix $\varepsilon>0$, and let $k \geq  (1+\varepsilon)\frac{\ln n}{\ln \ln n}$ and $m \coloneq \frac{n}{k+1}$ be integers.
Fix disjoint $m$-tuples of different vertices $\mathbf{a},\mathbf{b}\in [n]^m$. 
Then, whp the random graph $G\sim G(n,(1+\varepsilon)\ln n/n)$ contains an $(\mathbf{a},\mathbf{b};k)$-linking system.
\end{conjecture}
More generally, obtaining tight bounds on the minimum depth of linking systems in sparse random graphs appears to be a fundamental problem, whose resolution could have further applications to embedding spanning structures in random graphs.

\subsection{Organisation} In Section~\ref{sc:pre}, we introduce some notation and auxiliary results.
In Section~\ref{sec:linking system}, we prove Theorem~\ref{thm:linking_system}.
A detailed overview of the proof of Theorem~\ref{thm:univ_trees} is presented in Section~\ref{sc:overview}. Then, Section~\ref{sec:long bare paths} is dedicated to the proof of a general result which implies $\mathcal{C}$-universality and universality for the family $\mathcal{T}_1$ comprising trees with many long bare paths. Finally, Section~\ref{sec:manyleaves} proves $\mathcal{T}\setminus\cT_1$-universality, thus jointly with the general result in Section~\ref{sec:long bare paths} proving Theorem~\ref{thm:univ_trees}.

\section{Preliminaries}
\label{sc:pre}

\subsection{Notation}
We denote by $\mathbb N$ the set of all positive integers and write $[n]\coloneqq \{1,\ldots, n\}$, which will always serve as the set of vertices of $G(n,p)$. We also write $A\sqcup B$ to denote the union of disjoint sets $A$ and $B$ (as opposed to $A\cup B$ where the disjointness assumption is dropped).
Many of our results are asymptotic; we thus use standard asymptotic notation. By default, hidden constants are absolute and independent of the maximum degree parameter $\Delta$.  We often ignore rounding errors and assume divisibility conditions when needed. 
 Since in such cases full precise arguments proceed verbatim and require only routine modifications, we prefer to omit such details. 
 
For an event $A$, we write $A^c$ for its complement. 
Also, for a random variable $X$ and a distribution $\cL$, we denote by $X\sim \cL$ the fact that $X$ is distributed according to $\cL$.

For a graph $G$, we write $V(G)$ and $E(G)$ for its vertex and edge sets, respectively. 
We say that $|V(G)|$ is the \emph{order} of $G$ while $|E(G)|$ is the \emph{size} of $G$.
For a set $A\subseteq V(G)$, we denote by \[N_G(A)\coloneqq \{v\in V(G)\setminus A: \exists u\in A \text{ such that }\{u,v\}\in E(G)\}\] 
the external neighbourhood of $A$ in $G$. 
When the graph $G$ is clear from the context, we omit the dependency on $G$. In addition, when $A=\{u\}$, we simply write $N(u)$ instead of $N(\{u\})$. 
Furthermore, for disjoint sets $A,B\subseteq V(G)$, we write $G[A]$ for the subgraph of $G$ induced by $A$ and $G[A,B]$ for the bipartite subgraph of $G$ with parts $A,B$ containing all the edges in $G$ between $A$ and $B$.
We often abuse notation and treat vertex tuples as sets: for example, for a vertex tuple $\mathbf{a}$ and a vertex set $U$, we write $\mathbf{a}\subseteq U$.

A path with endpoints $u$ and $v$ is called a \emph{$uv$-path}, and a vertex-disjoint union of paths is called a \emph{path system}.
As usual, the length $k$ of a path is the number of edges it contains; a path of length $k$ is also called a \emph{$k$-path}, and a path system where all paths have length $k$ is called a \emph{$k$-path system}.

\subsection{Probabilistic lemmas}
In this section, we collect several well-known lemmas used throughout the paper. Recall that, for positive integers $m,n\le N$, the hypergeometric distribution with parameters $(m,n,N)$ is the distribution of the size of the intersection of a random $n$-subset of $[N]$ with the set $[m]$. We start with the version of Chernoff's bound from \cite[Theorems~2.1 and~2.10]{JLR00}.

\begin{lemma}[Chernoff's bound \cite{JLR00}]\label{lem:Chernoff}
For any binomial or hypergeometric random variable $X$ and $t\geq 0$,
    \[
    \mathbb{P}(X\geq \mathbb E[X]+t )\leq \exp \left(-\frac{t^2}{2(\mathbb E[X] +t/3)}\right)\quad \text{and}\quad
    \mathbb{P}(X\leq \mathbb E[X]-t )\leq \exp \left(-\frac{t^2}{2\mathbb E[X]}\right).
    \]
\end{lemma}

In the sequel, we will require the asymmetric version of the Lov\'asz Local Lemma (LLL) appearing as \cite[Lemma~5.1.1]{AS2000}. To introduce the lemma, we recall the notion of a (directed) dependency graph for a family of events. 
For a collection of events $\cA =\{A_1,\ldots, A_n\}$ in some probability space, a {\it dependency graph} $D$ of $\cA$ is a (directed) graph on $[n]$ such that, for every $i\in [n]$, the event $A_i$ is mutually independent of all events $A_j$ such that $j\neq i$ and $(i,j)\notin E(D)$.

\begin{lemma}[LLL \cite{AS2000}]\label{lem:LLL}
Fix a family of events $\cA=\{A_1,\ldots ,A_n\}$ with dependency graph $D$. If for some real numbers $x_1,\ldots,x_n\in [0,1)$ and for every $i\in [n]$, we have
    $\mathbb{P}(A_i)\leq x_i\cdot \prod_{(i,j)\in E(D)}(1-x_j)$,
    then
    \[
        \mathbb{P}(A_1^{c}\cap \ldots \cap A_n^c) \geq \prod_{i=1}^{n}(1-x_i).
    \]
\end{lemma}

\subsection{Matchings in graphs}
Throughout the paper, we use several convenient lemmas for finding matchings in well-expanding bipartite graphs. 
 The following is a version of Hall's theorem (implied by the more general Lemma 3.32 in \cite{Mon19}) which requires verifying Hall's condition for small sets only as long as an additional expansion condition is satisfied.

\begin{lemma}\label{lem:match-easy}
Fix a bipartite graph $G$ with equal parts $A$ and $B$ and a positive integer $t$ such that:
    \begin{enumerate}[label={\emph{(\roman*)}}]
        \item For all sets $U$ of size at most $t$ and such that $U\subseteq A$ or $U\subseteq B$, we have $|N(U)| \ge |U|$.
        \item\label{item:matching2} For all sets $U\subseteq A$ and $W\subseteq B$ of size $t$, we have $e(U,W)\geq 1$.
    \end{enumerate}
    Then, $G$ contains a perfect matching from $A$ to $B$.
\end{lemma}

We remark the following simple consequence of Lemma~\ref{lem:match-easy}. Although the statement of this corollary is fairly well-known, we give a short proof for the sake of completeness.
\begin{corollary}\label{cor:Hall}
Fix a bipartite graph $G$ with parts $A,B$ such that $|A|=|B|=n$. Assume that, for every set $U\subseteq A$ or $U\subseteq B$ of size $|U|\le \lceil n/2\rceil$, we have $|N(U)|\ge |U|$. Then, $G$ contains a perfect matching. 
\end{corollary}
\begin{proof}
Set $t=\lceil n/2\rceil$. If $n$ is odd, the corollary follows from Lemma~\ref{lem:match-easy}: indeed, for every set $U$ of size $t$ contained in one of the parts, say $A$, we have $|B\setminus N(U)|\le n-t < t$, thus ensuring assumption \emph{\ref{item:matching2}} in Lemma~\ref{lem:match-easy}.
If $n$ is even, then either Lemma~\ref{lem:match-easy} applies or there are sets $U\subseteq A$ and $W\subseteq B$ of size $t$ with $e(U,W)=0$.
In this case, $N(U)=B\setminus W$ and $N(W)=A\setminus U$, and the existence of a perfect matching in $G$ follows from Hall's theorem applied separately for $G[U,N(U)]$ and for $G[W,N(W)]$.
\end{proof}

We also need a more robust version of Hall's theorem providing a sufficient condition for the existence of a many-to-one matching.

\begin{definition}
\label{def:f-matching}
For disjoint sets $A$ and $B$ and a function $f:A\to \mathbb N$, we call an \emph{$f$-matching} from $A$ to $B$ a set of vertex-disjoint stars $(S_v)_{v\in A}$ where each $S_v$ is centred at $v$ and has exactly $f(v)$ leaves and all of them are in $B$.
\end{definition}
For a set $U\subseteq A$, we further write $f(U)$ for the sum of $f(v)$ over $v\in U$. 
The following theorem is the required extension of Hall's criterion to $f$-matchings, appearing as Theorem~3.31 in~\cite{Mon19} (see also~\cite{Bol98}).

\begin{theorem}\label{thm:3.31}
Fix a bipartite graph with parts $A,B$, and a function $f:A\to \mathbb N$. Suppose that, for every subset $U\subseteq A$, we have $|N(U)| \ge f(U)$. 
Then, there is an $f$-matching from $A$ to $B$.
\end{theorem}

Similarly to Lemma \ref{lem:match-easy}, if the bipartite graph satisfies an additional expansion condition, verifying the Hall-type condition in Theorem \ref{thm:3.31} is sufficient for small sets on both sides (see also \cite[Lemma~3.32]{Mon19}). 
In our context, we will use a straightforward generalisation of this lemma, which uses the following definition.

\begin{definition}\label{def:f-match}
    For a bipartite graph $(A,B)$, integers $1\leq \eta\leq t$, sets $A_1\subseteq A$ with $|A_1|\le \eta$ and $B_1\subseteq B$, and a function $f:A\to \mathbb{N}$ satisfying $f(A) = |B|$, we say that $(A,B)$ is an $(f,A_1,B_1,\eta,t)$-expander if each of the following holds: 
\begin{enumerate}[(i)]
    \item \label{item:extended-matching1}for every $U\subseteq A_1$, we have $|N(U)\cap B_1|\ge f(U)$,
    \item \label{item:extended-matching2} for every $U\subseteq A\setminus A_1$ with $|U|\le \eta$, we have $|N(U)\setminus B_1|\ge f(U)$,
    \item \label{item:extended-matching3} for every $U\subseteq A$ with $|U|\in [\eta+1,t]$, we have $|N(U)|\ge f(U\cup A_1)$,
    \item \label{item:extended-matching4} for every $U\subseteq B$ with $|U|\le t$, we have $|N(U)\setminus A_1| \ge |U|$,
    \item \label{item:extended-matching5} for all $A'\subseteq A$ and $B'\subseteq B$ with $|A'|=|B'|=t$, there is an edge between $A'$ and $B'$.
\end{enumerate}
We call an $f$-matching from $A$ to $B$ an \emph{$(f,A_1,B_1)$-matching}, if the neighbourhood of $A_1$ in the $f$-matching lies entirely inside $B_1$.
\end{definition}

The next lemma is the promised generalisation of the Hall-type matching Lemma~3.32 in~\cite{Mon19}, with the original result concerning the case $A_1=B_1=\varnothing$.

\begin{lemma}\label{lem:match}
If a bipartite graph $(A,B)$ is an $(f,A_1,B_1,\eta,t)$-expander, then
there is an $(f,A_1,B_1)$-matching from $A$ to $B$.
\end{lemma}
\begin{proof}
First of all, by property~\ref{item:extended-matching1} and Theorem~\ref{thm:3.31}, an $f$-matching from $A_1$ to $B_1$ exists.
Fix one such $f$-matching and denote by $B_2$ the set of all vertices in $B_1$ occupied by the matching.
To complete the constructed star forest to an $(f,A_1,B_1)$-matching from $A$ to $B$ we show that, for every set $U\subseteq A\setminus A_1$, 
\begin{equation}\label{eq:Hall}
|N(U)\setminus B_2|\ge f(U),  
\end{equation}
which suffices to conclude by Theorem~\ref{thm:3.31} applied to $(A\setminus A_1,B\setminus B_2)$. 

First,~\eqref{eq:Hall} holds for all sets $U\subseteq A\setminus A_1$ of size $|U|\le \eta$ by property~\ref{item:extended-matching2}.
Second, by property~\ref{item:extended-matching3}, if $|U|\in [\eta+1,t]$, then
\[|N(U)\setminus B_2| = |N(U)|-|N(U)\cap B_2| \ge f(U\cup A_1)-|B_2|=f(U)+f(A_1)-|B_2| = f(U),\]
where we used that $U\cap A_1=\varnothing$ and $f(A_1)=|B_2|$.
Third, for every set $U\subseteq A\setminus A_1$ of size $|U| > t$ and such that $f(U)\le f(A\setminus A_1)-t$, by property~\ref{item:extended-matching5}, we have
\[|N(U)\setminus B_2|\ge |B\setminus B_2|-t = f(A\setminus A_1)-t\ge f(U).\]
Finally, assume that~\eqref{eq:Hall} fails for some set $U\subseteq A\setminus A_1$ where $|U| > t$ and $f(U) > f(A\setminus A_1)-t$.
Note that the set $B\setminus (B_2\cup N(U))$ has size at most $t$ by property~\ref{item:extended-matching5} and its neighbourhood in $A\setminus A_1$ is disjoint from $U$. By combining the latter observations with property~\ref{item:extended-matching4} applied to $B\setminus (B_2\cup N(U))$, we have
\begin{align*}
|A\setminus A_1| &= |U|+|A\setminus (A_1\cup U)|
\ge |U|+|B\setminus (B_2\cup N(U))|\\ 
&= |U| + |B\setminus B_2|-|N(U)\setminus B_2| > |U|+f(A\setminus A_1)-f(U)\ge |A\setminus A_1|, 
\end{align*}
where the last inequality uses that $\min f\ge 1$ and thus 
\[f(A\setminus A_1) - f(U) = f(A\setminus (A_1\cup U)) \ge |A\setminus (A_1\cup U)| = |A\setminus A_1|-|U|.\]
This shows a contradiction, confirms~\eqref{eq:Hall}, and ensures the promised $(f,A_1,B_1)$-matching from $A$ to $B$. 
\end{proof}

\subsection{\texorpdfstring{Universality for trees on $(1-\eps)n$ vertices}{Universality for trees on (1-eps)n vertices}}

Next, we formally state the result of Alon, Krivelevich and Sudakov~\cite{AKS07} mentioned in the introduction.

\begin{theorem}[\cite{AKS07}]\label{thm:AKS07}
For every $\eps > 0$ and every integer $\Delta\ge 2$, there exists $C=C(\eps,\Delta)>0$ such that, with high probability, the random graph $G(n,C/n)$ contains copies of all trees on at most $(1-\eps)n$ vertices and maximum degree at most $\Delta$.
\end{theorem}

\subsection{Rollback}

A key embedding tool used in the paper is the rollback technique. 
 The exact version of the technique we require appears as \cite[Lemma~4.2]{Mon19} stated below. Before introducing it, we need a few definitions. 

\begin{definition}
\label{def:joined}
A graph $H$ is said to be \emph{$t$-joined} if there is an edge in $H$ between every two disjoint subsets of size $t$.
\end{definition}

\begin{definition}
Fix integers $t\geq 1$ and $d\geq 3$, and graphs $S\subseteq G$. We say that $S$ is \emph{$(d,t)$-extendable} (in $G$) if $S$ has maximum degree at most $d$ and, for all sets $X\subseteq V(G)$ with $|X|\leq 2t$,
\begin{equation}\label{eq:X}
\left|\bigcup_{x\in X}N_G(x)\setminus V(S)\right|\geq (d-1)|X|-\sum_{x\in X\cap V(S)}(|N_S(x)|-1).
\end{equation}
\end{definition}

Note that, in our proofs, we often show the stronger lower bound 
\[|N_G(X)\setminus V(S)|\geq d|X|,\]
which clearly implies~\eqref{eq:X}.
For a vertex set $X$, we denote by $I(X)$ the graph with vertex set $X$ and no edges.

\begin{definition}
\label{def:separated}
For an integer $k\ge 1$, we say that $X$ is a \emph{$k$-separated} set in $G$ if every pair of vertices in $X$ are at distance at least $k$ in $G$.
\end{definition}

The following proposition appears as \cite[Lemma~4.2]{Mon19}.

\begin{proposition}\label{embedparentsfinal}
Fix an integer $\Delta\ge 2$, a sufficiently large integer $n$, and integers $d\geq\ln n/\ln\ln n$ and $t\in[n]$. 
In addition, fix a $t$-joined graph $H$ containing a vertex set $X$ and a vertex $v\in V(H)\setminus X$ so that $|X|\geq n/(\ln n)^{2}$ and $I(X\cup\{v\})$ is a $(d,t)$-extendable graph in $H$.

Further, fix a $w$-rooted tree $T$ with order $|V(T)|\leq |V(H)|-|X|-10dt-\ln n$ and maximum degree $\Delta(T)\leq\Delta$. 
Also, fix a $16$-separated vertex set $U\subseteq V(T)$ with $|U|\geq 9|X|$. Then, $H$ contains a $v$-rooted copy $W$ of $T$ such that $W$ is $(d,t)$-extendable in $H$ and $X$ is contained in the copy of $U$.
\end{proposition}

\begin{remark}\label{rem:embed forest}
Note that Theorem~\ref {thm:AKS07} is also valid for forests instead of trees with the same assumptions on their order and maximum degree: indeed, every such forest can be merged into a single tree by iteratively connecting leaves of distinct components.

In a similar fashion, for Proposition~\ref{embedparentsfinal} in order to be valid for forests, we may join
the $c$ components of a forest by paths of length $16$,
adding $15(c-1)$ vertices. This preserves the maximum
degree bound and ensures that a set which is
$16$-separated within each component is
$16$-separated in the resulting tree.
Applying the proposition and discarding the added
vertices gives the required rooted covering embedding.
\end{remark}

\subsection{Spread measures}
Given a set $X$, we denote by $2^{X}$ the family of all subsets of $X$. 
Fix a set $X$ of size $N$ and a non-empty family $\cF\subseteq 2^{X}$. 
 For a set $S\subseteq X$, its \emph{up-set} is defined as $\langle S\rangle\coloneqq \{A\cup S:A\subseteq X\}$; we also write $\langle \cG\rangle \coloneqq \bigcup_{S\in \cG}\langle S\rangle$ for a family $\cG\subseteq 2^{X}$. The following is a key notion in the study of thresholds.

\begin{definition}[$q$-spread measures]
For $q\in (0,1)$, a probability measure $\nu$ on $(\cF,2^{\cF})$ is {\it $q$-spread} if, for every $S \subseteq X$,
\[\nu(\cF\cap \langle S\rangle) \leq q^{|S|}.\]
\end{definition}

A key component in Section \ref{sec:linking system} is the fractional expectation-threshold theorem due to Frankston, Kahn, Narayanan, and Park \cite{FraKahNarPar2021}. 
Denote by $X_p$ a random subset of $X$ where each element is retained with probability $p$ independently of the others. 
The following quantitative version of the expectation-threshold theorem is directly implied by~\cite[Theorem~3]{Bel22} (with $\eps=1/\ell$).

\begin{theorem}\label{thm:fracKK}
Denote by $\ell=\ell(\cF)$ the maximal size of a minimal element of a non-empty family $\cF$ with $\varnothing\not \in \cF$ which admits a $q$-spread measure, and set $p > 200 q\ln\ell$. 
Then, with probability at least $1-1/\ell$, we have $X_p\in\langle \cF \rangle$.
\end{theorem}

Theorem~\ref{thm:fracKK}, together with the expectation-threshold theorem of Park and Pham~\cite{PP24}, constitutes an important tool in probabilistic combinatorics. Both results, and the fragmentation technique developed in their proofs, have proved useful for establishing threshold functions for a broad range of properties of discrete structures (see, e.g.,~\cite{DKP,DiazPerson,Keevash,KNP,Pham,VZ,Z}). 
For more details, we refer the interested reader to the survey~\cite{Per2025}. In our applications of Theorem~\ref{thm:fracKK}, $X$ is the set of edges of $K_n$; in particular, $X_p$ is exactly the random graph $G(n,p)$.

\section{Linking systems via path systems}\label{sec:linking system}

Our goal in this section is to
prove Theorem~\ref{thm:linking_system}. In other words, we 
show the existence of linking systems in $G(n,p)$ 
whose depth matches up to a constant the lower bound coming from a direct first moment argument.

As mentioned in Section~\ref{sc:intro-proofs}, we reduce the problem to the existence of a path system defined below.
\begin{definition}
For two disjoint vertex $m$-tuples $\mathbf{a}=(a_i)_{i=1}^{m}, \mathbf{b}=(b_i)_{i=1}^{m}$ and an integer $k$, we call a $k$-path system $F$ an \emph{$(\mathbf{a},\mathbf{b};k)$-path system} if $F$ has $m$ connected components, and each connected component of $F$ is a path of length $k$ with endpoints $a_i,b_i$, for some $i\in[m]$.
\end{definition}

The rest of the proof is organised as follows. In Section~\ref{sc:linking-reduction}, we explain the reduction from the existence of spanning linking systems to the existence of spanning path systems. Section~\ref{sc:proof-path-systems} establishes that the random graph has the desired spanning path system whp, thus completing the proof of Theorem~\ref{thm:linking_system}.

\subsection{Proof of Theorem~\ref{thm:linking_system}: reduction to path systems}\label{sc:linking-reduction}

The following easy lemma reduces the problem to showing the existence of certain path systems whp.

\begin{lemma}\label{lem:combaining path systems}
Consider integers $n$, $k=k(n)\geq 3$, and $m=m(n)$ such that $n=(k+1)m$. Denote
\[
    n_1 \coloneqq \left\lceil \frac{k+2}{2}\right\rceil \cdot m,
\]
and define sets $S_1 := [n_1]$, $S_2:=[n]\setminus [n_1-m]$.
Fix disjoint $m$-tuples of vertices $\mathbf{a}\subseteq S_1\setminus S_2$, $\mathbf{b}\subseteq S_2\setminus S_1$ and $\mathbf{c}=S_1\cap S_2$. 
 Let $G$ be an exchangeable\footnote{That is, for any permutation $\sigma\in S_n$ acting on $[n]$, the image of $G$ under $\sigma$ and $G$ itself are identically distributed} random graph with vertex set $[n]$ such that:
\begin{enumerate}[label=\normalfont{(S\arabic*)}]
    \item\label{item:S1} whp $G[S_1]$ contains an $\left(\mathbf{a},\mathbf{c};\big\lceil \frac{k+2}{2}\big\rceil-1\right) $-path system;
	\item\label{item:S2} whp $G[S_2]$ contains a $\left(\mathbf{c},\mathbf{b};\big\lceil \frac{k+1}{2}\big\rceil-1\right)$-path system.
\end{enumerate}
Then, whp $G$ is an $(\mathbf{a},\mathbf{b};k)$-linking system.
\end{lemma}

\begin{proof}
For each $\mathbf{x}\in \{\mathbf{a},\mathbf{b},\mathbf{c}\}$ and a permutation $\sigma\in S_m$ acting on $[m]$,
let $\sigma(\mathbf{x}) = (x_{\sigma(1)},\ldots,x_{\sigma(m)})$.
Define $\Sigma_1$ to be the (random) family of permutations $\sigma\in S_m$ such that the graph $G[S_1]$ contains an 
$(\mathbf{a},\sigma(\mathbf{c});\left\lceil (k+2)/{2}\right\rceil-1)$-path system. Similarly, $\Sigma_2$ consists of $\sigma\in S_m$ such that $G[S_2]$ contains a $(\mathbf{c},\sigma(\mathbf{b});\left\lceil ({k+1})/{2}\right\rceil-1)$-path system.

Since $G$ is exchangeable, the event described in~\ref{item:S1} holds with the same probability $1-o(1)$ for any permutation $\sigma\in S_m$ and $\sigma(\mathbf{c})$ in place of $\mathbf{c}$, and the same is true for~\ref{item:S2}.
Thus, for both $i\in \{1,2\}$, Markov's inequality applied for $|S_m\setminus \Sigma_i|$ implies that whp $|\Sigma_i|=(1-o(1))m!$. 

To conclude, fix any $\sigma\in S_m$. Then, whp the set $\Sigma_2$ contains a permutation $\sigma\circ \pi^{-1}$ for some $\pi\in\Sigma_1$. 
Now, extending an $\left(\mathbf{a},\pi(\mathbf{c});\left\lceil \frac{k+2}{2}\right\rceil-1\right) $-path system by a $\left(\mathbf{c},(\sigma\circ \pi^{-1})(\mathbf{b});\left\lceil \frac{k+1}{2}\right\rceil-1\right) $-path system gives an $\left(\mathbf{a},\sigma(\mathbf{b});k\right) $-path system. As this holds for every $\sigma\in S_m$, this ensures the desired linking system.
\end{proof}

Given Lemma \ref{lem:combaining path systems}, from now on, we will be only interested in path systems.
Throughout this section, we will make use of the following set of constants:
\begin{equation}
    C_0= 10^5,\ 
    C_1 = 8,\ C_2=2\cdot 10^4,\
    \eps = 10^{-100} ,\ D=\ln C_1 +2\eps,\  C^* =2(1+\eps).
\label{eq:constants}
\end{equation}
In addition, to state the results of this section, fix 
a function $\psi=\psi(n) \in \big[1, \frac{\ln n}{4\ln\ln n}\big]$ and set
\[
    k^*=k^*(n,\psi)\coloneqq \lceil C^*\ln n/(\psi\ln \ln n)\rceil.
\] 
Lastly, fix an integer $k\geq k^*$. For ease of presentation, we further assume that $m\coloneqq n/(k+1)$ is an integer, or equivalently that $k+1$ divides $n$. 

\begin{theorem}\label{prop:path-systems-new}
Consider $p=C_0(\ln n)^{\psi}/n$ and disjoint vertex $m$-tuples $\mathbf{a},\mathbf{c}\subseteq [n]$. Then,
\[
    \bP\left(\text{$G(n,p)$ contains an $(\mathbf{a},\mathbf{c};k)$-path system}\right) = 1-o(1/(\ln\ln n)^2).
\]
\end{theorem}

We note that a path system guaranteed by the theorem is spanning since $m(k+1)=n$. 
In fact, Theorem~\ref{thm:linking_system} follows rather directly by combining Lemma~\ref{lem:combaining path systems} and Theorem~\ref{prop:path-systems-new} when the function $\psi$ is `smooth' enough (e.g.\ slowly varying). The general case involves several technicalities. Moreover, to apply Theorem~\ref{thm:linking_system} directly in the sequel, we are required to control the ratio $\psi(\delta n)/\psi(n)$ for some constant $\delta>0$. Since $\psi$ is arbitrary, this is not possible in general. We therefore require the following variant of Theorem~\ref{thm:linking_system}, that applies to $G(\delta n,p(n))$. We outsource its proof to Appendix~\ref{sec:proofthm2} and note that Theorem~\ref{thm:linking_system} follows\footnote{Formally, it follows directly only for even depths. The case of odd depths can be derived from the same argument with the obvious modification of using two path systems of lengths $\lceil k/2\rceil$ and $\lfloor k/2\rfloor$.} by taking $\delta=1$.

\begin{restatable}{lemma}{Lemmathm}
\label{cor:linking-systems}
Fix $\delta\in (0,1]$, set
$$
p=p(n)\coloneqq \frac{2C_0}{\delta}\frac{(\ln n)^{\psi(n)}}{n},
$$
and consider integers $k\geq 4\e k^*(n,\psi)$ and $m=\delta n/(2k+1)$.
Fix disjoint $m$-tuples with different elements $\mathbf{a},\mathbf{b}\subseteq [\delta n]$. 
Then, whp $G(\delta n,p)$ contains an $(\mathbf{a},\mathbf{b};2k)$-linking system.
\end{restatable}

\begin{remark}\label{remark:uniform bound}
    A technical point we will need later is that the conclusion of Lemma \ref{cor:linking-systems} holds uniformly for $\delta$ in any fixed compact subset of $(0,1]$. 
\end{remark}

To conclude this subsection, we use the first moment method to show that Theorem~\ref{prop:path-systems-new} is optimal and derive an analogous corollary for linking systems. 
The next lemma together with Theorem \ref{prop:path-systems-new} imply the order of magnitude of the threshold for the existence of a spanning $(\mathbf{a},\mathbf{c};k)$-path system for unboundedly many values of $k$; in particular,  for $k=\lceil3\ln n/\ln\ln n\rceil$ subject to $(k+1)\mid n$, the threshold is $\Theta(\ln n/n)$.\footnote{The lower bound here follows from the existence of an isolated vertex in $G(n,(1-\varepsilon)\ln n/n)$ whp.}

\begin{lemma}\label{lem:optimal-depth}
Consider $p=(\ln n)^{\psi}/2n$, an integer $k_* = \ln n/(\psi \ln\ln n)$ and an integer $m_*\coloneqq n/(k_*+1)$.
Let $\mathbf{a},\mathbf{c}$ be disjoint $m_*$-tuples of different vertices from $[n]$. Then, whp $G(n,p)$ contains no~$(\mathbf{a},\mathbf{c};k_*)$-path system.
\end{lemma}

\begin{proof}
    Let $X$ be the number of $(\mathbf{a},\mathbf{c};k_*)$-path systems in $G(n,p)$. We have
    \begin{align*}
        \bE[X] &= (n-2m_*)!\cdot p^{n-m_*} \leq 2^{-n/2} (\ln n)^{\psi(n-m_*)}/n^{m_*} = 2^{-n/2} \exp \left((n-m_*)\psi \ln\ln n - m_*\ln n\right).
    \end{align*}
    By Markov's inequality $\bP(X\geq 1)\le \bE[X]$, which concludes the proof upon using the above bound and noting that
    \[
        (n-m_*)\psi \ln\ln n - m_*\ln n = 0.\qedhere
    \]
\end{proof}

It remains to prove Theorem~\ref{prop:path-systems-new}. 

\subsection{Proof of Theorem \ref{prop:path-systems-new}.}
\label{sc:proof-path-systems}
The first natural way to attack Theorem \ref{prop:path-systems-new} is to apply Theorem~\ref{thm:fracKK} with $\cF$ being the family of all $(\mathbf{a},\mathbf{c};k)$-path systems. 
This approach does not work in the crucial case $\psi =1$ and $k=\Theta(\ln n/\ln\ln n)$ used in our proof of Theorem \ref{thm:univ_trees}.
In this case, the expectation threshold is $(\ln n)^{\Theta(1)}/n$ and so the direct application of Theorem~\ref{thm:fracKK} gives a suboptimal result. Moreover, the fragmentation technique fails as well since the above family does not satisfy the required spread bounds: fragments consisting of full $a_ic_i$-paths belong to an unsuitably large number of $(\mathbf{a},\mathbf{c};k)$-path systems, so that in order to improve the size of a typical fragment by a constant factor, one needs to sprinkle with probability $(\ln n)^{\Omega(1)}/n$.

To overcome this issue, we prepare the ground by initially revealing $G_1\sim G(n,p_1)$ with $p_1=C_1(\ln n)^\psi/n$. Then, for suitably chosen $d=\Theta(k^*)=O(k)$ --- 
defined in~\eqref{eq:d} --- we find a large family $\cP$ of $d$-path systems in $G_1$ which start at $\mathbf{a}$, remain disjoint from $\mathbf{c}$ and are typically well-spread in the sense that every set of paths is contained in not too many elements of $\cP$. Paths in this family play the role of possible prefixes of the desired full path system.
After conditioning on $\cP$ with such properties, we define a family $\cF$ of $(k-d)$-path systems ending at $\mathbf{c}$ such that, for every $F\in \cF$, the graph $P\cup F$ is an $(\mathbf{a},\mathbf{c};k)$-path system for at least one $P\in \cP$, and $F,P$ are edge-disjoint.
This allows us to overcome the issue mentioned in the previous paragraph: indeed, the paths in the family $\cF$ are now rooted only at $\mathbf{c}$ while their other endpoints remain relatively flexible.
 Of course, this flexibility might turn out to be insufficient to ensure that the family of path systems $\cF$ has comparable spread properties to the family $\cF^{all}$ of all $(k-d)$-path systems starting at $\mathbf{c}$.
The randomness inherent in $G_1$ --- and, hence, in $\cF$ --- is sufficient to guarantee that the `counting measure' on $\cF$ exhibits the required spread properties, which are only marginally weaker than those enjoyed by the uniform measure on $\cF^{all}$.

As we mentioned in the outline above, we need to distinguish between small and large values of~$\psi$. Therefore, in Section~\ref{sec:small}, we prove Theorem \ref{prop:path-systems-new} under the assumption that $\psi \leq 2$. 
Then, in Section~\ref{sec:large}, we give the proof of the theorem under the assumption that $\psi > 2$, which is considerably simpler. We note the choice of the threshold value 2 is somewhat arbitrary. Nevertheless, very large $\psi$ must be treated separately, since for technical reasons the argument from Section~\ref{sec:small} does not extend verbatim to $\psi=(\ln n)^{1-o(1)}$.

\subsubsection{Small values of $\psi$}\label{sec:small}

Recall constants defined in~\eqref{eq:constants}. In particular, here we will make use of constants $D,C_1,C_2.$
In this subsection, we assume that $\psi\leq 2$ and set $d=d(n)\in \mathbb N$ satisfying
\begin{equation}\label{eq:d}
(\ln n)^{d\psi} \in \bigg[\frac{n}{(\ln n)^{\psi}}\exp\left(-\frac{D\ln n}{\psi\ln\ln n}\right), n\exp\left(-\frac{D\ln n}{\psi\ln\ln n}\right)\bigg) \quad \text{so that}\quad d=(1-o(1))\frac{\ln n}{\psi\ln\ln n}.
\end{equation}

The following algorithm aims to construct a large family of $d$-path systems with good spreadness properties and such that the paths in each system start from $\mathbf{a}$.

\begin{algorithm}\label{alg:creating the initial segments}
Fix a graph $G$ with vertex set $[n]$ and disjoint $m$-tuples of vertices $\mathbf{a}=(a_1,\ldots ,a_{m})$ and $\mathbf{c}$.
 In addition, fix a list $L$ of $dm$ (not necessarily different) integers from $[\lfloor (
\ln n)^{\psi}\rfloor]$.
Initially, at step $i=1$, set $u_1:=a_1$, $\mathcal{P}_1\coloneqq\varnothing$, and $U_1\coloneqq\{v: v\in \mathbf{a}\cup \mathbf{c}\}$. At every step $i$, we refer to $u_i,\mathcal{P}_i,U_i$ as the current vertex, the path-system, and the set of unavailable vertices, respectively.
  At every step $i\geq 1$,
    \begin{enumerate}[label={(\bfseries A\arabic*)}]
        \item\label{item:algo1} If $d+1$ does not divide $i$, continue to \ref{item:algo2}. Else, if $d+1$ divides $i$ and $i/(d+1)<m$, set $u_{i+1}:= a_{i/(d+1)+1}$, $\cP_{i+1}=\cP_i$, and $U_{i+1}=U_{i}$ and repeat \ref{item:algo1} with $i\to i+1$. 
        Else, terminate the process and output $\mathcal{P}:=\mathcal{P}_i$.
        \item\label{item:algo2} Look for the $L[i-\lfloor i/(d+1)\rfloor]$-th smallest neighbour $v$ of $u_i$ in $[n]\setminus U_i$. If $v$ is not defined, terminate the process and~reject. Otherwise, set $u_{i+1}:=v$, $\mathcal{P}_{i+1}:=\mathcal{P}_i\cup \{\{u_i,u_{i+1}\}\}$, $U_{i+1}:= U_i\cup \{u_{i+1}\}$, and go to \ref{item:algo1} with $i\to i+1$.
    \end{enumerate}
\end{algorithm}

Upon successful termination of the algorithm, denote its output by $\mathcal{P}(G,L)$: this is a $d$-path system consisting of $m$ paths starting from $a_1,\ldots,a_m$ and vertex-disjoint from $\textbf{c}$.
Further note that, for two distinct lists $L,L'$ where each of $\mathcal{P}(G,L)$ and $\mathcal{P}(G,L')$ is well defined, we have $\mathcal{P}(G,L)\neq \mathcal{P}(G,L')$: indeed, by considering the smallest $j$ such that $L[j]\neq L'[j]$, the algorithm constructs identical path systems before reaching $L[j], L'[j]$ but extends these path systems differently upon processing $L[j], L'[j]$.

Let $G_1\sim G(n,p_1)$ with $p_1={C_1(\ln n)^{\psi}}/{n}$. Fix disjoint $m$-tuples $\textbf{a},\textbf{c}$ in $[n]$.
The next lemma shows that whp Algorithm \ref{alg:creating the initial segments} terminates successfully for almost all lists $L$. 
The collection of path systems produced by successful runs of the algorithm for different lists will constitute the family $\cP$ described in the proof outline.

\begin{lemma}\label{lem:algo not rejecting}
With probability $1-o(n^{-1/18})$, there are $(1-o(1))\lfloor (
\ln n)^{\psi}\rfloor^{dm}$ lists $L\in [\lfloor (
\ln n)^{\psi}\rfloor]^{dm}$ such that Algorithm~\ref{alg:creating the initial segments} terminates successfully on the input $(G_1,\textbf{a},\textbf{c},L)$. 
\end{lemma}

\begin{proof}
We show that there is a sequence $\eps_n = o(n^{-1/9})$ such that, for every list $L$, with probability at most $\eps_n$, Algorithm \ref{alg:creating the initial segments} rejects the input $(G_1,\textbf{a},\textbf{c},L)$.
To see that this suffices to conclude, denote by $X$ the number of lists rejected by Algorithm \ref{alg:creating the initial segments} and note that Markov's inequality implies
    \[
        \bP\left(X\geq  \sqrt{\eps_n}\cdot  \lfloor (
\ln n)^{\psi}\rfloor^{d m}\right) \leq \frac{\bE[X]}{\sqrt{\eps_n} \cdot \lfloor (
\ln n)^{\psi}\rfloor^{d m}}\leq \sqrt{\eps_n} = o(n^{-1/18}). 
    \]

    Fix a list $L$. Algorithm \ref{alg:creating the initial segments} rejects the input only if at some iteration the current vertex $u_i$ has fewer than $(\ln n)^{\psi}$ neighbours in $[n]\setminus U_i$. 
    To show that this is unlikely, we analyse Algorithm \ref{alg:creating the initial segments} by revealing the edges between the current vertex $u_i$ and the set $[n]\setminus U_i$ in each iteration: note that these edges were not revealed at any previous iteration since all previously processed vertices are in $U_i$.
    Further, using that
    \[
        \left|[n]\setminus U_i\right |\geq n-(d+2)m\geq (1+\eps/10)n/2,
    \]
    the probability of rejection at a fixed iteration is at most $\bP(\mathrm{Bin}(
    (1+\eps/10)n/2, p_1) \leq (\ln n)^{\psi}).$
    As $C_1/2=4$, by Chernoff's bound and a union bound, the probability of rejection at some iteration is at most
      \begin{align*}
        n\cdot \bP(\mathrm{Bin}((1+\eps/10)n/2, p_1) \leq (\ln n)^{\psi}) & = n\cdot \bP(\mathrm{Bin}((1+\eps/10)n/2, p_1)\leq np_1/8) \\
        & \leq n\cdot o\big(\e^{-9(\ln n)^{\psi}/8}\big)  = o(n^{-1/9}).
    \end{align*}
    This completes the proof.
\end{proof}

In order to prove that the constructed family of path systems has sufficient spread properties, we show an upper bound on the number of $d$-paths between any pair of vertices in $G_1\sim G(n,p_1)$. 

\begin{lemma}\label{lem:few paths}
For vertices $u\neq v$ in $[n]$, denote by $P_{u,v}$ the number of $d$-paths between $u$ and $v$ in $G_1$. Then, with probability $1-o(n^{-1/9})$,
    \begin{equation}\label{eq:multiplicity bound}
        \max_{u\neq v} P_{u,v} \leq \exp((\psi\ln\ln n)^5)\eqqcolon \rho.
    \end{equation}
\end{lemma}

\begin{proof}
We show that, for fixed $u\neq v$, with probability $1-o(n^{-2.5})$, $P_{u,v} \leq \rho$; the statement then follows from the union bound.
Fix $u\neq v$, set 
\begin{equation}
t\coloneqq \left\lfloor (3/\eps)\psi\ln\ln n\right \rfloor,
\label{eq:t-def}
\end{equation}
and define $\cB$ to be the collection of all $t$-tuples of paths $(P_1,\ldots,P_t)$ such that 
\begin{itemize}
\item $P_1$ is a $d$-path between $u$ and $v$,
\item for every $i\in\{2,\ldots,t\}$, $P_i$ has between $1$ and $d$ edges, it is edge-disjoint with $P_1\cup \ldots\cup P_{i-1}$ but has its terminal vertices in the latter graph.
\end{itemize}
The following two claims complete the proof of the lemma.
    
    \begin{claim}\label{claim:typically no bad tuple}
        With probability $1-o(n^{-2.5})$,  there is no $(P_1,\ldots ,P_t)\in \cB$ such that $P_1\cup\ldots\cup P_t\subseteq G_1$.
    \end{claim}

    \begin{claim}\label{claim:no bad tuple few paths}
        Suppose that $Q_1,\ldots ,Q_{\rho}\subseteq E(K_n)$ are distinct $d$-paths between $u$ and $v$. Then, there is a sequence $(P_1,\ldots ,P_t)\in \cB$ such that $P_1\cup\ldots\cup P_t\subseteq Q_1\cup\ldots\cup Q_\rho$. 
    \end{claim}

The statement of the lemma is now immediate. \end{proof}

    \begin{proof}[Proof of Claim \ref{claim:typically no bad tuple}]
        For every sequence $\mathbf{d} = (d_1,\ldots,d_t)\in [d]^{t}$, we have
        \begin{equation}\label{eq:the existence of a sequence of paths}
            \bP\left(\exists(P_1,\ldots,P_t)\in \cB \text{ such that }\forall i\in [t],\; |P_i|=d_i \text{ and }\bigcup_{i=1}^{t}P_i\subseteq G_1\right) \leq \prod_{i=1}^{t}(dt)^2 n^{d_i-1} p_1^{d_i},
        \end{equation}
        where $(dt)^2$ dominates the number of choices of endpoints of $P_i$ given $P_1,\ldots,P_{i-1}$, and $n^{d_i-1}$ bounds from above the number of choices of internal vertices of $P_i$. 
        Then, using~\eqref{eq:the existence of a sequence of paths} and summing up over all possible choices of $\textbf{d}$, we get
        \begin{equation}\label{eq:boundsP_i}
        \bP\left(\exists(P_1,\ldots ,P_{t})\in \cB \text{ such that }\bigcup_{i=1}^{t}P_i\subseteq G_1\right) \leq \sum_{\mathbf{d}\in [d]^{t}} \prod_{i=1}^{t}\frac{(dt)^2 (n p_1)^{d_i}}{n}\le d^t\cdot \bigg(\frac{(dt)^2 (np_1)^d}{n}\bigg)^t.
        \end{equation}
        Since $d(dt)^2 (np_1)^d \le d^5 C_1^d (\ln n)^{d\psi}$, recalling~\eqref{eq:constants},~\eqref{eq:d},~and~\eqref{eq:t-def}, we obtain that~\eqref{eq:boundsP_i} is at most
        \[\left(\frac{d^{5}C_1^{d}(\ln n)^{d\psi}}{n}\right)^{t}\leq  \left(d^{5} C_1^{d}\exp\left(-\frac{D\ln n}{\psi\ln\ln n}\right)\right)^{t}\leq \exp\left(-(D-\ln C_1-\eps)\cdot t\cdot \frac{\ln n}{\psi\ln\ln n}\right)= o(n^{-2.5}),\]
        where the last equality used that $D-\ln C_1-\eps =\eps$. 
        \end{proof}

    \begin{proof}[Proof of Claim \ref{claim:no bad tuple few paths}]
    Fix distinct $d$-paths $Q_1,\ldots ,Q_{\rho}\subseteq E(K_n)$ as described. Call a vertex in $Q\coloneqq \bigcup_{i=1}^{\rho} Q_i$ a \emph{pivot} if it has degree at least $3$ or belongs to $\{u,v\}$.
    For each $i\in[\rho]$, let $Q^{(i)}\coloneqq\bigcup_{j=1}^{i}Q_j$, and call a vertex an \emph{$i$-pivot} if it has degree at least $3$ in $Q^{(i)}$ or belongs to $\{u,v\}$.
    A path in $Q$ is called \emph{bare} if its internal vertices are all non-pivots.
    Similarly, a path in $Q^{(i)}$ is called \emph{$i$-bare} if its internal vertices are all non-$i$-pivots.
    Set $\hat{Q}_1\coloneqq Q_1$ and, for all $i\in [\rho-1]$, let $\hat{Q}_{i+1}$ be the union of the non-trivial connected components of $Q_{i+1}\setminus \bigcup_{j=1}^{i} E(Q_j)$; note that $\hat{Q}_1,\ldots \hat{Q}_{\rho}$ are path systems. For each $i\in [\rho]$, let $P^{i}_1,\ldots,P^{i}_{s_i}$ be the maximal $i$-bare paths in $\hat{Q}_i$.

    Next, we list the paths $P^{i}_j$ in the lexicographic order, namely
    \[
        (P_1,\ldots,P_{\ell}):=(P^{1}_1,\ldots,P^{1}_{s_1},P^{2}_{1},\ldots,P^2_{s_2},\ldots\ldots,P^{\rho-1}_1,\ldots,P^{\rho-1}_{s_{\rho-1}},P^{\rho}_{1},\ldots ,P^{\rho}_{s_{\rho}}),
    \]
    where $\ell = s_1+\ldots+s_{\rho}$.
    Note that $P_1=Q_1$. Moreover, for every $i\geq 2$, the endpoints of each component of $\hat Q_i$ belong to $Q^{(i-1)}$. Every vertex of $\hat Q_i$ belonging to $Q^{(i-1)}$ is an $i$-pivot, whereas every new vertex has degree $2$ in $Q^{(i)}$.
    Note that the intersection of the path $P_i$, $i\geq 2$, with the union of the previous paths is equal to the endpoints of $P_i$, and thus it only remains to show that $\ell\ge t$.
    To this end, note that, since every pivot is an endpoint of $P_i$ for some $i$ (indeed, for a pivot other than $u,v$, consider the first index at which its degree becomes at least $3$), and each $P_i$ contains at least $2$ pivots, there are at most $2\ell$ pivots in $Q$
    and its maximum degree is at most $2\ell$.
    As a result, $Q$ contains at most $(2\ell)^{4}$ maximal bare paths: indeed, there are at most $(2\ell)^{2}$ choices for pairs of pivots and at most $(2\ell)^{2}$ for their neighbours of degree $2$.
    As each path between pivots can be uniquely decomposed into a collection of maximal bare paths, the number of paths between $u$ and $v$ in $Q$ is at most $2^{16\ell^{4}}$. Thus, recalling that $Q_1,\ldots,Q_{\rho}\subset Q$ are $\rho$ different paths between $u$ and $v$, we get that $2^{16\ell^{4}}\geq\rho$. This implies that $\ell\ge t$ for all large enough $n$, as desired.
    \end{proof}

We are now ready to define $\cF$.
Recall the disjoint $m$-tuples  $\mathbf{a}$ and $\mathbf{c}$, the function $p_1\coloneqq C_1(\ln n)^{\psi}/n$, and define $\cP$ as the (random) collection of all path systems produced by Algorithm~\ref{alg:creating the initial segments} applied with $G_1,\mathbf{a},\mathbf{c},$ and all successful lists in $[\lfloor (
\ln n)^{\psi}\rfloor]^{dm}$. For $P\in\mathcal{P}$ and a $(k-d)$-path system $F\subseteq E(K_n)$, we call $F$ a {\it $P$-completion}, if $F,P$ are edge-disjoint and $F\cup P$ is an $(\mathbf{a},\mathbf{c};k)$-path system. Then, $\cF$ is defined as the (random) family of all $(k-d)$-path systems $F\subseteq E(K_n)$ for which there is $P\in \cP$ such that $F$ is a $P$-completion. 
We stress that, if a graph $G_2$ contains some $F\in \cF$, then $G_1\cup G_2$ contains some $(\mathbf{a},\mathbf{c};k)$-path system.

Let $p = C_0(\ln n)^{\psi}/n$ and let $p_2$ be such that $(1-p_1)(1-p_2)=1-p$; note that, for large enough $n$, this implies $p_2\geq C_2\ln n/n$. Since the probability bounds in Lemmas~\ref{lem:algo not rejecting} and~\ref{lem:few paths} are polynomial in $n$, 
to prove Theorem \ref{prop:path-systems-new}, it suffices to show that with probability $1-o(1/(\ln\ln n)^2)$ the graph $G_2\sim G(n,p_2)$ contains some path system in $\cF$. Moreover, we may (and do) assume in what follows that $G_1$ satisfies the assertions of these lemmas deterministically.

We construct an $O(1/n)$-spread measure on $\cF$ and apply Theorem \ref{thm:fracKK}. 
More precisely, define a probability measure $\nu$ on $(\mathcal{F},2^{\mathcal{F}})$: for every $F\in \cF$, let 
\[
    \nu(F) \coloneqq \frac{|\{P\in \cP:F\cup P\text{ is an $(\mathbf{a},\mathbf{c};k)$-path system}\}|}{|\cP| \cdot (n-(d+2)m)!}.
\]
 Note that $\nu$ is indeed a probability measure: by double counting,
\begin{align*}
    \nu (\cF)= \sum_{F\in \cF}\nu(F) 
    &= \sum_{F\in \cF} \sum_{P\in \cP} \frac{\mathbf{1}[F\cup P\text{ is an  $(\mathbf{a},\mathbf{c};k)$-path system}]}{|\cP| \cdot (n-(d+2)m)!}\\
    &= \sum_{P\in \cP} \frac{|\{F\in \cF:F\cup P\text{ is an $(\mathbf{a},\mathbf{c};k)$-path system}\}|}{|\cP| \cdot (n-(d+2)m)!} = \sum_{P\in \cP} \frac{1}{|\cP|} = 1.
\end{align*}

To establish the claimed spreadness property of $\nu$, we require the inequality $1-x\le \e^{-x}$ for $x\ge 0$ and the following version of Stirling's formula due to Robbins~\cite{Robbins}: 
$$
 \sqrt{2\pi n} (n/\e)^n\le n! \le \e^{1/(12n)}\sqrt{2\pi n} (n/\e)^n\quad \text{for all $n\in \mathbb N$}.
$$ 
In particular, we get that, for every large $n$ and an integer $s\in [n-1]$,
\begin{equation}\label{lem:factorial_bounds}
\frac{(n-s)!}{n!}\le \sqrt{\frac{n-s}{n}} \e^{1/(12(n-s))} \frac{((n-s)/\e)^{n-s}}{(n/\e)^n}\le \e^{-s/(2n)+1/(12(n-s))}\frac{(n-s)^{n-s}}{n^n \e^{-s}}\le \bigg(\frac{\e}{n}\bigg)^{s}.
\end{equation}
We note that the final bound also holds for $s=0$, as both sides equal $1$, and for $s=n$, since $n!\geq(n/\e)^n$.

\begin{lemma}\label{lem:nu is spread}
    The probability measure $\nu$ is $q$-spread for $q=87/n$.
\end{lemma}

\begin{proof}
    Fix a non-empty path system $S\subseteq E(K_n)$ with $s\geq 1$ edges and observe that, by definition, 
    \[
        \nu \left (\cF \cap\left \langle S \right\rangle \right) = \frac{|\{(P,F)\in \cP\times \cF : S\subseteq F\text{ and }F\text{ is a $P$-completion}\}|}{|\cP|\cdot (n-(d+2)m)!}.
    \]
    If $S$ is contained in no element of $\cF$, there is nothing to prove; thus, we assume $S\subset F$ for some $F\in\cF$. 
    Call a path in $S$ \emph{full} if it has length $k-d$: note that every such path is maximal in every path system in $\cF$ containing $S$.
      Denote by $S_{\rm full}\subseteq S$ the subgraph consisting of all full paths in $S$ and denote their number by $r$.
    Further, for every $P\in \cP$, denote by $\cF_P$ the collection of all $P$-completions. Then,
\begin{equation}
\label{eq:nu-S->P}
        \nu \left (\cF \cap\left \langle S \right\rangle \right) = \frac{\sum_{P\in \cP} |\cF_P\cap \left\langle S\right\rangle|}{|\cP|\cdot (n-(d+2)m)!}.
\end{equation}
We claim the following:
    \begin{claim}\label{claim:spread1}
        For every $P\in \cP$, we have $|\cF_P\cap \left\langle S\right\rangle|\leq 2^{s}(n-(d+2)m-s+r)!$.
    \end{claim}
    \begin{claim}\label{claim:spread2}
    There are at most $\lfloor (
\ln n)^{\psi}\rfloor^{d(m-r)}\cdot \e^{ r \left(\psi\ln\ln n\right)^{5}}$ path systems $P\in \cP$ that have a $P$-completion~$F$ such that $S_{\rm full}\subset F$.
    \end{claim}
    Before proving these claims, we complete the proof of Lemma~\ref{lem:nu is spread}. 
    By Lemma \ref{lem:algo not rejecting}, 
      for large enough~$n$ we have $|\cP|\geq \frac{1}{1+\eps}\lfloor (
\ln n)^{\psi}\rfloor^{d m}$.
    Then,~\eqref{eq:nu-S->P} together with Claims \ref{claim:spread1} and \ref{claim:spread2} imply that 
    \[
        \nu \left  (\cF \cap\left \langle S \right\rangle \right) \leq (1+\eps )\cdot \left(\frac{\exp \left(\left(\psi\ln\ln n\right)^{5}\right)}{\lfloor (
\ln n)^{\psi}\rfloor^d}\right)^{r}\cdot \frac{2^{s}(n-(d+2)m-s+r)!}{(n-(d+2)m)!}.
    \]    
    Since $d=O(\ln n/\ln\ln n)$, we have
    \[
    \lfloor (
\ln n)^{\psi}\rfloor^d\geq ((\ln n)^\psi-1)^d
    =(\ln n)^{d\psi}
    \left(1-\frac{1}{(\ln n)^\psi}\right)^d
    =(1-o(1))(\ln n)^{d\psi}.
    \]
    In particular, recalling~\eqref{eq:d} and the definition of $k$, we get that $n-(d+2)m>n/2$. Then, applying the bound $\psi\leq 2$ and~\eqref{lem:factorial_bounds}, 
    for large enough $n$, we have
    \begin{align*}
        \nu \left (\cF \cap\left \langle S \right\rangle \right) &\leq (1+\eps) \left( \frac{(\ln n)^{\psi}\exp \left(\frac{(D+\eps/2)\ln n}{\psi\ln\ln n}+(\psi\ln\ln n)^{5}\right)}{n}\right)^{r} 2^{s}\left(\frac{\e} 
        {n-(d+2)m}\right)^{s-r}\\
        & \leq (1+\eps ) \left(
        \exp\left({\frac{(D+\eps)\ln n}{\psi\ln\ln n}}\right)\right)^{r} \left(\frac{2\e} 
        {n-(d+2)m}\right)^{s}\\
        &\leq (1+\eps ) \exp\left({\frac{r(D+\eps)\ln n}{\psi\ln\ln n}}\right) \left(\frac{4\e}{n}\right)^{s}.
    \end{align*}
    Since $r$ is the number of full paths in $S$ and since $k-d\geq \ln n/(\psi\ln\ln n)$, we have $\frac{r\ln n}{\psi\ln\ln n}\leq r (k-d)\leq s$, implying the desired inequality
    \[
        \nu \left (\cF \cap\left \langle S \right\rangle \right) \leq (1+\eps )\left(\frac{4\e\cdot \e^{D+\eps}}{n}\right)^{s} \leq \left(\frac{87}{n}\right)^{s}.
    \]
    This concludes the proof of Lemma~\ref{lem:nu is spread}.
    \end{proof}
    We now turn to the proofs of Claims \ref{claim:spread1} and \ref{claim:spread2}.     
    \begin{proof}[Proof of Claim \ref{claim:spread1}]
      Fix $P\in\mathcal{P}$ such that $S$ belongs to at least one $P$-completion. Otherwise, the statement is immediate since $|\mathcal{F}_P\cap\langle S\rangle|=0.$ Let $\Sigma_0$ be the set of all vertices from $[n]$ that do not belong to $P\cup S\cup\mathbf{c}$. Let $\Sigma_1$ be the set of all maximal paths of $S$ that touch neither $P$ nor $\mathbf{c}$. Every $F\in\mathcal{F}_P\cap\langle S\rangle$ is uniquely determined by an ordering of $\Sigma_0\cup\Sigma_1$ and an orientation of every path in $\Sigma_1$. Recall that the number of vertices that do not belong to $P\cup\mathbf{c}$ equals $n-(d+2)m$. Therefore, $|\Sigma_0\cup\Sigma_1|=n-(d+2)m-s+r$, and so the total number of orderings of $\Sigma_0\cup\Sigma_1$ is exactly $(n-(d+2)m-s+r)!$. The number of orientations of the paths in $\Sigma_1$ equals $2^{|\Sigma_1|}\leq 2^s$. This completes the proof.       
    \end{proof}
    \begin{proof}[Proof of Claim \ref{claim:spread2}]
        By definition, there are $\mathbf{b}=(b_{i_1},\ldots ,b_{i_r})$ and $\mathbf{c}^*=(c_{i_1},\ldots ,c_{i_r})\subset\mathbf{c}$ 
        such that $S_{\rm full}$ is a $(\mathbf{b},\mathbf{c}^*;k-d)$-path system. Let $\mathbf{a^*} = (a_{i_1},\ldots ,a_{i_r})$ and note that, if $P\in \cP$ is such that some $P$-completion contains $S_{\rm full}$, then $P$ contains an $(\mathbf{a}^*,\mathbf{b};d)$-path system. Therefore, it suffices to bound the number of $P\in \cP$ that contain an $(\mathbf{a}^*,\mathbf{b};d)$-path system.

        By the assertion of Lemma~\ref{lem:few paths}, there are at most $\exp((\psi\ln\ln n)^{5})$ paths of length $d$ between any pair of vertices in $G_1$. 
        Hence, the number of $(\mathbf{a}^*,\mathbf{b};d)$-path systems in $G_1$ is bounded from above by $\exp(r(\psi\ln\ln n)^{5})$. 
        To conclude the proof, it remains to observe that, for every $(\mathbf{a}^*,\mathbf{b};d)$-path system $W$, there are at most $\lfloor (
\ln n)^{\psi}\rfloor^{d (m-r)}$ path systems $P\in \cP$ containing $W$. As every path system $P\in \cP$ is uniquely defined by a list $L\in[\lfloor (
\ln n)^{\psi}\rfloor]^{dm}$, it is enough to bound the number of lists corresponding to path systems containing $W$. 
In order to establish the required bound on the number of lists expose the entries of $L$ in order. At each of the $dr$ extension steps belonging to the prescribed paths in $W$, there is at most one possible entry: the preceding entries determine the unavailable set, while $W$ specifies the next vertex. Each of the remaining $d(m-r)$ entries has at most $\lfloor (
\ln n)^{\psi}\rfloor$ choices.
        Therefore, we are left with at most $\lfloor (
\ln n)^{\psi}\rfloor^{d(m-r)}$ choices, completing the proof.
\end{proof}

\begin{proof}[Proof of Theorem \ref{prop:path-systems-new} for $\psi\leq 2$]
    Note that $\ell(\cF)\leq n$, which in turn implies
    \[
        \frac{C_2  \ln \ell(\cF)}{n} \leq\frac{C_2\ln n}{n} < p_2.
    \]
    Thus, by applying Theorem \ref{thm:fracKK} and Lemma \ref{lem:nu is spread} together, we find that with probability $1-o(1/(\ln\ln n)^2)$ the graph $G_2$ contains a member of $\cF$.
    As mentioned before the proof of Lemma~\ref{lem:nu is spread}, this suffices to conclude the proof of Theorem \ref{prop:path-systems-new} under the assumption that $\psi \leq 2$.  
\end{proof}

\subsubsection{Large values of $\psi$}\label{sec:large}

Throughout this subsection, we assume that $\psi> 2$.
Our set $\cF$ in this subsection is simpler and consists of all $(\mathbf{a},\mathbf{c};k)$-path systems in $K_n$. We begin by showing that the uniform measure on the family $\cF$ is $O\left((\ln n)^{\psi-1}/n\right)$-spread, which is very similar to the proof of Claim \ref{claim:spread1}. 

\begin{lemma}
    The uniform measure on $\cF$ is $(4\e(\ln n)^{\psi-1}/n)$-spread.
\end{lemma}

\begin{proof}
     Fix $S\subseteq E(K_n)$ with $s$ edges. If $S\not\subseteq F$ for any $F\in \cF$, then there is nothing to prove, so assume otherwise. 
     Denote the uniform measure on $\cF$ by $\nu$. Then,
    \[
        \nu \left (\cF \cap\left \langle S \right\rangle \right) =  \frac{|\cF \cap\left \langle S \right\rangle|}{(n-2m)!}=\frac{|\{F\in \cF : S\subseteq F\}|}{(n-2m)!}.
    \]
    Since $S\subseteq F$ for some $F\in \cF$, $S$ is a path system. Let $r$ be the number of all $k$-paths in $S$. 

    As in Claim \ref{claim:spread1}, we have $|\cF\cap \left\langle S\right\rangle|\leq 2^{s}(n-2m-s+r)!$. Then, by \eqref{lem:factorial_bounds}, we have
    \[
    \nu \left (\cF \cap\left \langle S \right\rangle \right) \leq \frac{2^{s}(n-2m-s+r)!}{(n-2m)!} \leq 2^{s}\cdot \left(\frac{\e}{n-2m}\right)^{s-r}\leq  \left(\frac{4\e}{n}\right)^{s}\cdot n^{r} ,
    \]
    where the last inequality holds as $2m \leq  n/2$ by the upper bound assumption on $\psi$. Since $r$ is the number of $k$-paths in $S$, we have $r k\leq s$, implying the desired inequality:
    \[
        \nu \left (\cF \cap\left \langle S \right\rangle \right) \leq \left(\frac{4\e\cdot  n^{1/k}}{n}\right)^{s} \leq \left(\frac{4\e\cdot  (\ln n)^{\psi-1}}{n}\right)^{s} ,
    \]
    where the last inequality holds by the assumption that $\psi>2$.
\end{proof}

Finally, by Theorem \ref{thm:fracKK}, we find that,  with probability $1-o(1/(\ln\ln n)^2)$, the graph $G(n,p)$ contains an $(\mathbf{a},\mathbf{c};k)$-path system, thus concluding the proof of Theorem \ref{prop:path-systems-new} when $\psi >2$.

\section{A roadmap for tree universality}
\label{sc:overview}

We now outline the proof of Theorem~\ref{thm:univ_trees}. 
 Recall that the asymptotic notations $O$, $\Omega$, $\Theta$ used below only hide absolute constants.
Fix a large absolute constant $K$ and consider the function 
\begin{equation}\label{eq:alpha}
\alpha = \alpha(n) := \frac{n\ln\ln n}{K^3\ln n}.
\end{equation}
A path in a graph $G$ is called \emph{bare} if every internal vertex has degree 2 in $G$. We divide the family $\mathcal{T}$ of trees on $n$ vertices and maximum degree at most $\Delta$ in three subfamilies:
\begin{itemize}
    \item the family $\cT_1\subset\mathcal{T}$ consists of the trees that have $K\alpha$ vertex-disjoint bare paths of length $n/(K^2\alpha)$,
    \item the family $\cT_2$ consists of the trees in $\cT\setminus \cT_1$ with at most $\exp(\Delta^5) \alpha$ leaves, and $\cT_3 =  \cT\setminus (\cT_1\cup \cT_2)$.
\end{itemize}
The reason behind this case distinction is that showing $\cT_1$-universality and $\cT_2\cup \cT_3$-universality of the random graph $G(n,C\ln n/n)$ requires different embedding strategies.
In the next few paragraphs, we explain each of these in some detail.

\subsection{\texorpdfstring{Universality for $\cT_1$: perfect matchings and linking systems}{Universality for T1: perfect matchings and linking systems}}\label{sec:1.2.1}

Fix a tree $T$ in $\cT_1$ and write $G\sim G(n,C\ln n/n)$ for the host graph. Our first step here is to edge-decompose $T$ into two subforests: a forest $(P_i)_{i=1}^\alpha$ consisting of $\alpha$ vertex-disjoint bare paths in $T$, each with $\ell = \Theta(Kk)$ edges where $k = \big\lceil \frac{25\ln n}{\ln\ln n}\big\rceil$, and its complement $F=T\setminus E(P_{1}\cup \ldots \cup P_\alpha )$, where we discard the isolated internal vertices of the paths $P_i$, but retain their endpoints.
We also partition the vertex set of the random graph in three parts defined independently of the choice of $T$: $V_1$ of size $|V(F)|+\eps n$ for some small $\eps>0$, $V_3$ of size $3(2k+1)\alpha$ and $V_2=V\setminus (V_1\cup V_3)$.

We start by exposing the graph $G[V_1]$ and embedding $F$ in $G[V_1]$: such an embedding typically exists for all forests $F$ by the universality result of Alon, Krivelevich and Sudakov~\cite{AKS07} mentioned in the introduction (see Theorem~\ref{thm:AKS07}).
This embedding of $F$ fixes the images of the starting vertices $X_1$ and the ending vertices $X_2$ of the paths in $P$, and leaves a set $W$ of $\eps n$ vertices unused.
As a next step, we expose all edges of $G$ outside $G[V_3]$ and show that, with high probability, the Lov\'asz Local Lemma can be applied to produce subsets $U_1',\ldots,U_{N-6}'$ (partitioning $V_2\cup W$) and $U_{N-5}',\ldots,U_{N}'$ (contained in $V_3$), all of size $\alpha$ and grouped into triplets together with $X_1,X_2$, such that all pairs of sets connected by thin solid lines in Figure~\ref{fig:matchings} induce a perfect matching (see Claim~\ref{cl:matchings}). 
We then find a linking system in $G[V_3]$ which links the sets $A_1=U'_{N-5}\sqcup U'_{N-4}\sqcup U'_{N-3}$ and $A_2=U'_{N-2}\sqcup U'_{N-1}\sqcup U'_N$ 
 using  Theorem~\ref{thm:linking_system} applied with $|V_3|$ in place of $n$. This further allows to embed the paths $P_1,\ldots,P_{\alpha}$ as follows: for every set $Y\in \{X_1, X_2, U_1', U_2', U_{N-7}', U_{N-6}'\}$, denote by $\hat Y$ the set obtained by following the matching from $Y$ to $A_1\cup A_2$ from left to right in Figure~\ref{fig:matchings}. 
Then, embedding $P_1,\ldots,P_{\alpha}$ from $X_1$ to $X_2$ is obtained by following the chain 
\[X_1 \equiv\equiv \hat X_1 \overset{\textrm{LS}}{--\;} \hat U_{N-7}' \equiv\equiv U_{N-7}' \equiv U_1' \equiv\equiv \hat U_1' \overset{\textrm{LS}}{--\;} \hat U_{N-6}' \equiv\equiv U_{N-6}' \equiv U_2' \equiv\equiv \hat U_2' \overset{\textrm{LS}}{--\;} \hat X_2 \equiv\equiv X_2,\]
where $\equiv$ stands for using a single matching between two sets, $\equiv\equiv$ stands for using a sequence of matching between the left and the right part of Figure~\ref{fig:matchings}, and the remaining transitions use the $A_1-A_2$ linking system in $G[V_3]$.

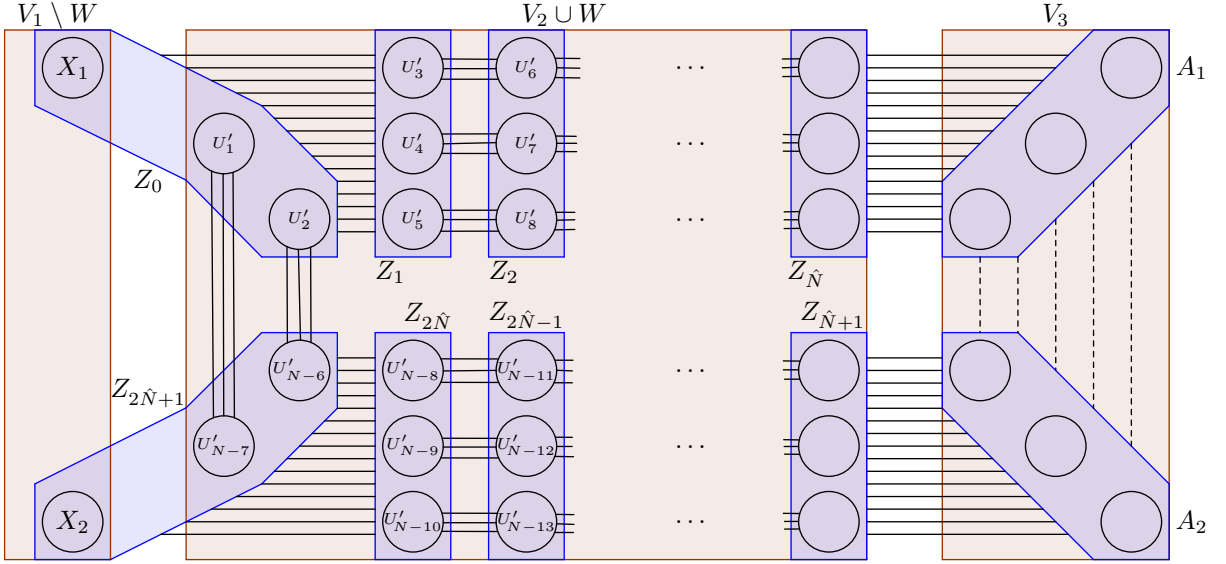
\begin{figure}
\centering
\definecolor{qqqqff}{rgb}{0,0,1}
\definecolor{zzttqq}{rgb}{0.6,0.2,0}
\begin{tikzpicture}[line cap=round,line join=round,x=1cm,y=1cm]
\clip(-7.8,-2.1) rectangle (12,5.5);
\fill[line width=0.5pt,color=zzttqq,fill=zzttqq,fill opacity=0.10000000149011612] (-7.4,5) -- (-7.4,-2) -- (-6,-2) -- (-6,5) -- cycle;
\fill[line width=0.5pt,color=zzttqq,fill=zzttqq,fill opacity=0.10000000149011612] (-5,5) -- (-5,-2) -- (4,-2) -- (4,5) -- cycle;
\fill[line width=0.5pt,color=zzttqq,fill=zzttqq,fill opacity=0.10000000149011612] (5,5) -- (5,-2) -- (8,-2) -- (8,5) -- cycle;
\fill[line width=0.5pt,color=qqqqff,fill=qqqqff,fill opacity=0.1] (-7,-1) -- (-7,-2) -- (-6,-2) -- (-4,-1) -- (-3,0) -- (-3,1) -- (-4,1) -- (-5,0.0061147787968263145) -- cycle;
\fill[line width=0.5pt,color=qqqqff,fill=qqqqff,fill opacity=0.1] (-7,5) -- (-7,4) -- (-5,3.015286946992065) -- (-4,2) -- (-3,2) -- (-3,3) -- (-4,4) -- (-6,5) -- cycle;
\fill[line width=0.5pt,color=qqqqff,fill=qqqqff,fill opacity=0.1] (-2.5,5) -- (-2.5,2) -- (-1.5,2) -- (-1.5,5) -- cycle;
\fill[line width=0.5pt,color=qqqqff,fill=qqqqff,fill opacity=0.1] (-1,5) -- (-1,2) -- (0,2) -- (0,5) -- cycle;
\fill[line width=0.5pt,color=qqqqff,fill=qqqqff,fill opacity=0.1] (3,5) -- (3,2) -- (4,2) -- (4,5) -- cycle;
\fill[line width=0.5pt,color=qqqqff,fill=qqqqff,fill opacity=0.1] (5,3) -- (7,5) -- (8,5) -- (8,4) -- (6,2) -- (5,2) -- cycle;
\fill[line width=0.5pt,color=qqqqff,fill=qqqqff,fill opacity=0.1] (-2.5,1) -- (-1.5,1) -- (-1.5,-2) -- (-2.5,-2) -- cycle;
\fill[line width=0.5pt,color=qqqqff,fill=qqqqff,fill opacity=0.1] (-1,1) -- (0,1) -- (0,-2) -- (-1,-2) -- cycle;
\fill[line width=0.5pt,color=qqqqff,fill=qqqqff,fill opacity=0.1] (3,1) -- (4,1) -- (4,-2) -- (3,-2) -- cycle;
\fill[line width=0.5pt,color=qqqqff,fill=qqqqff,fill opacity=0.1] (5,1) -- (6,1) -- (8,-1) -- (8,-2) -- (7,-2) -- (5,0) -- cycle;
\draw [line width=0.5pt,color=zzttqq] (-7.4,5)-- (-7.4,-2);
\draw [line width=0.5pt,color=zzttqq] (-7.4,-2)-- (-6,-2);
\draw [line width=0.5pt,color=zzttqq] (-6,-2)-- (-6,5);
\draw [line width=0.5pt,color=zzttqq] (-6,5)-- (-7.4,5);
\draw [line width=0.5pt,color=zzttqq] (-5,5)-- (-5,-2);
\draw [line width=0.5pt,color=zzttqq] (-5,-2)-- (4,-2);
\draw [line width=0.5pt,color=zzttqq] (4,-2)-- (4,5);
\draw [line width=0.5pt,color=zzttqq] (4,5)-- (-5,5);
\draw [line width=0.5pt,color=zzttqq] (5,5)-- (5,-2);
\draw [line width=0.5pt,color=zzttqq] (5,-2)-- (8,-2);
\draw [line width=0.5pt,color=zzttqq] (8,-2)-- (8,5);
\draw [line width=0.5pt,color=zzttqq] (8,5)-- (5,5);
\draw [line width=0.5pt] (-6.5,4.5) circle (0.4cm);
\draw [line width=0.5pt] (-6.5,-1.5) circle (0.4cm);
\draw [line width=0.5pt] (-4.5,-0.5) circle (0.4cm);
\draw [line width=0.5pt] (-4.5,3.5) circle (0.4cm);
\draw [line width=0.5pt] (-3.5,2.5) circle (0.4cm);
\draw [line width=0.5pt] (-3.5,0.5) circle (0.4cm);
\draw [line width=0.5pt] (-2,2.5) circle (0.4cm);
\draw [line width=0.5pt] (-2,3.5) circle (0.4cm);
\draw [line width=0.5pt] (-2,4.5) circle (0.4cm);
\draw [line width=0.5pt] (-0.5,4.5) circle (0.4cm);
\draw [line width=0.5pt] (-0.5,3.5) circle (0.4cm);
\draw [line width=0.5pt] (-0.5,2.5) circle (0.4cm);
\draw [line width=0.5pt] (3.5,4.5) circle (0.4cm);
\draw [line width=0.5pt] (3.5,3.5) circle (0.4cm);
\draw [line width=0.5pt] (3.5,2.5) circle (0.4cm);
\draw [line width=0.5pt] (-2,-1.5) circle (0.4cm);
\draw [line width=0.5pt] (-2,-0.5) circle (0.4cm);
\draw [line width=0.5pt] (-2,0.5) circle (0.4cm);
\draw [line width=0.5pt] (-0.5,0.5) circle (0.4cm);
\draw [line width=0.5pt] (-0.5,-0.5) circle (0.4cm);
\draw [line width=0.5pt] (-0.5,-1.5) circle (0.4cm);
\draw [line width=0.5pt] (3.5,0.5) circle (0.4cm);
\draw [line width=0.5pt] (3.5,-0.5) circle (0.4cm);
\draw [line width=0.5pt] (3.5,-1.5) circle (0.4cm);
\draw [line width=0.5pt,color=qqqqff] (-7,-1)-- (-7,-2);
\draw [line width=0.5pt,color=qqqqff] (-7,-2)-- (-6,-2);
\draw [line width=0.5pt,color=qqqqff] (-6,-2)-- (-4,-1);
\draw [line width=0.5pt,color=qqqqff] (-4,-1)-- (-3,0);
\draw [line width=0.5pt,color=qqqqff] (-3,0)-- (-3,1);
\draw [line width=0.5pt,color=qqqqff] (-3,1)-- (-4,1);
\draw [line width=0.5pt,color=qqqqff] (-4,1)-- (-5,0.0061147787968263145);
\draw [line width=0.5pt,color=qqqqff] (-5,0.0061147787968263145)-- (-7,-1);
\draw [line width=0.5pt,color=qqqqff] (-7,5)-- (-7,4);
\draw [line width=0.5pt,color=qqqqff] (-7,4)-- (-5,3.015286946992065);
\draw [line width=0.5pt,color=qqqqff] (-5,3.015286946992065)-- (-4,2);
\draw [line width=0.5pt,color=qqqqff] (-4,2)-- (-3,2);
\draw [line width=0.5pt,color=qqqqff] (-3,2)-- (-3,3);
\draw [line width=0.5pt,color=qqqqff] (-3,3)-- (-4,4);
\draw [line width=0.5pt,color=qqqqff] (-4,4)-- (-6,5);
\draw [line width=0.5pt,color=qqqqff] (-6,5)-- (-7,5);
\draw [line width=0.5pt] (5.5,2.5) circle (0.4cm);
\draw [line width=0.5pt] (6.5,3.5) circle (0.4cm);
\draw [line width=0.5pt] (7.5,4.5) circle (0.4cm);
\draw [line width=0.5pt] (5.5,0.5) circle (0.4cm);
\draw [line width=0.5pt] (6.5,-0.5) circle (0.4cm);
\draw [line width=0.5pt] (7.5,-1.5) circle (0.4cm);
\draw [line width=0.5pt,color=qqqqff] (-2.5,5)-- (-2.5,2);
\draw [line width=0.5pt,color=qqqqff] (-2.5,2)-- (-1.5,2);
\draw [line width=0.5pt,color=qqqqff] (-1.5,2)-- (-1.5,5);
\draw [line width=0.5pt,color=qqqqff] (-1.5,5)-- (-2.5,5);
\draw [line width=0.5pt,color=qqqqff] (-1,5)-- (-1,2);
\draw [line width=0.5pt,color=qqqqff] (-1,2)-- (0,2);
\draw [line width=0.5pt,color=qqqqff] (0,2)-- (0,5);
\draw [line width=0.5pt,color=qqqqff] (0,5)-- (-1,5);
\draw [line width=0.5pt,color=qqqqff] (3,5)-- (3,2);
\draw [line width=0.5pt,color=qqqqff] (3,2)-- (4,2);
\draw [line width=0.5pt,color=qqqqff] (4,2)-- (4,5);
\draw [line width=0.5pt,color=qqqqff] (4,5)-- (3,5);
\draw [line width=0.5pt,color=qqqqff] (5,3)-- (7,5);
\draw [line width=0.5pt,color=qqqqff] (7,5)-- (8,5);
\draw [line width=0.5pt,color=qqqqff] (8,5)-- (8,4);
\draw [line width=0.5pt,color=qqqqff] (8,4)-- (6,2);
\draw [line width=0.5pt,color=qqqqff] (6,2)-- (5,2);
\draw [line width=0.5pt,color=qqqqff] (5,2)-- (5,3);
\draw [line width=0.5pt,color=qqqqff] (-2.5,1)-- (-1.5,1);
\draw [line width=0.5pt,color=qqqqff] (-1.5,1)-- (-1.5,-2);
\draw [line width=0.5pt,color=qqqqff] (-1.5,-2)-- (-2.5,-2);
\draw [line width=0.5pt,color=qqqqff] (-2.5,-2)-- (-2.5,1);
\draw [line width=0.5pt,color=qqqqff] (-1,1)-- (0,1);
\draw [line width=0.5pt,color=qqqqff] (0,1)-- (0,-2);
\draw [line width=0.5pt,color=qqqqff] (0,-2)-- (-1,-2);
\draw [line width=0.5pt,color=qqqqff] (-1,-2)-- (-1,1);
\draw [line width=0.5pt,color=qqqqff] (3,1)-- (4,1);
\draw [line width=0.5pt,color=qqqqff] (4,1)-- (4,-2);
\draw [line width=0.5pt,color=qqqqff] (4,-2)-- (3,-2);
\draw [line width=0.5pt,color=qqqqff] (3,-2)-- (3,1);
\draw [line width=0.5pt,color=qqqqff] (5,1)-- (6,1);
\draw [line width=0.5pt,color=qqqqff] (6,1)-- (8,-1);
\draw [line width=0.5pt,color=qqqqff] (8,-1)-- (8,-2);
\draw [line width=0.5pt,color=qqqqff] (8,-2)-- (7,-2);
\draw [line width=0.5pt,color=qqqqff] (7,-2)-- (5,0);
\draw [line width=0.5pt,color=qqqqff] (5,0)-- (5,1);
\draw [line width=0.5pt] (4,2.3333)-- (5,2.3333);
\draw [line width=0.5pt] (4,2.5)-- (5,2.5);
\draw [line width=0.5pt] (4,2.666)-- (5,2.666);
\draw [line width=0.5pt] (4,2.8333)-- (5,2.8333);
\draw [line width=0.5pt] (4,3)-- (5,3);
\draw [line width=0.5pt] (4,3.1666)-- (5.1666,3.1666);
\draw [line width=0.5pt] (4,3.333)-- (5.333,3.333);
\draw [line width=0.5pt] (4,3.5)-- (5.5,3.5);
\draw [line width=0.5pt] (4,3.666)-- (5.666,3.666);
\draw [line width=0.5pt] (4,3.8333)-- (5.8333,3.8333);
\draw [line width=0.5pt] (4,4)-- (6,4);
\draw [line width=0.5pt] (4,4.1666)-- (6.1666,4.1666);
\draw [line width=0.5pt] (4,4.333)-- (6.333,4.333);
\draw [line width=0.5pt] (4,4.5)-- (6.5,4.5);
\draw [line width=0.5pt] (4,4.666)-- (6.666,4.666);

\draw [line width=0.5pt] (4,0.333)-- (5,0.333);
\draw [line width=0.5pt] (4,0.5)-- (5,0.5);
\draw [line width=0.5pt] (4,0.666)-- (5,0.666);
\draw [line width=0.5pt] (4,0.1666)-- (5,0.1666);
\draw [line width=0.5pt] (4,0)-- (5,0);
\draw [line width=0.5pt] (4,-0.1666)-- (5+0.1666,-0.1666);
\draw [line width=0.5pt] (4,-0.333)-- (5+0.333,-0.333);
\draw [line width=0.5pt] (4,-0.5)-- (5.5,-0.5);
\draw [line width=0.5pt] (4,-0.666)-- (5+0.666,-0.666);
\draw [line width=0.5pt] (4,-0.8333)-- (5+0.8333,-0.8333);
\draw [line width=0.5pt] (4,-1)-- (6,-1);
\draw [line width=0.5pt] (4,-1.333)-- (5+1.333,-1.333);
\draw [line width=0.5pt] (4,-1.1666)-- (5+1.1666,-1.1666);
\draw [line width=0.5pt] (4,-1.666)-- (5+1.666,-1.666);
\draw [line width=0.5pt] (4,-1.5)-- (6.5,-1.5);

\draw [line width=0.5pt] (-5.333,-1.666)-- (-2.5,-1.666);
\draw [line width=0.5pt] (-5,-1.5)-- (-2.5,-1.5);
\draw [line width=0.5pt] (-4.666,-1.333)-- (-2.5,-1.333);
\draw [line width=0.5pt] (-4.333,-1.1666)-- (-2.5,-1.1666);
\draw [line width=0.5pt] (-4,-1)-- (-2.5,-1);
\draw [line width=0.5pt] (-3.8333,-0.8333)-- (-2.5,-0.8333);
\draw [line width=0.5pt] (-3.666,-0.666)-- (-2.5,-0.666);
\draw [line width=0.5pt] (-3.5,-0.5)-- (-2.5,-0.5);
\draw [line width=0.5pt] (-3.333,-0.333)-- (-2.5,-0.333);
\draw [line width=0.5pt] (-3.1666,-0.1666)-- (-2.5,-0.1666);
\draw [line width=0.5pt] (-3,-0)-- (-2.5,-0);
\draw [line width=0.5pt] (-3,0.1666)-- (-2.5,0.1666);
\draw [line width=0.5pt] (-3,0.333)-- (-2.5,0.333);
\draw [line width=0.5pt] (-3,0.5)-- (-2.5,0.5);
\draw [line width=0.5pt] (-3,0.666)-- (-2.5,0.666);

\draw [line width=0.5pt] (-1.630741248736642,0.6537789797580461)-- (-0.8699469290148681,0.652115974547278);
\draw [line width=0.5pt] (-1.6000473682306313,0.4938443245085742)-- (-0.8999995171744701,0.5006214983433119);
\draw [line width=0.5pt] (-1.632265105322805,0.34261179448016965)-- (-0.8671451706103335,0.3412409886100663);
\draw [line width=0.5pt] (-1.613669030344414,-0.396325596770588)-- (-0.8866469891767438,-0.3975104602383297);
\draw [line width=0.5pt] (-1.6002025443635812,-0.5127277046888195)-- (-0.8998022991636745,-0.5125746404894845);
\draw [line width=0.5pt] (-1.6316960380800472,-0.65605188763378)-- (-0.8680563829436464,-0.6566349225889295);
\draw [line width=0.5pt] (-1.616480466971059,-1.3863656399443238)-- (-0.884105024895409,-1.388360715471223);
\draw [line width=0.5pt] (-1.6001960916524047,-1.512523372947719)-- (-0.8997890308218208,-1.5129896433572771);
\draw [line width=0.5pt] (-1.620560742074975,-1.626593244469468)-- (-0.8786187047054375,-1.629026650143204);

\draw [line width=0.5pt] (-5.333,4.666)-- (-2.5,4.666);
\draw [line width=0.5pt] (-5,4.5)-- (-2.5,4.5);
\draw [line width=0.5pt] (-4.666,4.333)-- (-2.5,4.333);
\draw [line width=0.5pt] (-4.333,4.1666)-- (-2.5,4.1666);
\draw [line width=0.5pt] (-4,4)-- (-2.5,4);
\draw [line width=0.5pt] (-3.8333,3.8333)-- (-2.5,3.8333);
\draw [line width=0.5pt] (-3.666,3.666)-- (-2.5,3.666);
\draw [line width=0.5pt] (-3.5,3.5)-- (-2.5,3.5);
\draw [line width=0.5pt] (-3.333,3.333)-- (-2.5,3.333);
\draw [line width=0.5pt] (-3.1666,3.1666)-- (-2.5,3.1666);
\draw [line width=0.5pt] (-3,3)-- (-2.5,3);
\draw [line width=0.5pt] (-3,2.8333)-- (-2.5,2.8333);
\draw [line width=0.5pt] (-3,2.333)-- (-2.5,2.333);
\draw [line width=0.5pt] (-3,2.5)-- (-2.5,2.5);
\draw [line width=0.5pt] (-3,2.666)-- (-2.5,2.666);

\draw [line width=0.5pt] (-1.6202094277114218,3.6255353384538167)-- (-0.8792899436036454,3.62703990979666);
\draw [line width=0.5pt] (-1.6000254719866844,3.4954859175322825)-- (-0.8998458058280852,3.511105474401807);
\draw [line width=0.5pt] (-1.6222092527453005,3.368568834408517)-- (-0.8773591520952032,3.3673347433199283);
\draw [line width=0.5pt] (-1.6001522330881521,4.488965359289557)-- (-0.8998371115852579,4.488585789595509);
\draw [line width=0.5pt] (-1.6176039883240028,4.617359661955421)-- (-0.8839135286547937,4.61229604853166);
\draw [line width=0.5pt] (-1.6294293617220224,4.34940982088379)-- (-0.8704511523764413,4.349116125106204);
\draw [line width=0.5pt] (-1.617520500273712,2.6170872849165465)-- (-0.8822035135062061,2.6179850594927654);
\draw [line width=0.5pt] (-1.6000439420953714,2.4940711092614736)-- (-0.8999532212277108,2.493882743296395);
\draw [line width=0.5pt] (-1.6333711481055018,2.340052243033789)-- (-0.8670518176124913,2.3410252750045895);
\draw [line width=0.5pt] (-4.510587080615212,3.1001401323913003)-- (-4.502667066678909,-0.10000889165466403);
\draw [line width=0.5pt] (-4.3794438254804575,3.1185996738527697)-- (-4.361287575103965,-0.1248215582160111);
\draw [line width=0.5pt] (-4.6633207448044764,3.134861212254147)-- (-4.6382099995405355,-0.12463618178225389);
\draw [line width=0.5pt] (-3.515321351421906,2.1002935374670475)-- (-3.487440332240429,0.8998027698075255);
\draw [line width=0.5pt] (-3.3535231193999406,2.127784305207752)-- (-3.3483394297495703,0.8701338560998095);
\draw [line width=0.5pt] (-3.665903171141988,2.1360272842574153)-- (-3.660878353005648,0.8662214569549278);
\draw [line width=0.5pt] (-0.1000005110486823,4.50063940494578)-- (0.2075650311500004,4.49450715847653);
\draw [line width=0.5pt] (-0.12039598107499683,4.626098329949234)-- (0.2075650311500004,4.623123585497105);
\draw [line width=0.5pt] (-0.12124806035580105,4.3713649806010935)-- (0.2075650311500004,4.365890731455955);
\draw [line width=0.5pt] (3.1180787124083844,4.61889545863641)-- (2.8840116315305266,4.616998993734221);
\draw [line width=0.5pt] (3.1000602020550185,4.506939597951431)-- (2.8840116315305266,4.500631750239414);
\draw [line width=0.5pt] (3.1163917584044265,4.386665464310511)-- (2.890136223293411,4.384264506744608);
\draw [line width=0.5pt] (-0.11229499867932069,3.598411543788989)-- (0.1585682970469244,3.594192169332503);
\draw [line width=0.5pt] (-0.10000612817171112,3.502214158940641)-- (0.16469288880980892,3.5023232928892347);
\draw [line width=0.5pt] (-0.11455194908213512,3.3930897570687493)-- (0.16469288880980892,3.3920806411573134);
\draw [line width=0.5pt] (-0.11023693466054496,2.5899152539739045)-- (0.1401945217582709,2.589759120219439);
\draw [line width=0.5pt] (-0.1002637468713593,2.4854766417562844)-- (0.1401945217582709,2.4917656520132865);
\draw [line width=0.5pt] (-0.12601322177245183,2.358106061753933)-- (0.1401945217582709,2.363149224992711);
\draw [line width=0.5pt] (3.116880049864953,3.614974361526903)-- (2.8962608150562956,3.618690536384041);
\draw [line width=0.5pt] (3.1001003187279474,3.508957952808017)-- (2.890136223293411,3.5084478846521194);
\draw [line width=0.5pt] (3.113404186001499,3.3973175935282165)-- (2.8962608150562956,3.3982052329201977);
\draw [line width=0.5pt] (3.1000866657619572,2.4916738304958623)-- (2.9085099985820646,2.4917656520132865);
\draw [line width=0.5pt] (3.1126755672281217,2.5998988677510546)-- (2.9085099985820646,2.5958837119823235);
\draw [line width=0.5pt] (3.1190304774563273,2.3780892831091283)-- (2.9085099985820646,2.37539840851848);
\draw [line width=0.5pt] (-0.12356096505709141,0.6352540312569339)-- (0.11608335581900116,0.6350086403674455);
\draw [line width=0.5pt] (-0.10000697183170348,0.5023616555117862)-- (0.107796207790477,0.49412712388253416);
\draw [line width=0.5pt] (-0.1344787662233694,0.33753699603168175)-- (0.107796207790477,0.3366713113405745);
\draw [line width=0.5pt] (-0.11770904444656571,-0.38229857561591646)-- (0.107796207790477,-0.38431056714103046);
\draw [line width=0.5pt] (-0.10002858481435817,-0.5047819488072273)-- (0.10578930870432303,-0.5050895144850224);
\draw [line width=0.5pt] (-0.12398444080068516,-0.6364269007198602)-- (0.12267298127416049,-0.6401588950437227);
\draw [line width=0.5pt] (-0.11211063698201928,-1.4023176471541292)-- (0.12830087213077296,-1.4055520515430244);
\draw [line width=0.5pt] (-0.10090026944597597,-1.5268217276047875)-- (0.12830087213077296,-1.5349935412451121);
\draw [line width=0.5pt] (-0.1260769786286658,-1.6420618671161715)-- (0.12267298127416049,-1.64192346752075);
\draw [line width=0.5pt] (3.1194596390625566,0.6232437978058141)-- (2.897223173584116,0.6261165476940926);
\draw [line width=0.5pt] (3.1000881698310003,0.5083981004328983)-- (2.897223173584116,0.5079308397052299);
\draw [line width=0.5pt] (3.1169655006716086,0.38474128091875803)-- (2.9141068461539534,0.3841172408597546);
\draw [line width=0.5pt] (3.100029781362447,-0.5048810043052798)-- (2.908478955297341,-0.5050895144850224);
\draw [line width=0.5pt] (3.1084879523537445,-0.41803466252228166)-- (2.9028510644407284,-0.41504326077922216);
\draw [line width=0.5pt] (3.114985982567921,-0.6084629262965502)-- (2.9141068461539534,-0.6063915499040476);
\draw [line width=0.5pt] (3.1150328111783328,-1.3913709820962064)-- (2.897223173584116,-1.3886683789731868);
\draw [line width=0.5pt] (3.100022889891602,-1.4957208190813733)-- (2.9141068461539534,-1.4955983052488246);
\draw [line width=0.5pt] (3.11152025269382,-1.5953073236059359)-- (2.9141068461539534,-1.59690034066785);
\draw [line width=0.5pt,dash pattern=on 2pt off 2pt] (5.5,2)-- (5.5,1);
\draw [line width=0.5pt,dash pattern=on 2pt off 2pt] (6,2)-- (6,1);
\draw [line width=0.5pt,dash pattern=on 2pt off 2pt] (6.5,2.5)-- (6.5,0.5);
\draw [line width=0.5pt,dash pattern=on 2pt off 2pt] (7,3)-- (7,0);
\draw [line width=0.5pt,dash pattern=on 2pt off 2pt] (7.5,3.5)-- (7.5,-0.5);

\node at (1.7, 4.5) {$\dots$};
\node at (1.7, 3.5) {$\dots$};
\node at (1.7, 2.5) {$\dots$};
\node at (1.7, 0.5) {$\dots$};
\node at (1.7, -0.5) {$\dots$};
\node at (1.7, -1.5) {$\dots$};

\node at (-6.5, 4.5) {$X_1$};
\node at (-4.5, 3.5) {\tiny{$U_1'$}};
\node at (-3.5, 2.5) {\tiny{$U_2'$}};

\node at (-5.5, 3) {$Z_0$};
\node at (-5.5, 0.2) {$Z_{2\hat N+1}$};

\node at (-2.3, 1.8) {$Z_1$};
\node at (-1.8, 1.22) {$Z_{2\hat N}$};

\node at (-0.8, 1.8) {$Z_2$};
\node at (-0.5, 1.22) {$Z_{2\hat N-1}$};

\node at (3.2, 1.78) {$Z_{\hat N}$};
\node at (3.55, 1.22) {$Z_{\hat N+1}$};

\node at (-2, 4.5) {\tiny{$U_3'$}};
\node at (-2, 3.5) {\tiny{$U_4'$}};
\node at (-2, 2.5) {\tiny{$U_5'$}};

\node at (-0.5, 4.5) {\tiny{$U_6'$}};
\node at (-0.5, 3.5) {\tiny{$U_7'$}};
\node at (-0.5, 2.5) {\tiny{$U_8'$}};

\node at (-6.5, -1.5) {$X_2$};
\node at (-4.5, -0.5) {\tiny{$U_{N-7}'$}};
\node at (-3.5, 0.5) {\tiny{$U_{N-6}'$}};

\node at (-2, -1.5) {\tiny{$U_{\hspace{-1.5pt}N-10}'$}};
\node at (-2, -0.5) {\tiny{$U_{N-9}'$}};
\node at (-2, 0.5) {\tiny{$U_{N-8}'$}};

\node at (-0.5, -1.5) {\tiny{$U_{\hspace{-1.5pt}N-13}'$}};
\node at (-0.5, -0.5) {\tiny{$U_{\hspace{-1.5pt}N-12}'$}};
\node at (-0.5, 0.5) {\tiny{$U_{\hspace{-1.5pt}N-11}'$}};

\node at (-6.7, 5.2) {$V_1\setminus W$};
\node at (0, 5.2) {$V_2\cup W$};
\node at (6.5, 5.2) {$V_3$};

\node at (8.3, 4.5) {$A_1$};
\node at (8.3, -1.5) {$A_2$};

\end{tikzpicture}
    \caption{The structure of the matchings found in Lemma~\ref{lm:paths} as part of the embedding of the paths $P_1,\ldots,P_{\alpha}$. We abbreviate $\hat N=(N-10)/6$ and hide some of the names of the sets for stylistic and spatial reasons.
    The sets $V_1\setminus W$, $V_2\cup W$ and $V_3$ are depicted in red, while the smaller sets $A_1,A_2$ and $Z_0,\ldots,Z_{2\hat N+1}$ appear in blue.
    The solid black segments indicate a perfect matching between two sets, which may be among the larger sets $Z_i,A_i$ or the smaller ones $U_i'$. 
    The dashed segments represent the linking system: note that the paths there have equal length (unlike the dashed lines in the schema).}
    \label{fig:matchings}
\end{figure}

\subsection{\texorpdfstring{Universality for $\cT_2$ and $\cT_3$: many-to-one matchings and rollback}{Universality for T2: many-to-one matchings and rollback}}

The proof of $\cT_3$-universality of $G(n,C\ln n/n)$ is a toy version of the one for $\cT_2$-universality; we present the argument for $\cT_3$ first and then explain the (quite a few) upgrades required for $\cT_2$. Here, our embedding technique shares certain similarities with Case A from the work of Montgomery~\cite{Mon19}. 

Fix a tree $T$ in $\cT_3$ and recall the host graph $G\sim G(n,C\ln n/n)$. We first divide $T$ into three subtrees $T_1',T_3,T_4$ such that each of the following holds:
\begin{itemize}
    \item each of the pairs $T_1',T_3$ and $T_3,T_4$ meet at a single vertex,
    \item each of the trees $T_1',T_3,T_4$ have linear size, and
    \item $T_1'$ contains most of the leaves of $T$ (of which there are $\ge \exp(\Delta^5)\alpha$).
\end{itemize}
Out of these leaves of $T_1'$, one can extract a subset $L$ of $\ell = \Omega(\Delta\alpha)$ leaves in $T'_1$ which are pairwise far, that is, at least a large constant distance away from each other and from the root; we say that these vertices are \emph{well separated}. 
Denote by $P$ the parents of $L$ and note that $|L|=|P|$ due to the separation property.
The first key idea here (already appearing as part of Case A in~\cite{Mon19}) is to prepare two disjoint sets, $U_L$ and $U_P$ inside $[n]$, which we refer to as matchmakers, with the following properties:
\begin{enumerate}
    \item $|U_L|=|U_P|=\ell/10$,
    \item the images of $L,P$ in the constructed embedding cover the sets $U_L,U_P$, respectively, and
    \item\label{prop:match} there is a set $\bar U\subseteq [n]\setminus (U_L\cup U_P)$ of size $o(\alpha)$ such that,
    for all disjoint vertex sets $S_L,S_P$ that are also disjoint from $\bar U\cup U_{L}\cup U_{P}$ and with sizes $|S_L|=|S_P|=9\ell/10$, $G[U_L\cup S_L,U_P\cup S_P]$ contains a perfect matching. 
\end{enumerate}
In fact, the set $\bar U$ here consists of vertices which have an atypically small neighbourhood (say half the expected size) in $U_L$ or in $U_P$.

Next, divide $V(G)\setminus U_L$ into sets $V_1,V_2,V_3$ with $|V_1\cap V_2| = 1$ and $(V_1\cup V_2)\cap V_3 = \varnothing$ and such that: 
\begin{itemize}
    \item $V_1\supseteq U_P$ has size $|V(T_1'\setminus L)|+\eps n$; this is where $T_1'\setminus L$ will be embedded so that $P$ covers $U_P$,
    \item $V_2\supseteq \bar U$ has size $|V(T_3)|+\eps n$; this is where $T_3$ will be embedded so that $\bar U$ is entirely covered.
\end{itemize}
We note a few important observations here. 
First, as $V_1$ and $V_2$ are slightly larger than the orders of the corresponding trees to be embedded there, remainders of size $\eps n$ will remain uncovered; these vertices are sent to $V_3$ for the embedding of $T_4$ later on. Second, the embeddings of $T_1'\setminus L$ and $T_3$ rely on the rollback method as used in~\cite{Mon19}: the key sufficient property to ensure here is expansion of small vertex sets in $V_1$ towards $V_1\setminus U_P$ (for the embedding of $T_1'\setminus L$) and of small vertex sets in $V_2$ towards $V_2\setminus \bar U$ (for the embedding of $T_3$).
Third, the choice of the sets $V_1,V_2,$ and $V_3$ only depends on the sizes of $T_1'$ and $T_3$, and thus our considerations are required to hold with probability $1-o(n^{-2})$ to accommodate the corresponding union bound.

After the above program is executed, it remains to embed $T_4$ in the union of $V_3$, the remainders of $V_1,V_2$ and the image of the only vertex in $T_3\cap T_4$ (serving as a root for $T_4$).  As this embedding happens away from a single root (the only vertex in $T_3\cap T_4$), the rollback method requires much less spare room here, roughly $\Delta \alpha\le 9\ell/10$.
Adding these buffer vertices to $U_L$ after the embedding of $T_4$ is completed allows to match this set to the image of $P$ thanks to the matchmaking property~\eqref{prop:match} of $U_L$ and $U_P$, and concludes the embedding.
We note that the dependence of $C$ on $\Delta$ in this regime was spared because $|L|=\Omega(\Delta \alpha)$; as $\cT_1$ contains all trees with less than $\alpha$ leaves (by Lemma~\ref{lem:separate}), our attention turns to trees having between $\alpha$ and $\exp(\Delta^5)\alpha$ leaves covered by $\cT_2$.

\vspace{0.5em}

Showing $\cT_2$-universality is the most technical part of this work. To do this, we aim to construct $\imax+1 \approx \Delta^3$ disjoint layers of vertices $L_1,\ldots,L_{\imax+1}\subseteq V(T_1')$ via a leaf-cutting procedure (where leaves lie in $L_1$) so that:
\begin{itemize}
    \item every layer has size $\Omega(\alpha)$,
    \item for every $i\in [\Gamma]$, every vertex in $L_i$ has a single neighbour in $L_{i+1}\cup\ldots\cup L_{\imax+1}$, 
    \item the vertices in $L_{\imax+1}$ are well separated in $T_1'\setminus (L_1\cup\ldots\cup L_{\imax})$.
\end{itemize}
The purpose of this construction is to select disjoint sets $U_{L_1},\ldots,U_{L_{\imax}}\subseteq [n]$ to be covered by the images of $L_1,\ldots,L_{\imax}$, respectively, such that each set can absorb a comparable number of vertices to $U_L$ from the previous construction, thus ensuring that the remainder after the rollback embedding of $T_4$ can be fully absorbed.

This strategy comes with a number of complications. First, in order to imitate the argument for $\cT_3$, we need that each of the layers $L_1,\ldots,L_{\imax+1}$ has order $\Omega(\alpha)$.
Unfortunately, the leaf cutting procedure does not ensure that this is the case for every tree in $\cT_2$: in fact, the number of leaves may significantly shrink even by one step of leaf-cutting. 
The key observation here is that, if such a shrinking takes place before $\imax$ rounds of leaf-cutting, then we end up with a tree with few leaves and many long bare paths, like the ones in $\cT_1$ (but with fewer than $n$ vertices).
This naturally splits our proof into two cases.

\vspace{0.2em}

\paragraph{Case $1$.} Suppose that the leaf-cutting procedure defined all $\Gamma+1$ layers successfully. 
Then, a natural avenue would be to follow the strategy for $\cT_3$-universality where $L_{\imax+1}$ plays the role of $P$. 
This approach encounters an important problem on the way: without tight control on the sizes of the layers, the union bound (which in the case of $\cT_3$ was only over $n^2$ choices of sizes of $V_1,V_2$) now has to be done over $n^{\imax+O(1)}$ choices of layer sizes. This is too much for our purposes: indeed, when $p=C\ln n/n$ for an absolute constant $C$, a single vertex may be isolated with probability $\approx n^{-C}$.
To resolve this problem, we `shift' some of the trees descending from layer $L_{\imax+1}$ as shown in Figure~\ref{fig:shift}.
By applying shifting to a carefully chosen number of trees, we manage to guarantee that, for each $i\in [\imax]$, the size of the $i$-th layer eventually takes one of a bounded predetermined number of values.

While the above idea resolves the problem with the enormous union bound, it creates a different issue: since we extend the embedding of  $L_1\cup \ldots\cup L_{\imax}$ one layer at a time without taking into account the structure of layers with smaller indices, 
 it may happen that all neighbours of some vertex $v\in U_{L_i}$ in $U_{L_{i+1}}$ are \emph{blocked}, meaning that these neighbours are occupied by leaves of $T$ shifted to layer $L_{i+1}$. The solution to this problem consists of splitting the tree $T_1'$ into two subtrees $T_1,T_2$ with a comparable number of vertices on layer $L_{\imax+1}$, and shifting only a few subtrees of $T_1$ and none of $T_2$.
More precisely, we prepare absorbing sets $(U_{1,L_i}, U_{2,L_i})_{i=1}^{\imax+1}$ with good expansion properties and require that the shifted vertices in the set $V(T_1)\cap L_i$ are embedded into $U_{1,L_i}$ for all $i$. 
This way, the second reservoir contains only unshifted vertices of $T_1$ or $T_2$, which,
in the lower layers, can serve as parents in the subsequent matching step.

\begin{figure}
\centering
\begin{minipage}{\textwidth}
\centering
\begin{tikzpicture}[scale=0.9,line cap=round,line join=round,x=1cm,y=1cm]
\clip(-13.6,-1.2) rectangle (13.32,5);
\draw [line width=0.8pt] (-6,4)-- (-7,3);
\draw [line width=0.8pt,color=red] (-6,4)-- (-6,3);
\draw [line width=0.8pt] (-6,4)-- (-5,3);
\draw [line width=0.8pt,color=red] (-6,3)-- (-5,2);
\draw [line width=0.8pt] (-7,3)-- (-7,2);
\draw [line width=0.8pt] (-7,3)-- (-6,2);
\draw [line width=0.8pt,color=red] (-5,2)-- (-5,1);
\draw [line width=0.8pt] (-7,2)-- (-7,1);
\draw [line width=0.8pt,color=red] (-5,1)-- (-6,0);
\draw [line width=0.8pt,color=red] (-6,0)-- (-6,-1);
\draw [line width=0.8pt] (-5,1)-- (-5,0);
\draw [line width=0.8pt] (-3,4)-- (-4,3);
\draw [line width=0.8pt,color=red] (-3,4)-- (-3,3);
\draw [line width=0.8pt,color=red] (-3,3)-- (-3,2);
\draw [line width=0.8pt] (-3,3)-- (-4,2);
\draw [line width=0.8pt] (-4,2)-- (-4,1);
\draw [line width=0.8pt,color=red] (-3,2)-- (-3,1);
\draw [line width=0.8pt] (-3,1)-- (-4,0);
\draw [line width=0.8pt,color=red] (-3,1)-- (-3,0);
\draw [line width=0.8pt] (-1,4)-- (-2,3);
\draw [line width=0.8pt] (-1,4)-- (-1,3);
\draw [line width=0.8pt] (-1,4)-- (0,3);
\draw [line width=0.8pt] (-1,3)-- (-1,2);
\draw [line width=0.8pt] (-1,3)-- (-2,2);
\draw [line width=0.8pt] (-1,2)-- (-1,1);
\draw [line width=0.8pt] (-1,2)-- (0,1);
\draw [line width=0.8pt] (2,4)-- (1,3);
\draw [line width=0.8pt] (2,4)-- (2,3);
\draw [line width=0.8pt] (2,3)-- (2,2);
\draw [line width=0.8pt] (3,4)-- (3,3);
\begin{scriptsize}
\draw [fill=black] (-6,4) circle (2.5pt);
\draw[color=black] (-6,4.3) node {\large{$v_1$}};
\draw [fill=black] (-7,3) circle (2.5pt);
\draw [fill=black] (-6,3) circle (2.5pt);
\draw [fill=black] (-5,3) circle (2.5pt);
\draw [fill=white] (-5,2) circle (2.5pt);
\draw [fill=black] (-7,2) circle (2.5pt);
\draw [fill=black] (-6,2) circle (2.5pt);
\draw [fill=black] (-5,1) circle (2.5pt);
\draw [fill=black] (-7,1) circle (2.5pt);
\draw [fill=black] (-6,0) circle (2.5pt);
\draw [fill=black] (-6,-1) circle (2.5pt);
\draw (-5.72,-0.5) node {\large{$F_1$}};
\draw [fill=black] (-5,0) circle (2.5pt);
\draw [fill=black] (-3,4) circle (2.5pt);
\draw[color=black] (-3,4.3) node {\large{$v_2$}};
\draw [fill=black] (-4,3) circle (2.5pt);
\draw [fill=white] (-3,3) circle (2.5pt);
\draw [fill=black] (-3,2) circle (2.5pt);
\draw [fill=black] (-4,2) circle (2.5pt);
\draw [fill=black] (-4,1) circle (2.5pt);
\draw [fill=black] (-3,1) circle (2.5pt);
\draw [fill=black] (-4,0) circle (2.5pt);
\draw [fill=black] (-3,0) circle (2.5pt);
\draw (-2.68,0.5) node {\large{$F_2$}};
\draw [fill=black] (-1,4) circle (2.5pt);
\draw[color=black] (-1,4.3) node {\large{$v_3$}};
\draw [fill=black] (-2,3) circle (2.5pt);
\draw [fill=black] (-1,3) circle (2.5pt);
\draw [fill=black] (0,3) circle (2.5pt);
\draw [fill=black] (-1,2) circle (2.5pt);
\draw [fill=black] (-2,2) circle (2.5pt);
\draw [fill=black] (-1,1) circle (2.5pt);
\draw[color=black] (-1.28,1.5) node {\large{$F_3$}};
\draw [fill=black] (0,1) circle (2.5pt);
\draw [fill=black] (2,4) circle (2.5pt);
\draw[color=black] (2,4.3) node {\large{$v_4$}};
\draw [fill=black] (1,3) circle (2.5pt);
\draw [fill=black] (2,3) circle (2.5pt);
\draw [fill=black] (2,2) circle (2.5pt);
\draw[color=black] (1.7,2.5) node {\large{$F_4$}};
\draw [fill=black] (3,4) circle (2.5pt);
\draw[color=black] (3,4.3) node {\large{$v_5$}};
\draw [fill=black] (3,3) circle (2.5pt);
\draw[color=black] (3.32,3.5) node {\large{$F_5$}};
\end{scriptsize}
\end{tikzpicture}
\end{minipage}

\begin{minipage}{\textwidth}
\centering
\begin{tikzpicture}[scale=0.9,line cap=round,line join=round,x=1cm,y=1cm]
\clip(-13.6,0.9) rectangle (10.7,5);
\draw [line width=0.8pt] (-11,4)-- (-12,3);
\draw [line width=0.8pt,color=black] (-11,4)-- (-11,3);
\draw [line width=0.8pt] (-11,4)-- (-10,3);
\draw [line width=0.8pt] (-12,3)-- (-12,2);
\draw [line width=0.8pt] (-12,3)-- (-11,2);
\draw [line width=0.8pt,color=black] (-5,4)-- (-5,3);
\draw [line width=0.8pt] (-12,2)-- (-12,1);
\draw [line width=0.8pt,color=black] (-5,3)-- (-6,2);
\draw [line width=0.8pt,color=black] (-6,2)-- (-6,1);
\draw [line width=0.8pt] (-5,3)-- (-5,2);
\draw [line width=0.8pt,color=black] (-3,4)-- (-3,3);
\draw [line width=0.8pt] (-3,4)-- (-4,3);
\draw [line width=0.8pt] (-4,3)-- (-4,2);
\draw [line width=0.8pt,color=black] (-3,3)-- (-3,2);
\draw [line width=0.8pt] (-3,2)-- (-4,1);
\draw [line width=0.8pt,color=black] (-3,2)-- (-3,1);
\draw [line width=0.8pt] (-1,4)-- (-2,3);
\draw [line width=0.8pt,color=black] (-1,4)-- (-1,3);
\draw [line width=0.8pt] (-1,4)-- (0,3);
\draw [line width=0.8pt,color=black] (-1,3)-- (-1,2);
\draw [line width=0.8pt] (-1,3)-- (-2,2);
\draw [line width=0.8pt] (-1,2)-- (-1,1);
\draw [line width=0.8pt,color=black] (-1,2)-- (0,1);
\draw [line width=0.8pt] (2,4)-- (1,3);
\draw [line width=0.8pt] (2,4)-- (2,3);
\draw [line width=0.8pt] (2,3)-- (2,2);
\draw [line width=0.8pt] (3,4)-- (3,3);
\draw [line width=0.8pt,color=black] (-11,3)-- (-10,2);
\draw [line width=0.8pt,color=black] (-8,4)-- (-8,3);
\draw [line width=0.8pt] (-8,4)-- (-9,3);
\begin{scriptsize}
\draw [fill=black] (-11,4) circle (2.5pt);
\draw [fill=black] (-12,3) circle (2.5pt);
\draw [fill=black] (-11,3) circle (2.5pt);
\draw [fill=black] (-10,3) circle (2.5pt);
\draw [fill=white] (-5,4) circle (2.5pt);
\draw [fill=black] (-12,2) circle (2.5pt);
\draw [fill=black] (-11,2) circle (2.5pt);
\draw [fill=black] (-5,3) circle (2.5pt);
\draw [fill=black] (-12,1) circle (2.5pt);
\draw[color=black] (-11.25,1.45) node {\large{$F_1\setminus F_1'$}};
\draw [fill=black] (-6,2) circle (2.5pt);
\draw [fill=black] (-6,1) circle (2.5pt);
\draw[color=black] (-5.72,1.5) node {\large{$F_1'$}};
\draw [fill=black] (-5,2) circle (2.5pt);
\draw [fill=black] (-8,4) circle (2.5pt);
\draw [fill=white] (-3,4) circle (2.5pt);
\draw [fill=black] (-3,3) circle (2.5pt);
\draw [fill=black] (-4,3) circle (2.5pt);
\draw [fill=black] (-4,2) circle (2.5pt);
\draw [fill=black] (-3,2) circle (2.5pt);
\draw [fill=black] (-4,1) circle (2.5pt);
\draw [fill=black] (-3,1) circle (2.5pt);
\draw[color=black] (-2.68,1.5) node {\large{$F_2'$}};
\draw [fill=black] (-1,4) circle (2.5pt);
\draw [fill=black] (-2,3) circle (2.5pt);
\draw [fill=black] (-1,3) circle (2.5pt);
\draw [fill=black] (0,3) circle (2.5pt);
\draw [fill=black] (-1,2) circle (2.5pt);
\draw [fill=black] (-2,2) circle (2.5pt);
\draw [fill=black] (-1,1) circle (2.5pt);
\draw[color=black] (-1.32,1.5) node {\large{$F_3$}};
\draw [fill=black] (0,1) circle (2.5pt);
\draw [fill=black] (2,4) circle (2.5pt);
\draw [fill=black] (1,3) circle (2.5pt);
\draw [fill=black] (2,3) circle (2.5pt);
\draw [fill=black] (2,2) circle (2.5pt);
\draw[color=black] (1.7,2.5) node {\large{$F_4$}};
\draw [fill=black] (3,4) circle (2.5pt);
\draw [fill=black] (3,3) circle (2.5pt);
\draw[color=black] (3.32,3.5) node {\large{$F_5$}};
\draw [fill=white] (-10,2) circle (2.5pt);
\draw [fill=white] (-8,3) circle (2.5pt);
\draw[color=black] (-8.5,2.5) node {\large{$F_2\setminus F_2'$}};
\draw [fill=black] (-9,3) circle (2.5pt);
\end{scriptsize}
\end{tikzpicture}
\end{minipage}
\caption{Illustration of the shifting procedure in a forest $F_1,\ldots,F_5$: the original roots of the highest two trees, $F_1$ and $F_2$, are shifted down along the red paths, thus creating the new forest $F_1',F_2',F_3,F_4,F_5$.}
\label{fig:shift}
\end{figure}
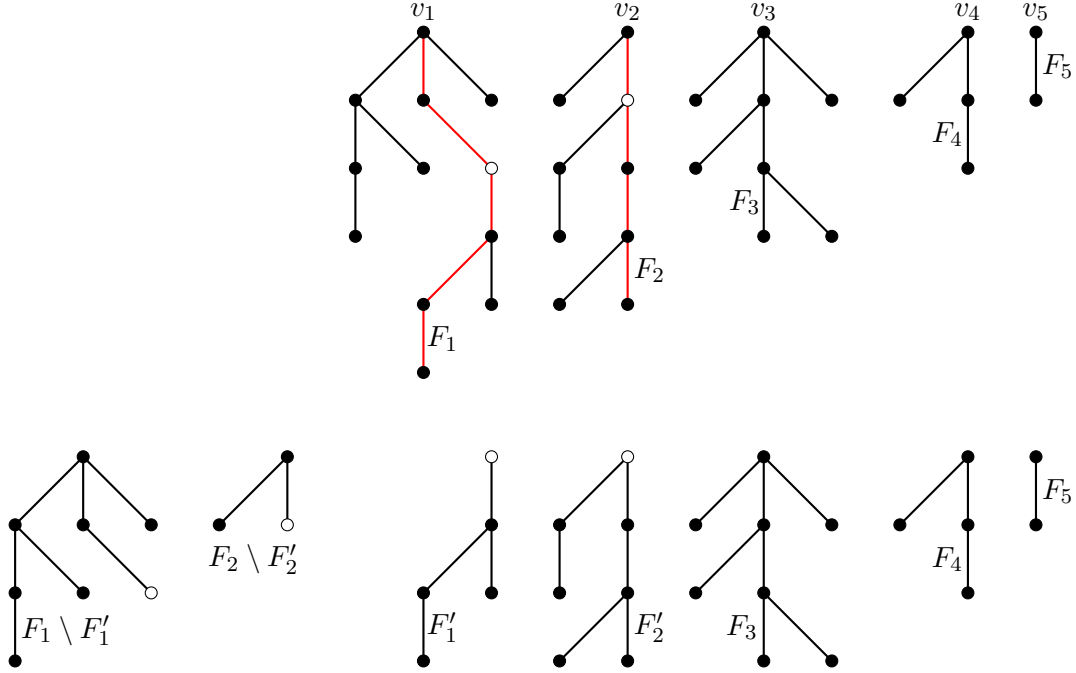

\vspace{0.2em}

\paragraph{Case $2$.} 
Suppose that the leaf-cutting procedure eventually produces layers $L_1,\ldots,L_r,L_{r+1}$ for some $r < \imax$ and a tree $T^- = T\setminus L$, where $L:=L_1\cup\ldots\cup L_r$.  Our embedding strategy in this case combines the embedding strategies from Section~\ref{sec:1.2.1} and Case 1, and employs a new idea.

In more detail, as a first step, we apply shifting to the layers $L_1,\ldots,L_r,L_{r+1}$, similarly to Case 1, in order to reduce the number of possibilities for the layer sizes.
At the same time, to prepare the random graph for the embedding, we separate sets $V_1,V_2$ and $V_3$ similarly to Section~\ref{sec:1.2.1}, and select subsets $U_1,\ldots,U_r$ within $V_1$ to host (after possibly cleaning few badly-expanding vertices) the layers $L_1,\ldots,L_r$.

Our strategy from Case~1 meets two major obstructions to be applied as such. First, the set of vertices of $G$ which the layering may absorb could be much smaller than the buffer required to embed the trees $T_1\setminus L,T_2\setminus L,T_3,T_4$ using the rollback method. 
  Instead, this buffer is absorbed within matchmaker sets similar to $(U_i')_{i=1}^N$ from Section~\ref{sec:1.2.1} (defined with respect to $T^-\subseteq T$ here).  Another problem is that vertices in $L_{r+1}$ may turn out to be on long bare paths in $T^-$ which are embedded within the linking system following Section~\ref{sec:1.2.1}, and our construction of linking systems does not allow us to identify the role of any of these vertices (as roots of small subtrees in the layering) in advance.
The new idea here is to pick (say) half of the long paths in $T^-$ which contain (at least) half of the parents of layer $L_{r+1}$, 
  and embed these together with the forest $F$ as defined in the first paragraph of Section~\ref{sec:1.2.1} into $V_1$.
In particular, unlike the case of $\cT_1$, the embedding of the augmented forest $F$ is done via the rollback strategy used in Case 1, which allows to plant some subset of roots from $L_{r+1}$ into a suitable matchmaking subset of $V_1$.

\section{Trees with many long bare paths: embedding via linking systems}\label{sec:long bare paths}

In this section we prove Theorem~\ref{thm:main_cycles} and Theorem~\ref{thm:univ_trees} for the family of trees $\cT_1$ with many long bare paths defined in the outline, Section~\ref{sc:overview}. 
Recall that $K$ is a large enough constant. Let 
$\psi=\psi(n)\in[1,\frac{\ln n}{4\ln\ln n}]$ be as in Section~\ref{sc:linking-reduction}. Set $$
\alpha:=\frac{n\psi\ln\ln n}{K^3 \ln n}.
$$
In particular, when $\psi=1$, then $\alpha$ is as defined in~\eqref{eq:alpha}, and it corresponds to the most interesting and challenging regime.   In this section, we use only $\alpha$ out of the $K\alpha$ long bare paths ensured by the definition of $\cT_1$; the fact that the respective family of paths is larger will be used in the next section. To this end, define $\cT_1^*$ to be the family of all \emph{graphs} $T$ on $n$ vertices with $\alpha$ vertex-disjoint bare paths $P_1,\ldots,P_{\alpha}$ of length $n/(K^2\alpha)$ such that $T\setminus E(P_1\cup\ldots\cup P_\alpha)$ is a forest with maximum degree at most $\Delta$.
Note that the graphs in $\cT_1^*$ can contain cycles, and $\cT_1$ is a subfamily of $\cT_1^*$.
Thus, the following theorem clearly implies $\cT_1$-universality.
\begin{theorem}\label{thm:univ_adapted}
    There is a constant $C>0$ such that, for every $\Delta\geq 2$, whp $G(n,C(\ln n)^{\psi}/n)$ is $\cT_1^*$-universal.
\end{theorem}

To deduce Theorem~\ref{thm:main_cycles}, take the constant in its girth assumption to be at least $2K^3$, where $K$ is the construction constant used here. Indeed, writing $L=n/(K^2\alpha)$, a cycle factor of girth at least $2K^3\ln n/(\psi\ln\ln n)=2n/\alpha$ has at most $\alpha/2$ cycles. It contains a collection of at least $n/(L+1)-\alpha/2>\alpha$ vertex-disjoint paths of length $L$, with at least one in each cycle. Selecting one path from each cycle and supplementing to exactly $\alpha$ paths shows that the factor belongs to $\cT_1^*$ with $\Delta=2$.

We begin with some notation. Fix $k\coloneqq \lceil 25\ln n/{(\psi\ln\ln n)}\rceil$ and fix $D=D(K)>1$ suitably large.
In addition, let $p_1 \coloneqq D{(\ln n)^{\psi}}/n$ and let $p_2:=\frac{p-p_1}{1-p_1}$ (so that $(1-p_1)(1-p_2)=1-p$). We let $G_1\sim G(n,p_1)$ and $G_2\sim G(n,p_2)$ be sampled independently; in particular, $G_1\cup G_2\sim G(n,p)$.
In what follows, we define a family of \emph{adapted graphs} and show that, for any adapted graph~$G$, whp $G\cup G_2$ is $\cT^*_1$-universal, and that whp $G_1$ is adapted. This will conclude the proof of Theorem~\ref{thm:univ_adapted}.

\begin{definition}[adapted graph]\label{def:adapted}
A graph $G$ on $[n]$ is said to be \emph{adapted} if there exists a partition of $[n]=V_1\sqcup V_2\sqcup V_3$ where
\[|V_2|=(1/(2K^2)-1/K^3)n,\quad |V_3|=3(2k+1)\alpha,\]
 and there exist disjoint sets $A_1,A_2\subseteq V_3$ with $|A_1|=|A_2|=3\alpha$, such that each of the following holds:
\begin{enumerate}[(i)]
    \item\label{def:i} The induced subgraph $G[V_1]$ contains an isomorphic copy of every forest $F$ with at most $|V_1|-n/K^3$ vertices and maximum degree at most $\Delta +1$.
    \item\label{def:iii} The maximum degree of $G$ is at most $2p_1n$.
    \item\label{def:iv} For every vertex $v\in [n]$, $$
    |N(v)\cap V_2|\geq 0.9 |V_2|p_1\quad\text{ and }\quad |N(v)\cap V_3|\geq 0.9 |V_3|p_1.$$
    \item\label{def:v} 
    For every pair of disjoint sets $U_1,U_2\subset V(G)$ of size $|U_1|=|U_2|=:x\in [\alpha/2]$,  
    we have
    \[
        e(U_1,U_2)\leq (2/3)D \psi(\ln n)^{\psi-1}(\ln\ln n) x/K^3 = (2/3)p_1 \alpha x.
    \]
    \item \label{def:vi} The graph $G$ is $(\alpha/3)$-joined (see Definition~\ref{def:joined}).
\end{enumerate}
\end{definition}

Note that, by choosing $K$ suitably large, we may (and we do) assume that 
\begin{equation}\label{eq:V1}
|V_1| = n-(1/(2K^2)-1/K^3)n-3(2k+1)\alpha\ge (1-2/(3K^2))n.
\end{equation}

In the rest of the section, for clarity of presentation, we assume that $n/(2K^2\alpha)$ is an integer congruent to 4 modulo 6; omitting this condition requires a slight modification of the parameters and has no significant effect on our arguments.

The proof of Theorem~\ref{thm:univ_adapted} is based on the following lemma which says that, roughly speaking, every adapted graph contains chains of matchings as depicted in Figure~\ref{fig:matchings}.

\begin{lemma}
\label{lm:paths}
Fix an adapted graph $G$ with vertex partition $V_1$, $V_2$, $V_3$ as in Definition~\ref{def:adapted}.
Then, there are disjoint sets $A_1,A_2\subset V_3$ of size $3\alpha$ such that for every choice of three pairwise disjoint sets $X_1,X_2,W\subseteq V_1$ with $|W|=n/K^3$ and $|X_1|=|X_2|=\alpha$, there exist vertex-disjoint path systems $F_1,F_2,F_3\subseteq G$ where:
\begin{itemize}
\item $F_1$ has terminal sets $X_1$ and $A_1'\subset A_1$ \emph{(}with $|A_1'|=\alpha$\emph{)} and depth $\frac{n/K^2+4\alpha}{12\alpha}$,
\item $F_2$ has terminal sets $X_2$ and $A_2'\subset A_2$ \emph{(}with $|A_2'|=\alpha$\emph{)} and depth $\frac{n/K^2+4\alpha}{12\alpha}$,
\item $F_3$ has terminal sets $A_1\setminus A'_1$ and $A_2\setminus A'_2$ and depth $\frac{n/K^2+10\alpha}{6\alpha}$,  
\item the union of the sets of internal vertices of $F_1$, $F_2$, and $F_3$ is $W\cup V_2$. 
\end{itemize}
\end{lemma}

\begin{proof}[Proof of Lemma \ref{lm:paths}]
We construct a set $V_2^*\subset V_2$ with 
\begin{equation}\label{eq:V2*}
|V_2^*|=|V_2|-5n/K^3\qquad \text{and}\qquad |N(v)\cap V_2^*|\geq 0.8|V_2^*|p_1\qquad \text{for every $v\in V(G)$.}
\end{equation} 
To this end, take a subset of $V_2$ of size $|V_2|-5n/K^3$ uniformly at random and observe that, by \ref{def:iv} and Chernoff's bound for hypergeometric random variables (Lemma~\ref{lem:Chernoff}), we get that for a fixed vertex $v$, the number of its neighbours in the random subset is less than $0.8|V_2^*|p_1$ with probability $O(\e^{-2\ln n})=o(1/n)$, say. Then the union bound over all vertices $v$ implies the desired assertion.

Define $N=N(n)\coloneq 6+n/(2K^2\alpha)$ and $q_2 \coloneq 1/(N-6)$. 
Call a family of disjoint sets $U_1,\ldots,U_N$ a {\it Match-Maker System} (MMS) if the following properties hold. 
\begin{enumerate}[label={(\bfseries M\arabic*)}]
    \item \label{item:MMS1} $U_1\cup\ldots\cup U_{N-6}=V_2^*$,
    \item \label{item:MMS2} $U_{N-5}\cup \ldots\cup U_N\subseteq V_3$,
    \item \label{item:MMS3} for every $i\in [N]$, we have $\bigl||U_i| - q_2|V_2^*|\bigr| \leq q_2|V_2^*|/(\ln n)^{2}$, and
    \item \label{item:MMS4} for every $i\in [N]$ and every $v\in V(G)$, we have $|N(v)\cap U_i|\geq (1-1/K) 0.8 q_2 p_1|V_2^*|$.
\end{enumerate}

\begin{claim}
$G$ contains an MMS.
\end{claim}
\begin{proof}
Sample independent random variables $(\xi_v)_{v\in V_2^*\cup V_3}$ where
\begin{itemize}
\item for every $v\in V_2^*$,
$\xi_v$ has uniform distribution on $[N-6]$, and
\item for every $v\in V_3$, $\xi_v\in\{0,N-5,N-4,\ldots,N\}$ and
\[
\forall i\in \{N-5,\ldots,N\},\qquad \mathbb{P}(\xi_v = i) = \frac{q_2 |V_2^*|}{|V_3|}\eqqcolon q_3.
\]
\end{itemize} 
For every $i\in [N]$, we define the (random) set $U_i = \{v:\,\xi_v=i\}$ and show that these sets satisfy properties \ref{item:MMS1}--\ref{item:MMS4} with positive probability.
Properties \ref{item:MMS1} and \ref{item:MMS2} hold deterministically and property \ref{item:MMS3} is satisfied with probability $\rho \ge 1-\exp(-\Omega(n/(\ln n)^5))$ by a routine application of Chernoff's bound and a union bound over $i\in [N]$. 

We show that \ref{item:MMS4} holds with probability significantly larger than $1-\rho$, implying that \ref{item:MMS1}--\ref{item:MMS4} hold jointly with positive probability, as required.
Indeed, for every $v\in V(G)$, define 
\[
    \cB_v \coloneqq \left\{\exists i\in [N]\text{ such that }|N(v)\cap U_i|<\left(1-\frac{1}{K}\right)0.8 q_2p_1|V_2^*|\right\}.
\]
Observe that, for every $v\in V(G)$ and every $i\in [N]$, $|N(v)\cap U_i|$ has distribution $\mathrm{Bin}(|N(v)\cap V_2^*|,q_2)$, if $i\le N-6$, and distribution $\mathrm{Bin}(|N(v)\cap V_3|,q_3)$ otherwise.
Then, by~\ref{def:iv} and~\eqref{eq:V2*}, for every $v\in V(G)$,
\begin{align*}
\mathbb{P}(\mathcal{B}_v)\leq (N-6) \mathbb{P}\bigg(\mathrm{Bin}\left(0.8 p_1|V_2^*|,q_2\right)
&<\left(1-\frac{1}{K}\right) 0.8q_2p_1|V_2^*|\bigg)\\
&+6\mathbb{P}\bigg(\mathrm{Bin}\left(0.9p_1|V_3|,q_3\right)<\left(1-\frac{1}{K}\right) 0.8q_3p_1|V_3|\bigg),
\end{align*}
which is of order $o(\exp(-7(\ln n)^{\psi-1}\ln\ln n))$ provided that $D=D(K)$ and $K$ are suitably large.

Moreover, the events $\mathcal{B}_v$ and $\mathcal{B}_u$ are independent when the vertices $u,v$ are at distance at least 3 in $G$. 
Thus, property \ref{def:iii} implies that, for all $v\in V(G)$, the event $\mathcal{B}_v$ depends on at most $(2p_1n)^{2}$ other events $\mathcal{B}_u$. 
By the Lov\'{a}sz Local Lemma (Lemma~\ref{lem:LLL}) applied to $(\cB_v)_v$ with $x=x_v=\exp(-6(\ln n)^{\psi-1}\ln\ln n)$, where we note that, for all $v\in V(G)$,
\[\mathbb P(\cB_v)\le x (1-x)^{(2p_1n)^2},\]
the probability that $U_1,\ldots,U_N$ satisfy property \ref{item:MMS4} is at least $(1-x)^n = \omega(1-\rho)$, completing the proof.
\end{proof}

Next, fix an MMS $U_1,\ldots,U_N$. For each $i\in\{N-5,\ldots,N\}$, we extend $U_i$ arbitrarily to a set $U'_i\subseteq V_3$ such that $U'_{N-5},\ldots ,U_{N}'$ are pairwise disjoint and have size exactly $\alpha$.
Define 
$$
A_1\coloneqq U'_{N-5}\cup U'_{N-4}\cup U'_{N-3}\quad\text{ and }\quad A_2\coloneqq U'_{N-2}\cup U'_{N-1}\cup U'_{N}.
$$ 

Now, fix disjoint sets $X_1, X_2, W\subseteq  V_1$ with $|X_1|=|X_2|=\alpha$ and $|W|=n/K^3$. Partition $W\cup (V_2\setminus V_2^*)$ into $N-6$ sets $Y_1,\ldots,Y_{N-6}$ so that the set $U'_i:=U_i\cup Y_i$ is of size $\alpha$ for every $i\in[N-6]$. Note that such a partition exists since every set in the MMS has size strictly less than $\alpha$ and the set $V_2\cup W$ is of size $n/(2K^2) = (N-6)\alpha$.
Recalling Figure~\ref{fig:matchings}, define $\hat N \coloneq (N-10)/6$ and set
$$
Z_0\coloneqq X_1\cup U'_1\cup U'_2,\;\;  Z_i\coloneqq U'_{3i}\cup U'_{3i+1}\cup U'_{3i+2}\text{ for } 
i\in[2\hat N],\;\; \text{and}\;\;
Z_{2\hat N+1}\coloneqq X_2\cup U'_{N-7}\cup U'_{N-6}.
$$

\begin{claim}\label{cl:matchings}
$G$ contains a perfect matching between sets in each of the following pairs:
\begin{itemize}
    \item $(U_1',U_{N-7}')$ and $(U_2',U_{N-6}')$,
    \item $(A_1,Z_{\hat N})$ and $(A_2,Z_{\hat N+1})$,
    \item $(Z_{i-1},Z_i)$ for all $i\in [2\hat N+1]$.
\end{itemize}  
\end{claim}

\begin{proof}
We show that there is a perfect matching between each pair $(U'_i, U'_j)$ with distinct $i,j\in [N]$: note that this ensures each of the claimed matchings except between $(Z_0,Z_1)$ and $(Z_{2\hat N},Z_{2\hat N+1})$, whose proof we delay. To this end,
we verify Hall's condition for any fixed pair $(U'_i, U'_j)$ with distinct $i,j\in [N]$ and sets of size at most $\alpha/2$ (recall this is enough by Corollary~\ref{cor:Hall}).
Fix such $i,j$ and, for every set $R\subseteq U'_i$ of size at most $\alpha/2$, denote by $N_R$ the neighbourhood of $R$ in $U'_j$.  Assume for contradiction that the condition fails for such a set $R$. Since $U_1,\ldots,U_N$ is an MMS, every vertex in $R$ has at least $(1-1/K) 0.8 q_2 p_1 |V_2^*|$ neighbours in $U_j\subseteq U'_j$. 
In particular, $G[R,N_R]$ has at least $|R|\cdot (1-1/K) 0.8 q_2 p_1 |V_2^*|$ edges. 
However, by property \ref{def:v} and our assumption that $|N_R|<|R|\leq \alpha/2$, this number of edges is dominated by
\[\frac{2}{3} p_1\alpha |R| \ge |R|\cdot \bigg(1-\frac{1}{K}\bigg) 0.8 q_2 p_1 |V_2^*|\implies \frac{2}{3} \ge \bigg(1-\frac{1}{K}\bigg) \bigg(1-\frac{12}{K}\bigg)\cdot \frac{4}{5},\]
which fails provided that $K$ is large enough, and shows the promised statement for every pair $(U_i',U_j')$ with $i\neq j$.

To ensure the perfect matchings between $(Z_0,Z_1)$ and $(Z_{2\hat N},Z_{2\hat N+1})$, we show that there exists a perfect matching between $(Z_0\setminus U_2',Z_1\setminus U'_{5} )$ and $(Z_{2\hat N} \setminus U'_{N-8},Z_{2\hat N+1} \setminus U'_{N-6})$. 
To this end, we use property \ref{def:vi} and Lemma \ref{lem:match-easy} with $t=\alpha/3$.
Without loss of generality, we verify the conditions only for the pair $(Z_0\setminus U_2',Z_1\setminus U'_{5})$, as the other case is identical. 
Assume for contradiction that there is $R\subseteq Z_1\setminus U_5'$ with $|R|\leq \alpha/3$ and $|N_R|<|R|$, where $N_R = N(R) \cap (Z_0 \setminus U_2')$ (the case when $R\subseteq Z_0\setminus U_2'$ is similar). 
Then, by \ref{item:MMS4}, every vertex $v\in V(G)$ has at least $(1-1/K)0.8q_2p_1|V_2^*|$ neighbours in $U'_1$. 
In particular, $G[R,N_R]$ has at least $|R|\cdot (1-1/K)0.8q_2p_1|V_2^*|$ edges and, provided that $K$ is sufficiently large, at most 
\[\frac{2}{3} p_1\alpha |R| < |R|\cdot \bigg(1-\frac{1}{K}\bigg)0.8q_2p_1|V_2^*|\]
edges by property \ref{def:v}, again leading to a contradiction.
\end{proof}

We are ready to construct the required path systems. Denote by $F$ the union of all matchings exhibited in Claim~\ref{cl:matchings}, except the matching between $Z_{\hat{N}}$ and $Z_{\hat{N}+1}$. 
Let $F_1$ (resp.\ $F_2$) be the path system consisting of all connected components intersecting $X_1$ (resp.\ $X_2$), and $F_3 = F\setminus (F_1\cup F_2)$.
Recalling Figure~\ref{fig:matchings}, it is easy to check that the constructed path systems satisfy the requirements of the lemma.
\end{proof}

Next we show that whp $G_1$ is adapted for a suitably large constant $D$ in assumption~\ref{def:iii}.

\begin{lemma}\label{lem:random-graphs-adapted}
Fix a partition $(V_1,V_2,V_3)$ of $[n]$ as in Definition \ref{def:adapted}. There is $D=D(K)>0$ such that, for $p_1 = D(\ln n)^{\psi}/n$, whp $G_1\sim G(n,p_1)$ is adapted with respect to the partition $(V_1,V_2,V_3)$. 
\end{lemma}

\begin{proof}
    To prove the lemma, it suffices to show that whp $G_1$ satisfies each of the five properties in Definition \ref{def:adapted}.
    First, property \ref{def:i} holds whp due to Theorem \ref{thm:AKS07} when applied with degree bound $\Delta+1$ (see Remark \ref{rem:embed forest}). Second, both properties \ref{def:iii} and \ref{def:iv} hold whp by a simple union bound (and a standard use of Chernoff's bound). Third, the fact that $G_1$ is whp $(\alpha/3)$-joined, i.e.\ property \ref{def:vi}, also follows from a simple union bound: letting $t:=\alpha/3$, for large enough $D=D(K)$, the converse has probability at most
\begin{equation}\label{eq:UB3.11.1-new}
    \binom{n}{t}\binom{n-t}{t} (1-p_1)^{t^2}\le \bigg(\frac{\e n}{t}\bigg)^{2t} \exp(-p_1t^2) = \exp(t(2\ln(\e n/t) - p_1t)) = o(1).
    \end{equation}

    It remains to verify that property \ref{def:v} holds whp. We consider three cases. In each case, we will show that for every fixed $x$, the failure probability of \ref{def:v} is $o(1/n)$. This and the union bound over the at most $n$ choices for $x$ will then conclude the proof.
    For simplicity of the presentation, given $x\ge 1$, set 
    \[
        y=y(x)\coloneqq \frac{2D(\ln n)^{\psi-1}\psi\ln\ln n}{3K^3}x=\frac{2}{3}p_1\alpha x.
    \]

\vspace{1em}
\noindent
\textbf{Case 1:} Set $x_0=\ln n/(3\psi\ln \ln n)$ and assume $x\in [1,x_0]$. Then, the probability that there exist two disjoint sets of size $x$ with at least $2x+3 = o(p_1 \alpha x)$ edges between them is at most 
\[\binom{n}{x}^2 \binom{x^2}{2x+3} p_1^{2x+3}\le \bigg(\frac{\e n p_1}{x}\bigg)^{2x} x^{4x+6} p_1^{3}\le \bigg(\frac{\e n p_1}{x_0}\bigg)^{2x_0} x_0^{4x_0+6} p_1^{3} = o(1/n).\]

\vspace{1em}
\noindent
\textbf{Case 2:} assume that $x\in [x_0,\alpha/10]$. 
 The probability that there exist two disjoint sets of size $x$ with at least $y$ edges between them is at most
\begin{align*}
 \binom{n}{x}^2\mathbb{P}(\mathrm{Bin}(x^2,p_1)\geq y)&\leq
 \left(\frac{\e n}{x}\right)^{2x}\cdot  \left(\frac{\e x^{2}}{y}\right)^{y} p_1^y= 
 \left(\frac{\e n}{x}\right)^{2x}\cdot \left(\frac{3ex}{2\alpha}\right)^y \\
 &=
 \left(\left(\frac{\e n}{x}\right)^2\left(\frac{3ex}{2\alpha}\right)^{(2D/3K^3)(\ln n)^{\psi-1}\psi\ln\ln n}\right)^x\\
 &\leq
 \left(\left(\frac{3\e^2 K^3 \ln n}{2\psi\ln\ln n}\right)^2\left(\frac{3\e}{20}\right)^{(2D/3K^3)(\ln n)^{\psi-1}\psi\ln\ln n-2}\right)^{\ln n/(3\psi\ln\ln n)}\hspace{-1.2pt}=o(1/n),
\end{align*}
where the last inequality holds by the assumptions on the range of $x$, the inequality $D\geq  K^4$, and by assuming that $K$ is large. 

\vspace{1em}
\noindent
\textbf{Case 3:} assume that $x\in [\alpha/10,\alpha/2]$.
In this case, by Chernoff's bound (Lemma~\ref{lem:Chernoff}), we obtain
\begin{align*}
 \binom{n}{x}^2\mathbb{P}(\mathrm{Bin}(x^2,p_1)\geq y)&\leq \left(\frac{\e n}{x}\right)^{2x} \exp\left(-\frac{(y-x^2p_1)^2}{2x^2p_1+2(y-x^2p_1)/3}\right)\\
 &\leq
 \left(\frac{\e n}{x}\right)^{2x} \exp\left(-\frac{(1-3/4)^2}{2\cdot 3/4+2(1-3/4)/3}\cdot \frac{2D(\ln n)^{\psi-1}\psi\ln\ln n}{3K^3}x\right)\\
 &\leq
 \left(\left(\frac{\e n}{x}\right)^2 \e^{-0.02(D/K^3)(\ln n)^{\psi-1}\psi\ln\ln n}\right)^x\\
 &\leq
 \left(\left(\frac{10\e K^3 \ln n}{\psi\ln\ln n}\right)^2e^{-0.02(D/K^3)(\ln n)^{\psi-1}\psi\ln\ln n}\right)^{\alpha/10}=o(1/n),
\end{align*}
where the second inequality holds as the function under the exponent is equal to $-\frac{(1-t)^{2}}{2t+2(1-t)/3}y$, where $t=x^2p_1/y=\frac{3x}{2\alpha}\leq 3/4$, and as $t\mapsto -\frac{(1-t)^{2}}{2t+2(1-t)/3}$ is increasing for all $t\in (0,1)$; and the final asymptotic bound holds when $D\ge  K^4$ and $K$ is large. This completes the proof.
\end{proof}

Finally, we are ready to prove Theorem~ \ref{thm:univ_adapted}.

\begin{proof}[Proof of Theorem~\ref{thm:univ_adapted}.] Since $C$ is large, we may assume that $p_2\geq D(\ln n)^{\psi}/n$. Recall that $G_2\sim G(n,p_2)$ and, having the conclusion of Lemma~\ref{lem:random-graphs-adapted}, our aim is to derive that, for any adapted graph $G$, whp $G\cup G_2$ is $\mathcal{T}_1^*$-universal.

 Let $G$ be an adapted graph. It satisfies the conclusion of Lemma~\ref{lm:paths}.
 Fix sets $A_1$ and $A_2$ as given by the lemma, a graph $T\in\cT_1^*$, and pairwise vertex-disjoint bare paths $P_1,\ldots ,P_{\alpha}$ in $T$, each of length $\ell = \ell(n)$ where
\begin{equation}\label{def:ell}
\alpha(\ell
-1) = |V_2\cup V_3| + \frac{n}{K^3} = \frac{n}{2K^2} + 3(2k+1)\alpha\implies \ell = \frac{n}{2K^2\alpha} + 6k +4 \le K\frac{\ln n}{\psi\ln\ln n} = \frac{n}{K^2\alpha},
\end{equation}
such that $T\setminus(E(P_1\cup\ldots\cup P_{\alpha}))$ is a forest. Such paths are obtained by shortening the paths in the definition of $\cT_1^*$; this only restores pendant path segments to the remaining forest and increases the degree of each original endpoint by at most one.
Denote by $F$ the forest obtained from $T$ by removing the internal vertices (together with their incident edges) of each path in $\{P_1,\ldots ,P_{\alpha}\}$. Note that $F$ has maximum degree at most $\Delta+1$, that it may have isolated vertices and that, for every $i\in [\alpha]$, $|V(P_i)\cap V(F)|=2$.
Since $|V(F)|=|V_1|-n/K^3$, by property~\ref{def:i}, the graph $G[V_1]$ contains a copy $\Tilde{F}$ of $F$. 
Denote by $W$ the set of vertices in $V_1$ which do not belong to $\Tilde{F}$ and note that $|W|=n/K^3$. 
In addition, let $X\subseteq V_1\times V_1$ be a collection of pairs $(x,y)\in V(\Tilde{F})\times V(\Tilde{F})$ such that connecting each pair by a (vertex-disjoint) path of length $\ell$ yields an isomorphic copy of $T$.
As $P_1,\ldots,P_{\alpha}$ are pairwise vertex-disjoint, the projections $X_1\coloneqq \{x:(x,y)\in X\text{ for some }y\in V(\Tilde{F})\}$ and $X_2\coloneqq \{y:(x,y)\in X\text{ for some }x\in V(\Tilde{F})\}$ are disjoint and are both of size $\alpha$. 
By the assertion of Lemma \ref{lm:paths} applied to $X_1,X_2,$ and $W$, there are vertex-disjoint path systems $F_1,F_2$ and $F_3$ such that:
\begin{itemize}
\item $F_1$ is a path system of depth $\frac{n/K^2+4\alpha}{12\alpha}$ joining $X_1$ with some $A'_1\subseteq A_1$,
\item $F_2$ is a path system of depth $\frac{n/K^2+4\alpha}{12\alpha}$ joining $X_2$ with some $A'_2\subseteq A_2$,
\item $F_3$ is a path system of depth $\frac{n/K^2+10\alpha}{6\alpha}$ between $A_1\setminus A'_1$ and $A_2\setminus A'_2$,
\end{itemize}
and $V(F_1\cup F_2\cup F_3)=X_1\cup X_2\cup V_2\cup W\cup A_1\cup A_2$. 

For every pair of vertices $(x,y)\in X$, consider a triplet $t_{(x,y)}$ of pairs $\{(x_1,y_1),(x_2,y_2),(x_3,y_3)\}$ in $A_1\times A_2$ such that each of the following holds:
\begin{itemize}
    \item there is a path between $x$ and $x_1$ in $F_1$,
    \item there is a path between $y_1$ and $x_2$ and between $y_2$ and $x_3$ in $F_3$,
    \item there is a path between $y_3$ and $y$ in $F_2$,
    \item the pairs of vertices in the triplets $t_{(x,y)}$ for $(x,y)\in X$ partition $A_1\cup A_2$.
\end{itemize}
Finally, by taking $D>K^4$, Lemma~\ref{cor:linking-systems} implies that whp $G_2[V_3]$ contains an $(A_1,A_2;2k)$-linking system.
On this event, $G_2[V_3]$ contains paths connecting $(x_i,y_i)$ for every $(x,y)\in X$ and $i\in [3]$, thus allowing to extend the partial embedding of $T$ in $G$ to the paths $P_1,\ldots,P_{\alpha}$.
This completes the embedding of $T$ and finishes the proof.
\end{proof}

\section{Trees with many leaves: rollback and absorption}\label{sec:manyleaves}

We focus on proving that $G\sim G(n,p=C\ln n/n)$ is typically $\cT_2$-universal when $C$ is suitably large;~a rather simplified version of the argument for $\cT_3$-universality (reminiscent of Case~A in the embedding strategy of Montgomery~\cite{Mon19}) is presented at the end of the section.
Before dealing with the universality statement, we establish a structural decomposition of the trees $T\in \cT_2$. To this end, fix some $T\in \cT_2$.
To begin with, we state a simple tree-partitioning lemma appearing as \cite[Proposition~3.19]{Mon19}.

\begin{lemma}[Proposition 3.19 in~\cite{Mon19}]\label{lem:tree-part}
For every tree $T$ and every set of vertices $Q\subseteq V(T)$, there are trees $T_1,T_2$ intersecting in a single vertex such that $T=T_1\cup T_2$ and $\min\{|V(T_1)\cap Q|,|V(T_2)\cap Q|\} \ge |Q|/3$.
\end{lemma}

\begin{remark}\label{rem:split}
It is a simple (but useful) observation that Lemma~\ref{lem:tree-part} also holds for forests. Indeed, on the one hand, every forest can be turned into a tree $T$ by adding edges. On the other hand, every pair of trees $T_1$ and $T_2$ satisfying the conclusion of Lemma~\ref{lem:tree-part} applied with $T$, satisfies the analogous conclusion for every spanning subforest of $T$, with all isolated vertices retained.
\end{remark}

By applying Lemma~\ref{lem:tree-part} three times with $Q=V(T)$, we divide the tree $T$ into four trees $T_1,T_2,T_3,T_4$ such that, for every $i\in \{1,2,3\}$, each of the following statements holds:
\begin{itemize}
    \item $\min\{|V(T_i)|, |V(T_{i+1}\cup \ldots\cup T_4)|\}\ge |V(T_i\cup \ldots\cup T_4)|/3$, and
    \item $V(T_i)\cap V(T_{i+1})=\{r_{i,i+1}\}$ and, for $j\ge i+2$, the trees $T_i$ and $T_j$ share a vertex if and only if $r_{i,i+1}=\ldots=r_{j-1,j}$.
\end{itemize}
Throughout the section, the trees $T_1,T_3$ and $T_4$ are considered as rooted in the vertices $r_{1,2},r_{2,3},$ and $r_{3,4}$, respectively, while $T_2$ is considered as a tree with two roots: $r_{1,2}$ and $r_{2,3}$.

Next, we state a structural lemma for bounded-degree trees appearing as \cite[Lemma~3.14]{Mon19}. Recall Definition~\ref{def:separated}.

\begin{lemma}[Lemma 3.14 in~\cite{Mon19}]\label{lem:separate}
Fix positive integers $n,h,d$ satisfying $n\ge 60h$ and $h\ge 4d$. 
Suppose that an (unrooted) tree $T$ with $n$ vertices has at most $n/(5d)$ leaves. 
Then, $T$ contains either a $2d$-separated set of at least $n/(40h)$ leaves or a collection of $n/(40h)$ vertex-disjoint bare paths with length~$h$.
\end{lemma}

The above lemma holds for unrooted trees, but a version for trees rooted in one or more vertices follows as a simple corollary.

\begin{corollary}\label{cor:separate}
Consider $n,h,d,\Delta,s$ satisfying $n\ge 200h\Delta^{2d+1}s$ and $h\ge 4d$. 
Fix a tree $T$ with $n$ vertices, maximum degree at most $\Delta$, and at most $n/(5d)$ leaves and a vertex set $S\subseteq V(T)$ of size $s$. 
Then, in the subforest of $T$ spanned by the vertices at distance at least $2d+1$ from $S$, there is either a $2d$-separated set of at least $n/(50h)$ leaves or a collection of $n/(50h)$ vertex-disjoint bare paths with length $h$.
\end{corollary}
\begin{proof}
We consider the two outcomes of Lemma~\ref{lem:separate} separately. Suppose first that the unrooted version of $T$ contains a $2d$-separated set of at least $n/(40h)$ leaves. Out of these, there are at most 
\[s\cdot \sum_{i=0}^{2d} \Delta^i\le s \Delta^{2d+1}\le \frac{n}{200h}\]
leaves within distance $2d$ from $S$, so $T$ contains a $2d$-separated set of $n/(40h)-n/(200h)=n/(50h)$ leaves at distance at least $2d+1$ from $S$, which suffices to conclude.

Second, suppose that $T$ contains at least $n/(40h)$ vertex-disjoint bare paths with length $h$.
Out of these, at least $n/(40h) - n/(200h) = n/(50h)$ are at distance at least $2d+1$ from $S$, which suffices to conclude.
\end{proof}

Before proceeding, we briefly comment on the notation. Throughout this section, we continue to use the notation $\alpha$ from Section \ref{sec:long bare paths} while assuming that $\psi\equiv 1$, so that $\alpha =\frac{n\ln \ln n}{K^{3}\ln n}$.

We continue with the above-mentioned tree decomposition and define a leaf-cutting algorithm aiming to construct a layering of $T_1\cup T_2$, which will be central for our embedding strategy later on.
In the rooted tree $T_1$, the \emph{descending tree} of a vertex $v\in V(T_1)\setminus \{r_{1,2}\}$ is the tree $(T_1)_{v}$ spanned by the vertices whose path to $r_{1,2}$ contains $v$. 
Furthermore, for an independent set $S\subseteq V(T_1)\setminus \{r_{1,2}\}$, we write $(T_1)_{S}$  
 for the union of $(T_1)_{v}$ with $v\in S$.
The tree $(T_1)_{v}$ is seen as rooted at $v$ and, in turn, $(T_1)_{S}$ is seen as a rooted forest with set of roots $S$.   For the rooted tree $T_2$, the \emph{descending tree} of a vertex $v\in V(T_2)\setminus \{r_{1,2},r_{2,3}\}$ is the tree $(T_2)_{v}$ spanned by the vertices whose paths to each of $r_{1,2}$ and $r_{2,3}$ contain $v$; this definition is extended similarly to sets $S\subseteq V(T_2)\setminus \{r_{1,2},r_{2,3}\}$.
We emphasise that the set of leaves of a \emph{rooted} tree is obtained from the set of leaves of its \emph{unrooted} counterpart by discarding the roots.

Fix $\imax = \Delta^3$ and $\delta_1=\delta_1(\Delta)>0$ suitably small; the properties $\delta_1$ needs to satisfy will become clear in the sequel. The next algorithm is used to construct the promised layering of (parts of) $T_1$ and $T_2$.

\begin{algorithm}\label{algo:leaf-cut}
Fix a tree $T$ with one or more roots. Initiate $\mathbf{i}\leftarrow 1$ and $\hat T\leftarrow T$.
Define $D_{\mathbf{i}}$ as the set of leaves of $\hat T$ (according to the definition, $D_{\mathbf{i}}$ never contains roots of $\hat T$).
\begin{itemize}
    \item If $D_{\mathbf{i}}\cap T$ contains no 24-separated set of vertices of size $(10+\delta_1)\alpha$ and at distance at least $2$ from the roots $r_{1,2},r_{2,3}$ of $T$, reject this instance and return $\mathbf{i}, D_{0}:=\varnothing, D_{1},\ldots,D_{\mathbf{i}},\hat T$.
    \item If $D_{\mathbf{i}}\cap T$ contains a 24-separated set of vertices of size $(10+\delta_1)\alpha$ and at distance at least $2$ from the roots $r_{1,2},r_{2,3}$ of $T$, and: 
    \begin{itemize}
        \item if $\mathbf{i}\le \imax$, update $\hat T\leftarrow \hat T\setminus D_{\mathbf{i}}$, $\mathbf{i}\leftarrow \mathbf{i}+1$, define $D_{\mathbf{i}}$ as the set of leaves of $\hat T$ and reiterate;
        \item if $\mathbf{i}=\imax+1$, terminate the algorithm and return $D_{0}:=\varnothing, D_{1},\ldots,D_{\imax+1},\hat T$.
    \end{itemize}
\end{itemize}
\end{algorithm}

For each $j\in \{1,2\}$, apply Algorithm \ref{algo:leaf-cut} to the tree $T_j$ with output
\[\mathbf{i_j},D_{0}^{j}:=\varnothing, D_{1}^{j},\ldots,D_{\mathbf{i_j}}^{j},\hat T_j,\qquad \text{and further define}\qquad D^j_{\mathbf{i_j}+1}=\ldots=D^j_{\imax+1}=\varnothing.\]
For notational convenience, we also set $\hat{T}\coloneqq \hat{T}_1\cup \hat{T}_2$ and $D_i \coloneqq D_i^{1}\cup D_i^{2}$ for every $i\in [\imax+1]$.  We remark that, upon termination of Algorithm~\ref{algo:leaf-cut} applied as mentioned above, for every $j\in \{1,2\}$, the set of leaves of $\hat T\cap T_j$ is $D_{\mathbf{i_j}}\cap V(T_j)$. In addition, the descending tree $(T_j)_{v}$ of each vertex $v\in D_{\mathbf{i_j}}$ has height $\mathbf{i_j}-1$: indeed, every step of Algorithm~\ref{algo:leaf-cut} reduces the height of the descending tree of $v$ by exactly 1, ending up with 0 at step $\mathbf{i_j}-1$.

From this point, our proof splits into two parts depending on the values of $\mathbf{i_1}$ and $\mathbf{i_2}$. The first case is when $\mathbf{i_1},\mathbf{i_2} = \imax+1$, and the subfamily of trees corresponding to this case is denoted by $\cT_2'\subseteq \cT_2$. The complementary case concerning the subfamily $\cT_2'' \coloneqq \cT_2\setminus \cT_2'$ is treated differently. We begin with the former.

\subsection{\texorpdfstring{Case $1$: rollback and absorption}{Case 1: rollback and absorption}}\label{sec:Gamma+1}

In this subsection, we assume that $T\in \cT_2'$, which implies that $\mathbf{i_1}=\mathbf{i_2}=\imax+1$. 
Define $S$ to be a set of parents of $D_{\imax}$ with $S\cap V(T_1)=\{v_1,\ldots,v_{(10+\delta_1)\alpha}\}$ and $S\cap V(T_2)=\{v_{(10+\delta_1)\alpha+1},\ldots,v_{2(10+\delta_1)\alpha}\}$, and such that $S\cap V(T_1)$ is 22-separated in $T_1$ while $S\cap V(T_2)$ is 22-separated in $T_2$.\footnote{Observe that the parents of a 24-separated set are 22-separated. } 
We note that, throughout this selection and the later applications of the rollback technique, $22$-separation is important within each of $T_1$ and $T_2$ but not necessarily for the entire set $S$.

For every $j\in [2(10+\delta_1)\alpha]$ and $r=\lceil j/((10+\delta_1)\alpha)\rceil$, define $Z_j\subseteq (T_r)_{v_j}$ to be a path of length $\imax$ starting from $v_j$ and going through a neighbour of it inside $D_{\imax}$. 
Observe that $D_i\cap Z_j$ consists of a single vertex for each $(i,j)\in [\imax]\times [2(10+\delta_1)\alpha]$. We denote this vertex by $v_{i,j}$ and further extend the notation by setting $v_{\imax+1,j}=v_j$.

Next, define $T_{1,S}$ to be the forest obtained as union of $(T_1)_{v_{\imax,j}}\cup \{v_{\imax,j}v_{\imax+1,j}\}$ for $j\in [(10+\delta_1)\alpha]$, considered as trees rooted respectively at $(v_{\imax+1,j})$, and similarly $T_{2,S}$ to be the forest obtained as union of $(T_2)_{v_{\imax,j}}\cup \{v_{\imax,j}v_{\imax+1,j}\}$ for $j\in [(10+\delta_1)\alpha+1,(20+2\delta_1)\alpha]$, considered as trees rooted respectively at $(v_{\imax+1,j})$.
The next algorithm uses $D_1,\ldots,D_{\imax+1}$ to construct an alternative layering of $T_{1,S}\cup T_{2,S}$
with several desirable properties (listed in Lemma~\ref{lem:layers}) used in the final stage of our embedding in Section~\ref{subsec:embedL}.
We give a description of the first step of the algorithm and comment briefly on the next steps.  
Initially, set $S_0\coloneq S\cap V(T_1)$ and $S_{i}:= \varnothing$ for every $i\in [\imax+1]$.  
We start with the forest $T_{1,S}\cup T_{2,S}$ layered by $D_1,\ldots,D_{\imax+1}$ and a threshold value $\lambda_1$ which is slightly smaller than $|D_1\cap (T_{1,S}\cup T_{2,S})|$.
One step of the algorithmic procedure is going to shift vertices of the tree with root $v_{\imax+1,j}$ for some $v_j\in S_0$ to the next layer with a larger index: 
 the set of vertices $v_j$ for which this happens is going to be the final value of the set $S_1$.
Every such shift is accompanied by pruning of $(T_1)_{v_{\imax+1,j}}$. 
As our goal is to preserve the size of layer $\imax+1$, at one step, the tree $(T_1)_{v_{\imax+1,j}}$ is replaced by the tree $(T_1)_{v_{\imax,j}}$, with its root $v_{\imax,j}$ embedded on layer $\imax+1$ and the remaining vertices spread across layers $2,\ldots,\imax$.
Note that every step of the algorithm decreases the size of the first layer (and also modifies the size of every layer) by a number in $[\Delta^{\imax}]$ and thus, after a small number of steps (bounded from above by $2\delta_1\alpha$), the threshold $\lambda_1$ for the size of the first layer is reached.
When the size of the first layer enters the interval $[\lambda_1,\lambda_1+\Delta^{\imax}]$, we define the threshold $\lambda_2$ for the next layer and proceed further.
Note that, during the processing of layer $i$, we choose the roots of the trees to be pruned and shifted up from a set $S_i\subseteq S_{i-1}\subseteq \ldots \subseteq S_0$: this ensures that the roots $S_0\setminus S_{i-1}$   
of the trees intersecting layers $1,\ldots,i-1$ will not be used again, and therefore the sizes of these layers remain fixed (and close to their respective thresholds).

\begin{algorithm}\label{algo:layers}
Fix $T\in \cT_2'$, initiate $\mathbf{i}\leftarrow 1$ and, 
for every integer $i\in [\imax+1]$, initiate $S_i\leftarrow \varnothing$. 
Furthermore, initiate $L_i\leftarrow V(T_{1,S}\cup T_{2,S})\cap D_i$ for all $i\in [\imax]$ and $L_{\imax+1}\leftarrow V(T_{1,S}\cup T_{2,S})\cap S$.
Further, define $\lambda_1 = \max\{(20\alpha + (\delta_1\alpha) \mathbb N_0)\cap \{1,\ldots,|L_1|-\delta_1\alpha\}\}$.
\begin{itemize}
\item If $\mathbf{i} \le \imax+1$ and $|L_{\mathbf{i}}|\le \lambda_{\mathbf{i}}+\Delta^{\imax}$, update $\mathbf{i}\leftarrow \mathbf{i}+1$ and, if $\mathbf{i}\le \imax+1$, set 
$$
\lambda_{\mathbf{i}} = \max\{(20\alpha + (\delta_1^{\mathbf{i}}\alpha) \mathbb N_0)\cap \{1,\ldots,|L_{\mathbf{i}}|-\delta_1^{\mathbf{i}}\alpha\}\}
$$
(where $\delta_1^{\mathbf{i}}$ denotes $\delta_1$ to the power of $\mathbf{i}$).

\item If $\mathbf{i} \le \imax+1$ and $|L_{\mathbf{i}}| > \lambda_{\mathbf{i}}+\Delta^{\imax}$, select an arbitrary $v_j\in S_{\mathbf{i}-1}\setminus S_{\mathbf{i}}$ 
and, for every $i\in [\mathbf{i},\imax+1]$, update 
\[L_i\leftarrow (L_i\setminus V((T_{1,S})_{v_j}))\cup (V((T_{1,S})_{v_{\imax+1-\mathbf{i},j}})\cap D_{i-\mathbf{i}}),\qquad S_{\mathbf{i}}\leftarrow S_{\mathbf{i}}\cup \{v_j\},\]
unless $\Gamma+1-\mathbf{i}=0$ (in which case $v_{0,j}$ is not defined); in this case, update $L_i\leftarrow (L_i\setminus V((T_{1,S})_{v_j}))$ and $S_{\mathbf{i}}\leftarrow S_{\mathbf{i}}\cup \{v_j\}$.
We then reiterate.
Recall that $D_0=\varnothing$ and thus layer $L_{\mathbf{i}}$ loses some vertices but receives no new ones.

\item Finally, if $\mathbf{i}=\imax+2$, 
terminate the algorithm and return the layers $L_1,\ldots,L_{\imax+1}$.
\end{itemize}
\end{algorithm}

We note that the gradual refining of the step size of the set $20\alpha+(\delta_1^i \alpha)\mathbb N_0$ in the definition of $\lambda_i$ is done to guarantee the termination of the algorithm. Indeed, when $\delta_1=\delta_1(\Delta)$ is suitably small, the number of trees shifted when $\mathbf{i}=i$ is substantially larger than $\delta_1^{i+1}\alpha$. Thus, in the next round $\mathbf{i}=i+1$ of the algorithm,  $S_i\setminus S_{i+1}$ contains many elements at every shifting step. This is explained in more detail in the proof of the next lemma.  For every $i\in [2,\imax+1]$, we define $L_i^*$ to be the set of all vertices on layer $L_i$ whose descending tree has all its leaves in $L_2\cup \ldots\cup L_{i-1}$ (i.e., $L_i^*$ consists of all vertices on layer $L_i$ that were shifted at least once), and we set $L_1^*=\varnothing$ for convenience.

\begin{lemma}\label{lem:layers}
Fix $\delta_1=\delta_1(\Delta)$ suitably small and set $\delta_2:= \Delta^{\imax+2} \delta_1$.
Then, Algorithm~\ref{algo:layers} terminates and each of the following properties holds:
\begin{enumerate}[(i)]
    \item\label{item:leaf-shifting1} for every $i\in [\imax+1]$, there is an integer $j\in [0,21(\Delta/\delta_1)^{\imax+1}]$ with $|L_i|-20\alpha\in [j\delta_1^i\alpha, j\delta_1^i\alpha+\Delta^{\imax}]$,
    \item \label{item:leaf-shifting2} $|L_2^*\cup L_3^*\cup\ldots\cup L_{\imax+1}^*|\le \delta_2 \alpha$,
    \item \label{item:leaf-shifting3} none of $L_1,\ldots,L_{\imax+1}$ contains any of $r_{1,2}$ and $r_{2,3}$,
    \item \label{item:leaf-shifting4} for every $i\in [\imax+1]$, all vertices in $L_i$ are leaves of the subforest of $T_1\cup T_2$ induced by $L_i\cup \ldots\cup L_{\imax+1}$,
    \item \label{item:leaf-shifting5} for each $r\in \{1,2\}$, the set $L_{\imax+1}\cap V(T_r)$ is $22$-separated in $T_r\setminus (L_1\cup\ldots\cup L_{\imax})$.
\end{enumerate}
\end{lemma}
\begin{proof}
At the beginning of round $i$, every component descending from a vertex in $S$ still meets $L_i$, so $|L_i|\ge |S|=(20+2\delta_1)\alpha$ and the set defining $\lambda_i$ is nonempty.
To show that Algorithm~\ref{algo:layers} terminates, it suffices to justify that, for every $i\in [\imax+1]$, $|L_i|$ enters the interval $[\lambda_i,\lambda_i+\Delta^{\imax}]$ before $S_i$ coincides with $S_{i-1}$.
To this end, note that, at every step when $\mathbf{i}=i$, $L_i$ loses between $1$ and $\Delta^{\imax}$ vertices and, therefore, the final size of $S_i$ is between $\delta_1^i\alpha/\Delta^\imax-1$ and $2\delta_1^i\alpha$. 
Hence, choosing $\delta_1<1/(4\Delta^\imax)$ and using that $n$ is sufficiently large gives
\[
2\delta_1^{i+1}\alpha
<\delta_1^i\alpha/\Delta^\imax-1
\qquad\text{for every }i\in[\imax].
\]
Thus, $S_{i+1}$ remains a strict subset of $S_i$ for every $i\in[\imax]$, and also $|S_1|\le 2\delta_1\alpha<|S_0|$.

To justify (i), it suffices to ensure that $j<21(\Delta/\delta_1)^{\imax+1}$, which holds since all layers $L_i$ with $i\in [\imax+1]$ at all stages of the algorithm have size at most 
\[|V(T_{1,S}\cup T_{2,S})|
\le |S|\cdot \Delta^{\imax+1} < 21\Delta^{\imax+1} \alpha \le 21(\Delta/\delta_1)^{\imax+1} \delta_1^i\alpha.\]
 Further, (ii) follows by noticing that the total number of vertices ever shifted is bounded from above by
\[|S_1|\cdot (1+\Delta+\ldots+\Delta^{\imax})\le 2\delta_1\alpha\cdot\Delta^{\imax+1}\le \delta_1\alpha\cdot \Delta^{\imax+2} = \delta_2\alpha.\]
Since the selected parents and their designated descending branches avoid the roots, and shifting only discards vertices, (iii) holds.
Finally, (iv) and (v) are satisfied by construction, which concludes the proof.
\end{proof}

Upon termination of Algorithm~\ref{algo:layers}, we set $L = L_1\cup\ldots\cup L_{\imax}$, $L^* = L_2^*\cup\ldots\cup L_{\imax}^*\subset V(T_1)$ and, for every $i\in [\imax]$, denote by $P_i\subseteq L_{i+1}\cup\ldots\cup L_{\imax+1}$ the set of the parents of vertices from $L_i$ in $T_1\cup T_2$. 
 Given $(x_1,\ldots,x_{\imax+1})$ in 
\begin{equation}
\label{eq:Lambda}
    \Lambda_\imax := \prod_{i=1}^{\imax+1} \bigg(\bigcup_{j=0}^{21(\Delta/\delta_1)^{\imax+1}} [(20+j\delta_1^i)\alpha,(20+j\delta_1^i)\alpha+\Delta^\imax]\bigg),
\end{equation}
we are going to show that whp all trees $T\in \cT_2'$ with $(|L_1|,\ldots,|L_{\imax+1}|)=(x_1,\ldots,x_{\imax+1})$ can be simultaneously embedded in $G$; the universality of $\cT_2'$ then follows from a union bound over the constant number of choices of the vector $(x_1,\ldots,x_{\imax+1})$. 

\subsubsection{Preparing the absorbers}\label{sec:absorb}
In our subsequent considerations, we fix $p_1$ such that $(1-p_1)^3=1-p$ and independently sample:
\begin{itemize}
    \item a random directed graph $D$ on $[n]=V(G)$ where all directed edges appear independently and with probability $p_1$, and 
    \item a random (non-directed) graph $G_1\sim G(n,p_1)$ on $[n]$.
\end{itemize}
By denoting $D'$ the non-directed graph obtained from $D$ by ignoring the orientations of its edges and identifying double edges, we obtain that $G_1\cup D'\sim G(n,p)$. 
For a set $U\subseteq [n]$, we denote by $N^-_D(U)$ the set of vertices $v\in [n]\setminus U$ such that there exists at least one directed edge in $D$ from $U$ to $v$; we call such vertices \emph{out-neighbours} of $U$. 
We also denote the number of these edges by $e_D(U,[n]\setminus U)$.

Further, fix a vector in $\Lambda_{\imax}$ indicating the layer sizes, and define the vertex set $L':=[|L|-\Delta \alpha]\subset [n]$. 
 Then, arbitrarily partition $L'$ into subsets $(L_i')_{i=1}^{\imax}$ of sizes $(|L_i|-\Delta^{-2}\alpha)_{i=1}^{\imax}$, respectively. 
By default, for subsets of $[n]$ responsible for embedding a set $X\subseteq V(T)$, we often use the same letter decorated by a single or double prime, like $X'$ and $X''$.
Define $\beta=\beta(n):=n/(\ln n)^2$, denote by $L_{1,i}'$ the set of the first $10\alpha$ vertices in $L_i'$, and set $L_{2,i}'=L_i'\setminus L_{1,i}'$.
The next result ensures a key property of the partition $(L_i')_{i=1}^{\imax}$.

\begin{proposition}\label{prop:props L'}
One can find (random) subsets $(L_{1,i}'',L_{2,i}'')_{i=1}^{\imax}$ of $(L_{1,i}',L_{2,i}')_{i=1}^{\imax}$, respectively, such that, by setting $L_i'' \coloneqq L_{1,i}''\cup L_{2,i}''$ and $L'' \coloneqq L_1''\cup\ldots\cup L_{\imax}''$,
the following two statements hold:
\begin{enumerate}[(i)]
    \item\label{item:props L'1} Whp $|L''|\ge |L'|-\beta$ and, for every $r\in \{1,2\}$ and distinct $i,j\in [\imax]$,  
    every vertex $v\in L_i''$ satisfies $e_D(v,L_{r,j}'')\ge |L_{r,j}'| p_1/2$. 
    \item\label{item:props L'2}
    Moreover, for every $r\in\{1,2\}$, every distinct $i,j\in[\Gamma]$, and every vertex $v\in L''_i$,
conditionally on $e_D(v,L_{r,j}'')$, the set of neighbours of $v$ in $L_{r,j}''$ is distributed uniformly among the subsets of $L_{r,j}''$ of size $e_D(v,L_{r,j}'')$ and independently for different choices of $i,j,r,v$.
\end{enumerate}
\end{proposition}

\begin{proof}
Our proof is algorithmic. 
Starting with $W=\varnothing$, iteratively and as long as possible, find $i\neq j$, $r\in \{1,2\}$ and $v\in L'_i\setminus W$  with less than $|L_{r,j}'| p_1/2$ out-neighbours in $L_{r,j}'\setminus W$ and update $W\leftarrow W\cup \{v\}$; note that the same vertex $v$ can be tested several times throughout the algorithm.
We also remark that all choices of vertices in this vertex-deletion procedure are made according to a fixed deterministic order.
 
Then, since $\min_{r,j} |L_{r,j}'| \ge 9.5\alpha = \omega(\beta)$ and by Chernoff's inequality, we get that the probability that, at some point in the algorithmic construction, $W$ reaches size $\beta$ is at most
\begin{align*}
&\mathbb P(\exists R\subseteq L': |R|=\beta\text{ and }\forall v\in R,\exists r\in \{1,2\},\exists j\in [\imax], |N^-_D(v)\cap (L_{r,j}'\setminus R)| < |L_{r,j}'| p_1/2)\\
&\hspace{2em}\le \binom{n}{\beta} \bigg(
\sum_{j=1}^{\imax} 2\max_{r\in \{1,2\}} \mathbb P(\mathrm{Bin}(|L_{r,j}'|-\beta,p_1)\le |L_{r,j}'| p_1/2)\bigg)^\beta\\
&\hspace{2em}\le \bigg(\frac{\e n}{\beta}\bigg)^\beta \bigg(2\imax\cdot \exp\bigg(-\frac{(\min_{r,j} |L_{r,j}'|-\beta)p_1}{9}\bigg)\bigg)^{\beta}
\le \bigg(\frac{\e n}{\beta}\bigg)^\beta\ \bigg(2\imax\cdot \e^{-\alpha p_1}\bigg)^\beta.
\end{align*}
In particular, the above expression is of order $o(1)$ for any $C\ge 7K^3$.
As a result, whp the construction of $W$ terminates before this set reaches size $\beta$, which ensures the validity of (i) by setting $L''_{r,i}:=L'_{r,i}\setminus W$. 

For part (ii), observe that the event that a vertex $v$ joins the set $W$ is measurable in terms of the sizes of the sets $|N^-_D(v)\cap (L_{r,j}'\setminus W)|$.
Thus, the construction of $W$ requires revealing the out-degrees of each vertex in $L'\setminus W$ towards sets $(L_{r,j}'\setminus W)_{r,j}$ but not the exact out-neighbourhoods, which remain uniformly distributed conditionally on their size. 
(More precisely, for a single vertex we can reveal sizes of several neighbourhoods since the set $W$ is changing dynamically, and the same vertex $v$ can be tested several times. However, in the end, if we remove $W$, then only the size of the desired neighbourhood is revealed.) 
\end{proof}

The sets $(L_i'')_{i=1}^{\imax}$ will be covered by parts of the sets $(L_i)_{i=1}^{\imax}\subseteq V(T_1\cup T_2)$ 
in a way that all vertices in $(L_i^*)_{i=2}^{\imax}$ are embedded in $(L_{1,i}'')_{i=2}^{\imax}$, respectively.
Next, we fix $L''$ and $(L_{r,i}'')_{r,i}$ ensured by Proposition~\ref{prop:props L'}: they will be useful a bit later in the argument. 
We direct our attention to the graph $G_1[[n]\setminus L']$. 
 Our next step can be formulated as follows: the vertices in $L'\setminus L''$ turned out to expand insufficiently in the determined layering, so we would like to use them for the embedding of the subtree $T_3$ instead, thus ensuring that they remain far from $T_1\cup T_2$. 
Before we proceed to the described embedding step, we note that some vertices in $[n]\setminus L'$ may also either expand insufficiently towards some of the layers (set later called $Q_L'$) or expand insufficiently towards the image of the parent set $P_{\imax}$ of the layer $L_{\imax}$ 
 (set later called $Q_P'$).
  The suboptimally expanding vertices in $Q_L'\cup Q_P'$ are thus going to be covered by the image of the tree $T_3$ together with $L'\setminus L''$. 
The next lemma shows that $Q_L'$ is typically quite small.

\begin{lemma}\label{lem:QL2}
Whp there are at most $\beta$ vertices $v\in [n]\setminus L'$ such that, for some $r\in \{1,2\}$ and $j\in [\imax]$, the vertex $v$ has less than $|L_{r,j}'| p_1/2$ out-neighbours in $D$ in the set $L_{r,j}''$. 
\end{lemma}
\begin{proof}
Condition on the sets $(L_{r,i}'')_{r,i}$ as well as the conclusion of Proposition~\ref{prop:props L'} (note that both are determined by the directed graph $D[L']$). 
Then, the probability that a vertex $v\in [n]\setminus L'$ has the desired property is at most
\[\sum_{j=1}^{\imax} 2\max_{r\in\{1,2\}} \mathbb P(\mathrm{Bin}(|L_{r,j}''|, p_1) < |L_{r,j}'| p_1/2)\le 2\imax\cdot \exp\bigg(-\frac{(\min_{r,j} |L_{r,j}'|-\beta)p_1}{9}\bigg)\le 2\imax\e^{-\alpha p_1},\]
where the first inequality uses Chernoff's bound. Since the last expression is of order $o((\ln n)^{-2})$ for $C\ge 7K^3$, the statement then follows from Markov's inequality.
\end{proof}

We denote the set described in Lemma~\ref{lem:QL2} by $Q_L'$. 

\begin{remark}\label{rem:QLprime}
Note that revealing the out-neighbourhood sizes of the vertices in $[n]\setminus L'$ within the sets $(L_{r,i}'')_{r,i}$ but not the exact neighbourhoods themselves is sufficient to determine $Q_L'$.
By combining Proposition~\ref{prop:props L'}~\ref{item:props L'2} and Lemma~\ref{lem:QL2}, we assume in what follows that the out-neighbourhood sizes of all vertices in the sets $(L_{r,i}'')_{r,i}$ are revealed, and only vertices $v$ in $(L'\setminus L'')\cup Q_L'$ are such that, for some $r\in \{1,2\}$ and $j\in [\imax]$, $v$ has less than $|L_{r,j}'| p_1/2$ out-neighbours in $D$ in the set $L_{r,j}''$.
Moreover, for all $j\in [\imax]$ and all remaining vertices $v$ outside $L_j''$,
conditionally on $e_D(v,L_{r,j}'')$, the set of neighbours of $v$ in $L_{r,j}''$ is distributed uniformly among the subsets of $L_{r,j}''$ of size $e_D(v,L_{r,j}'')$ and independently for different choices of $j,r,v$.
\end{remark}

Further observe that currently $G_1[[n]\setminus L']$ is distributed as an Erd\H{o}s-R\'enyi graph with parameters $n-|L'|=n-o(n)$ and $p_1$.
This graph will serve to embed the vertices outside of $L$ (with some room to spare; the remaining vertices to be fed into the sets $L''$ for the embedding of $L$ later in the proof).

The next few lemmas are dedicated to an algorithmic description of sets $P_{1,\imax}',P_{1,\imax}'',P_{2,\imax}',P_{2,\imax}''$ where some of the vertices in $L_{\imax+1}$ are going to be embedded.
Let $P_{1,\imax}',P_{2,\imax}'$ be disjoint subsets of $[n]\setminus (L'\cup Q_L')$ with size $\alpha$, and further denote by $Q_P'$ the subset of vertices of $[n]\setminus (L'\cup P_{1,\imax}'\cup P_{2,\imax}'\cup Q_L')$ 
 with at most $\alpha p_1/2$ out-neighbours in $P_{1,\imax}'$ or in $P_{2,\imax}'$ in $D$.
Note that $Q_P'$ is measurable in terms of the out-neighbourhood sizes $(|N^-_D(v)\cap P_{i,\imax}'|)_{i\in \{1,2\}}$ of the vertices $v\in [n]\setminus (L'\cup P_{1,\imax}'\cup P_{2,\imax}'\cup Q_L')$ and not the neighbourhoods themselves.

\begin{lemma}\label{lem:Q'}
Whp $|Q_P'|\le \beta$.
\end{lemma}
\begin{proof}
The construction of the sets $(L_{r,i}'')_{r,i}$, $P_{1,\imax}'$, $P_{2,\imax}'$, and $Q_L'$ is measurable in terms of edges outside $D[[n]\setminus L']$. 
Thus, conditionally on these sets, Chernoff's inequality shows that the probability that a fixed vertex $v\in [n]\setminus (L'\cup P_{1,\imax}'\cup P_{2,\imax}'\cup Q_L')$  
 has less than $\alpha p_1/2$ out-neighbours in $P_{1,\imax}'$ or in $P_{2,\imax}'$ is at most
\[2\mathbb P(\mathrm{Bin}(\alpha,p_1)\le \alpha p_1/2)\le
2\exp(-{\alpha p_1}/{8})\le (\ln n)^{-3},\]
where the last inequality holds for $C \ge 75 K^3$. 
An application of Markov's inequality ends the proof.
\end{proof}

Next, we construct subsets $P_{1,\imax}'', P_{2,\imax}''$ of size $\beta=o(\alpha)$ algorithmically
as follows. Initially, these sets are empty. 
For every $i\in \{1,2\}$, expose $\{|N_D^-(v)\cap P_{i,\imax}'|: v\in L_{\imax}''\}$  
 and arbitrarily order the vertices in $L_\Gamma''$ with less than $\alpha p_1/2$ out-neighbours in $P_{i,\imax}'$.
For each such vertex $v$, consecutively select $\alpha p_1/2$ out-neighbours of $v$ outside $P_{1,\imax}''\cup P_{2,\imax}''$ and in 
\[V' := [n]\setminus (L'\cup Q_L'\cup Q_P'\cup P_{1,\imax}'\cup P_{2,\imax}'),\] 
and add them to $P_{i,\imax}''$ (while keeping $P_{1,\imax}''$ and $P_{2,\imax}''$ disjoint). Finally, complete $P_{i,\imax}''$ to size $\beta$ by arbitrarily adding new vertices from $V'\setminus (P_{1,\imax}''\cup P_{2,\imax}'')$.
The next lemma shows that the construction of $P_{i,\imax}''$ presented above is feasible.

\begin{lemma}\label{lem:W}
For every $i\in \{1,2\}$, the construction of $P_{i,\imax}''$ succeeds whp.
\end{lemma}
\begin{proof} 
An application of Chernoff's inequality shows that the probability that a fixed vertex $v\in L_{\Gamma}''\subseteq L_{\Gamma}'$ 
  has less than $\alpha p_1/2$ out-neighbours in $P_{i,\imax}'$ is bounded from above by 
\[\mathbb P(\mathrm{Bin}(\alpha,p_1)\le \alpha p_1/2)\le
\exp(-{\alpha p_1}/{8})\le (\ln n)^{-3},\]
where the last inequality holds for $C \ge 75K^3$. 
In particular, by Markov's inequality, whp there are at most $\beta/(\ln n)^{1/2}$ vertices in $L_{\imax}''$ with less than $\alpha p_1/2$ out-neighbours in $P_{i,\imax}'$; we call this set $U'$ and assume the said upper bound on its size. 
We order these vertices arbitrarily and process them one by one.
Then, using that $|V'| = n-o(n)$, consecutive applications of Chernoff's inequality and a union bound imply that whp every vertex $v\in U'$ to be processed at the next step sends at least $np_1/2\ge \alpha p_1$ out-edges in $D$ towards $V'\setminus (P_{1,\imax}''\cup P_{2,\imax}'')$. 
Since $\alpha p_1\cdot |U'| = o(\beta)$, the construction of each of $P_{1,\imax}''$ and $P_{2,\imax}''$ is successful whp, as desired.
\end{proof}

\subsubsection{\texorpdfstring{Embedding $T\setminus L$ in $G_1[[n]\setminus L'']$: rolling back}{Embedding T1 and T2 by rolling back}}\label{sec:rollback}

Recalling~\eqref{eq:Lambda}, fix $(x_1,\ldots,x_{\Gamma+1})\in\Lambda_{\Gamma}$ and assume that $T\in\mathcal{T}'_2$ satisfies $(|L_1|,\ldots,|L_{\Gamma+1}|)=(x_1,\ldots,x_{\Gamma+1})$.
In this section, we obtain several typical properties of $G_1[[n]\setminus L']$. Their importance is two-fold: 
\begin{itemize}
    \item first, any graph $G_1$ 
     with these properties contains $T\setminus L$ as a subgraph, and
    \item second, these properties ensure that the vertices of $G_1$ where no vertex of $T\setminus L$ is embedded expand sufficiently towards the layers $(L_i'')_{i=1}^{\imax+1}$, thus allowing us to extend our embedding of $T\setminus L$ to an embedding of $T$.
\end{itemize}

In this section, we condition on the sets $(L_{r,i}'')_{r,i}$, $(P_{r,\imax}',P_{r,\imax}'')_r$, $Q'_L$ and $Q'_P$ constructed in the previous section. 
Recall $T_1,T_2,T_3,T_4$ from the beginning of Section~\ref{sec:manyleaves}, the set $L$ defined after the proof of Lemma~\ref{lem:layers}, and define $\hat T_1\coloneqq T_1\setminus L$, $\hat T_2\coloneqq T_2\setminus L$, and
\[F' := L'\cup Q_L'\cup Q_P'\cup P_{1,\imax}'\cup P_{1,\imax}''\cup P_{2,\imax}'\cup P_{2,\imax}''.\]
Next, denote by $V_1, V_2, V_3$ subsets of $[n]\setminus L''$ such that:
\begin{enumerate}[(i)]
\item\label{pt:i} $V_1$ and $V_2$ intersect in the first vertex $w_{1,2}$ outside $F'$, and $(V_1\cup V_2)\cap V_3 = \varnothing$, 
\item $P_{1,\imax}'\cup P_{1,\imax}''\subseteq V_1$, $P_{2,\imax}'\cup P_{2,\imax}''\subseteq V_2$, and $(L'\setminus L'')\cup Q_L'\cup Q_P'\subseteq V_3$,
\item $|V_1|=|V(\hat T_1)|+n/300\ge n/3$, $|V_2|=|V(\hat T_2)|+n/300\ge n/9$ and $|V_3|=|V(T_3)|+n/300\ge n/27$,
\item\label{pt:iv} $V_1\setminus (F'\cup \{w_{1,2}\})$ is the set of the first $|V_1|-|P_{1,\imax}'\cup P_{1,\imax}''|-1$ available vertices after $w_{1,2}$, $V_2\setminus (F'\cup \{w_{1,2}\})$ is the set of the next $|V_2|-|P_{2,\imax}'\cup P_{2,\imax}''|-1$ available vertices, and finally $V_3\setminus (F'\cup \{w_{1,2}\})$ is the set of the following $|V_3|-|(L'\setminus L'')\cup Q_L'\cup Q_P'|$ available vertices.
\end{enumerate}
We note that the choice of the vertex $w_{1,2}$ is fully determined by the choice of sizes $(x_1,\ldots,x_{\imax+1})$ of the layers and the algorithmic procedure from Section~\ref{sec:absorb}, which is independent of the choice of a tree.
On the other hand, the sizes of $\hat T_1$, $\hat T_2$ and $T_3$ are sufficient to determine the above construction but may vary. 
The following statements are therefore shown to hold with probability $1-o(n^{-3})$ to accommodate the union bound coming from the choice of sizes described above.

We now give our first application of the rollback technique stated in Proposition~\ref{embedparentsfinal}. 
 In this section, we set $d=d(n):=\ln n/\ln\ln n$ and $t=t(n):=4(\ln\ln n)/p_1$.

\begin{lemma}\label{cor:embedparentsfinal)3}
Whp the graph $G_1$ is $t$-joined. Moreover, for every choice of integers $n_1,n_2,n_3\in [n]$, with probability $1-o(n^{-3})$, for every choice of trees $\hat T_1,\hat T_2,T_3$ satisfying \textup{\ref{pt:i}--\ref{pt:iv}} and such that \[(|V(\hat T_1)|,|V(\hat T_2)|,|V(T_3)|)=(n_1,n_2,n_3),\]
each of the following exists: 
\begin{itemize}
    \item A $(d,t)$-extendable copy $W_1$ of $\hat T_1$ into $G_1[V_1]$ such that the set $P_{1,\imax}'\cup P_{1,\imax}''$ is covered by the image of the set $L_{\imax+1}\cap V(T_1)$ and the vertex $r_{1,2} = T_1\cap T_2$ is mapped to the vertex $w_{1,2}$.
    \item A $(d,t)$-extendable copy $W_2$ of $\hat T_2$ into $G_1[V_2]$ such that the set $P_{2,\imax}'\cup P_{2,\imax}''$ is covered by the image of the set $L_{\imax+1}\cap V(T_2)$ and the vertex $r_{1,2}$ is mapped to the vertex $w_{1,2}$. Denote by $w_{2,3}$ the copy of $r_{2,3}$ in $W_2$.
    \item A $(d,t)$-extendable copy $W_3$ of $T_3$ into $G_1[V_3\cup \{w_{2,3}\}]$ covering the set $(L'\setminus L'')\cup Q_L'\cup Q_P'$ and where the vertex $r_{2,3} = T_2\cap T_3$ is mapped to the vertex $w_{2,3}$.
\end{itemize}
\end{lemma}

\begin{proof}
The proof of the first two points is exactly the same up to the choice of an absolute constant $C$; we show the statement for $\hat T_1$ only.
The fact that $G_1$ is a $t$-joined graph follows from a simple union bound: indeed, the converse has probability at most 
\begin{equation}
\binom{n}{t}\binom{n-t}{t} (1-p_1)^{t^2}\le \bigg(\frac{\e n}{t}\bigg)^{2t} \exp(-p_1t^2) = \exp(t(2\ln(\e n/t) - p_1t)) = o(1).
\end{equation}
We note also that $|P_{1,\imax}'\cup P_{1,\imax}''|\ge \beta$ and next show that, with probability $1-o(n^{-3})$, the `empty' graph $I(P_{1,\imax}'\cup P_{1,\imax}''\cup \{w_{1
,2}\})$ is a $(d,t)$-extendable subgraph of $G_1[V_1]$ provided that $C\ge 4000$. 
 For this purpose, fix a set $X\subseteq V_1$ of size $i\in [2t]$ and note that the probability that a vertex outside $P_{1,\imax}'\cup P_{1,\imax}''\cup X\cup \{w_{1,2}\}$ is adjacent to $X$ is $1-(1-p_1)^i$.
Thus, the probability that there exists a set $X\subseteq V_1$ of size $i\in [2t]$ and at most $di$ neighbours outside $P_{1,\imax}'\cup P_{1,\imax}''\cup X\cup \{w_{1,2}\}$ is at most

\begin{equation}\label{eq:UB3.11.1+}
\sum_{i=1}^{2t} \binom{|V_1|}{i} \mathbb P(\mathrm{Bin}(|V_1|-i-2\alpha-1,1-(1-p_1)^i)\le di).    
\end{equation}
We consider two different regimes with respect to the latter sum. When $2p_1i\leq 1$, we have 
\begin{equation}\label{eq:binomial-estimate}
\begin{split}
1-(1-p_1)^i=1-\sum_{j=0}^i (-1)^j \binom{i}{j} p_1^j
&\ge p_1i - \sum_{j=2}^i \frac{(p_1i)^j}{j!}\\ 
&\ge p_1i\bigg(1-2\sum_{j=2}^{\infty} \frac{(1/2)^{j}}{j!}\bigg)=p_1i(1-2(\sqrt{\e}-3/2))\ge \frac{p_1i}{2}.
\end{split}
\end{equation}
By combining the latter inequality, the relations $i+2\alpha+1=o(|V_1|)$ and $|V_1|p_1 = \Theta(np_1) = \omega(d)$, and
 Chernoff's inequality, the $i$-th term in~\eqref{eq:UB3.11.1+} is at most
\begin{equation}\label{eq:3.11.1small}
\binom{|V_1|}{i} \mathbb P\bigg(\mathrm{Bin}\bigg(\frac{|V_1|}{2},\frac{p_1i}{2}\bigg)\le di\bigg)\le \bigg(\frac{\e |V_1|}{i}\bigg)^i \exp\bigg(-\frac{|V_1|p_1i}{9}\bigg)\le \exp\bigg(-\frac{|V_1|p_1i}{10.5}\bigg) = o(n^{-3}),    
\end{equation}
where we used twice that $|V_1|p_1 \ge 66\ln n$ since $C$ is large enough.
We turn our attention to the terms in~\eqref{eq:UB3.11.1+} corresponding to $2p_1i>1$.
In this case, by the inequality $1-\e^{-1/2}>3/10$, we have that $1-(1-p_1)^{i}\geq 3/10$. Noting that $i+2\alpha+1=o(|V_1|)$ and choosing $C$ large enough so that $2dt\le (3n/10)/40\le 3|V_1|/40$, each term in \eqref{eq:UB3.11.1+} with $2p_1i>1$ is bounded from above~by
\begin{equation}\label{eq:3.11.1large}
\bigg(\frac{\e |V_1|}{i}\bigg)^i \mathbb P\bigg(\mathrm{Bin}\bigg(\frac{|V_1|}{2},\frac{3}{10}\bigg)\le di\bigg)\le \bigg(\frac{\e |V_1|}{i}\bigg)^i \exp\bigg(-\frac{3|V_1|}{160}\bigg)\le \exp\bigg(-\frac{|V_1|}{60}\bigg) = o(n^{-6}),
\end{equation}
where the second inequality uses that $|V_1| = O(i \ln n)$ and $|V_1| = \omega(t \ln\ln n)$. Plugging \eqref{eq:3.11.1small} and \eqref{eq:3.11.1large} into \eqref{eq:UB3.11.1+} shows that \eqref{eq:UB3.11.1+} is of order $o(n^{-3})$, ensuring whp the $(d,t)$-extendability of $I(P_{1,\imax}'\cup P_{1,\imax}''\cup \{w_{1,2}\})$.

Next, the condition $|V(\hat{T_1})|\le |V_1|-|P_{1,\imax}'\cup P_{1,\imax}''|-10dt-\ln n$
holds when $C$ is large enough (implying $10dt\le n/400$), and  
the inequality $|L_{\imax+1}\cap V(\hat T_1)|\ge 10\alpha\geq 9|P'_{1,\imax}\cup P''_{1,\imax}|$ ensured by Lemma~\ref{lem:layers}~\ref{item:leaf-shifting1} and the separability condition of $L_{\imax+1}$ 
hold by Lemma~\ref{lem:layers}~\ref{item:leaf-shifting5} and the fact that only the final round can delete top-layer vertices, at most $2\delta_1^{\imax+1}\alpha$ of them, thus guaranteeing that Proposition~\ref{embedparentsfinal} applies and gives the desired conclusion.

It remains to show the last point. 
To apply Proposition~\ref{embedparentsfinal} for the embedding of $T_3$, we note that: 
\begin{itemize}
    \item $T_3$ contains a $22$-separated set of vertices $U$ of size $30\beta$ not containing $r_{2,3}$, which may be constructed by a straightforward greedy algorithm,
    \item by Proposition~\ref{prop:props L'}, Lemma~\ref{lem:QL2} and Lemma~\ref{lem:Q'}, $|(L'\setminus L'')\cup Q_L'\cup Q_P'|\le 3\beta$,  
    \item the set $Z=(L'\setminus L'')\cup Q_L'\cup Q_P'$ may have size less than $\beta$. To match the assumption on the size of the set from Proposition~\ref{embedparentsfinal}, if $|Z| < \beta$, add vertices from $V_3$ to $Z$ until it reaches size $\beta$.
\end{itemize}

Moreover, $|V(T_3)|\le |V_3\cup \{w_{2,3}\}|-|Z|-10dt-\ln n$ since $|Z|\le 3\beta=o(n)$ and $C$ is large enough that, for all sufficiently large $n$, $|Z|+10dt+\ln n\le n/300$.
 It remains to show that, with probability $1-o(n^{-3})$, the `empty' graph $I(Z\cup \{w_{2,3}\})$ on $Z\cup \{w_{2,3}\}$ is a $(d,t)$-extendable subgraph of $G_1[V_3 \cup \{w_{2,3}\}]$. 
 Fix any set $Y\subseteq (V_3 \cup \{w_{2,3}\})$ of size $|Y|\le 2t$. 
By noting that $|V_3\setminus (Y\cup Z\cup \{w_{2,3}\})|\ge |V(T_3)|-3\beta -2t-1\ge n/40$
 and repeating the analysis of the union bound in \eqref{eq:UB3.11.1+}--\eqref{eq:3.11.1large} for suitably large $C$ (with $V_1$ replaced by $V_3$), it follows that, for all $Y$ with $|Y|\le 2t$, $Y$ vertex-expands by a factor of $d$ towards $V_3\setminus (Z\cup \{w_{2,3}\})$ in $G_1$, and hence another application of Proposition~\ref{embedparentsfinal} completes the proof.
\end{proof}

\begin{remark}\label{rem:T3} 
Similarly to Lemma~\ref{cor:embedparentsfinal)3}, with probability $1-o(n^{-3})$, 
the image $w_{3,4}$ of the root $r_{3,4}$ of $T_4$ is a $(d,t)$-extendable subgraph of 
$G_1'\coloneq G_1[[n]\setminus ((L''\cup W_1\cup W_2\cup W_3)\setminus \{w_{3,4}\})]$.
\end{remark}

Finally, to embed $T_4$ in the subset of still unoccupied vertices in $[n]\setminus L''$, we use the previous remark together with the following result appearing as \cite[Corollary~3.7]{Mon19} (where we substitute $R=\{w_{3,4}\}$, $T=T_4$, $\{r_{3,4}\}=T_3\cap T_4$ and $H=G_1'$). Recall $t=4(\ln\ln n)/p_1$.

\begin{lemma}\label{lem:edgeinsert3} 
Fix integers $\Delta\ge 3$, $t\ge 1$ and a tree $T$ with maximum degree at most $\Delta$ containing a vertex $r'$. 
Fix also a $t$-joined graph $G$ and suppose $R$ is a $(2\Delta,t)$-extendable subgraph of $G$ with maximum degree at most $\Delta$. 
Suppose that $w\in V(R)$ and $|V(R)| + |V(T)|\le |V(G)| - 4\Delta t - 3t$. Then, there is a copy $W$ of $T$ in $G\setminus (V(R)\setminus \{w\})$ where $r'$ is mapped to $w$ and $R\cup W$ is $(2\Delta,t)$-extendable in $G$.
\end{lemma}

This lemma is indeed applicable since exactly $|L|-|L''|\geq|L|-|L'|\ge \Delta \alpha\ge 4\Delta t+3t$ vertices in $[n]$ 
 remain unused in the embeddings of $\hat T_1$, $\hat T_2$, $T_3$ and $T_4$.

\subsubsection{\texorpdfstring{Embedding $L$ in the remainder of $D$: a generalised matching construction}{Embedding the remainder by constructing flexible matchings}}\label{subsec:embedL}

Finally, we need to redistribute the vertices in $[n]\setminus L''$ not used in the embedding of $T$ up to now to complement the layers $(L_i'')_{i=1}^{\imax}$ (thus ensuring that layer $L_i''$ obtains the right number of vertices to match the size of $L_i$) and finish the embedding.
Denote this set of vertices by $R'$ and partition it arbitrarily into sets $(R_i')_{i=1}^{\imax}$ so that, for every $i\in [\imax]$, $|R_i'|+|L_i''|=|L_i|$. We also denote $R_i'' = R_i'\cup L_i''$, let $R''$ be their union, and define $R_{\imax+1}''$ to be the image of $P_{\imax}$ under the constructed embedding.

Recall that, for every $i\in [\imax]$, $P_i$ is the set of parents of $L_i$ in $T$.
Note that $P_i\supseteq (L_{i+1}\setminus L_{i+1}^*)$ for all $i\in [\imax]$. The need to take $L_{i+1}^*$ into account in the previous inclusion is an unfortunate inconvenience; 
these vertices will be prioritised and embedded before everyone else into $L_{1,i+1}''$ at each embedding step.
Recalling that $L_1^*=\varnothing$, for every $i\in [\imax]$, we call a surjective embedding of the layer $L_i$ into $R''_i$ \emph{good} if $L_i^*$ embeds into $L_{1,i}''$; we also assume that the embedding of $L_{\imax+1}$ (which is done as part of the rollback procedure) is good by default.

Write $\varphi$ for the embedding of $T\setminus L$
constructed in Section~\ref{sec:rollback}, and define
\[
L_{r,\imax+1}''
:=\varphi(L_{\imax+1}\cap V(T_r))
\qquad \text{for each }r\in\{1,2\}.
\]
These two sets are disjoint. Moreover, since
shifting is performed only in $T_1$, the images
of the parents of shifted vertices belong to
$L_{1,i+1}''\cup\ldots\cup L_{1,\imax+1}''$,
whereas $P_{2,\imax}'\cup P_{2,\imax}''
\subseteq L_{2,\imax+1}''$.
All partial embeddings considered below extend
such an embedding $\varphi$; in particular,
their parent images avoid
$Q_L'\cup(L'\setminus L'')$.
Next, by relying on Lemma~\ref{lem:match}, we show that a good embedding of the layers $(L_i)_{i=1}^{\Gamma}$ into $(R''_i)_{i=1}^{\Gamma}$, respectively, typically exists.

\begin{proposition}\label{prop:expand}
Fix $t=4(\ln\ln n)/p_1$, a tuple $(x_1,\ldots,x_{\imax+1})$ from $\Lambda_{\imax}$, and $i\in [\imax]$.
The following holds whp: for every tree $T\in \cT_2'$ with layer sizes $(x_1,\ldots,x_{\imax+1})$ and a good embedding of each of the layers $L_{i+1},\ldots,L_{\imax+1}$, there exists a good embedding of the layer $L_i$. 
\end{proposition}
\begin{proof}
Fix a tree $T\in \cT_2'$ with layer sizes $(x_1,\ldots,x_{\imax+1})$ and good embeddings of the layers $L_{i+1},\ldots,L_{\imax+1}$.
We denote by $A_1=A_1(i)\subseteq (L''_{1,i+1}\cup \ldots\cup L_{1,\imax+1}'')$  the image of the parents of the vertices in $L_i^*$, $A=A(i) \supseteq A_1$ the set of images of the vertices in $P_i$, 
$a=|A|$, $B=B(i)=R_i''$, and $b=|B|$.
 Recall Definitions~\ref{def:f-matching} and~\ref{def:f-match}. 
By Lemma~\ref{lem:match} and since whp $G_1\subseteq G$ is a $t$-joined graph (due to Lemma~\ref{cor:embedparentsfinal)3}), it remains to show that points \ref{item:extended-matching1}--\ref{item:extended-matching4} from Definition \ref{def:f-match} 
 typically hold for $B_1=L_{1,i}''$ 
  and for all valid choices of pairs of sets $A_1\subseteq A$ (the precise conditions that these two sets have to satisfy are disclosed below) and $f$ (corresponding to the degree sequence in the bipartite graph $(L_i,P_i)$), with $t$ defined in the statement and $\eta$ to be defined shortly.
  Set a real number $k=(b-a)/(\Delta-1)$ so that $b=a+(\Delta-1)k$.

In order to show the expansion conditions \ref{item:extended-matching1}--\ref{item:extended-matching3} from Definition \ref{def:f-match} hold whp, we instead prove whp the following stronger conditions hold 
for any admissible function $f$, any set $A_1\subseteq (L''_{1,i+1}\cup \ldots\cup L_{1,\imax+1}'')$ with $|A_1| \le f(A_1)\le \delta_2 \alpha\eqqcolon \eta$ (where the second inequality is implied by \ref{item:leaf-shifting2} in Lemma~\ref{lem:layers} for some small enough $\delta_2 = \delta_2(\Delta)$), any $A\supseteq A_1$, and $ B_1:=L_{1,i}''$: 
\begin{enumerate}[start=0,label=(i\arabic*)]
    \item \label{item:prop5.15a} for all $U\subseteq A$ with $u = |U| \in [1,\eta]$, we have\footnote{Recall $N^-_D(U)$ denotes the external neighbourhood of $U$ in the directed graph $D$.} $\min\{|N^-_D(U)\cap L_{1,i}''|,|N^-_D(U)\cap L_{2,i}''|\}\ge \Delta u$, 
    \item \label{item:prop5.15b} for all $U\subseteq A$ with $u = |U|\in [\eta, \min\{k,t\}]$, we have $|N^-_D(U)\cap L_i''|\ge \Delta u + \eta$,
    \item \label{item:prop5.15c} for all $U\subseteq A$ with $u = |U|\in (\max\{\eta, \min\{k,t\}\},t]$, we have $|N
    ^-_D(U)\cap L_i''|\ge \Delta k + (u-k) + \eta$.
\end{enumerate}
We note that an index range with lower endpoint greater than its upper endpoint is understood to be empty (equality may be included depending on the interval being open or not); in particular, the range $(\eta,\min\{k,t\}]$ is empty when $k\le \eta$.
To understand why \ref{item:prop5.15a}--\ref{item:prop5.15c} imply points \ref{item:extended-matching1}--\ref{item:extended-matching3} from Definition~\ref{def:f-match}, note that
\begin{itemize}
    \item for every $U\subseteq A_1$, by \ref{item:prop5.15a}, we have $|U|\le |A_1|\le \eta$ and thus $|N(U)\cap B_1|\ge \Delta u\ge f(U)$,
    \item  for every $U\subseteq A\setminus A_1$ with $|U|\le \eta$, by \ref{item:prop5.15a}, we have $|N(U)\setminus B_1| \ge |N(U)\cap L_{2,i}''|\ge \Delta u\ge f(U)$,
    \item for every $U\subseteq A$ with $|U|\in [\eta+1,\min\{k,t\}]$, by \ref{item:prop5.15b}, we have 
    \[|N(U)|\ge \Delta u+\eta\ge f(U)+f(A_1)\ge f(U\cup A_1),\]
    \item for every $U\subseteq A$ with $|U|\in (\max\{\eta,\min\{k,t\}\},t]$, using \ref{item:prop5.15c} and that $b=\Delta k+(a-k)$, we have 
    \[|N(U)|\ge \Delta k+(u-k)+\eta\ge b - (a-u) + \eta\ge f(A)-|A\setminus U| + f(A_1)\ge f(U)+f(A_1)\ge f(U\cup A_1).\] 
\end{itemize}

With an eye towards the proof of \ref{item:prop5.15a}--\ref{item:prop5.15c}, define $\cE$ to be the event that, for every $j\in[\imax]$, every
$v\in[n]\setminus\bigl(Q_L'\cup(L'\setminus L'')\cup L_j''\bigr)$ and every $r\in\{1,2\}$,
\[
e_D(v,L_{r,j}'')\ge |L_{r,j}'|p_1/2.
\]
By Proposition~\ref{prop:props L'}, Lemma~\ref{lem:QL2}, and Remark~\ref{rem:QLprime}, this event has probability $1-o(1)$. 
Every eligible vertex which may potentially end up in $R_j''\setminus L_j''$ for some $j\in [\Gamma]$ belongs to this good vertex set, including parents in the already embedded top layer.
In the remainder of the proof, we condition on $\cE$ and continue the proof with three claims asserting that the above properties hold whp.

\begin{claim}
Property \ref{item:prop5.15a} is satisfied whp.
\end{claim}
\begin{proof}
Fix $u\in [1,\eta]$, set $b_1\coloneqq |L_{1,i}''|$ and $b_2\coloneqq |L_{2,i}''|$, and let $S_1\subseteq L_{1,i}''$ and $S_2\subseteq L_{2,i}''$ be arbitrary sets with $|S_1|=|S_2|=\Delta u-1$.
 Then, by the union bound, the probability that \ref{item:prop5.15a} fails for some set $U\subseteq A$ 
of size $u$ is bounded from above by
\begin{equation}\label{eq:union_bound0}
\begin{split}
\binom{n}{u} 
&\binom{b_1}{\Delta u-1} \mathbb P(N^-_D(U)\cap L_{1,i}''\subseteq S_1\mid \cE)
+\binom{n}{u} \binom{b_2}{\Delta u-1} \mathbb P(N^-_D(U)\cap L_{2,i}''\subseteq S_2\mid \cE).
\end{split}
\end{equation}
Recall Remark~\ref{rem:QLprime}  
which asserts the following: for each $r\in \{1,2\}$, conditionally on the fixed out-degree $e_D(v,L_{r,i}'')\ge b_rp_1/2$ of each vertex $v\in [n]\setminus \bigl(Q_L'\cup(L'\setminus L'')\cup L_i''\bigr)$
towards $L_{r,i}''$, the out-neighbours of $v$ in~$L_{r,i}''$ are distributed uniformly at random over all subsets of $L_{r,i}''$ of size $e_D(v,L_{r,i}'')$ and independently for different vertices.
Thus, for $r\in\{1,2\}$, the corresponding probabilities in \eqref{eq:union_bound0} are at most 
\begin{equation}\label{eq:bounds-for-uniform}
    \bigg(\prod_{i=1}^{b_r p_1/2} \bigg(\frac{|S_r|-i+1}{b_r-i+1}\bigg)\bigg)^u\le \bigg(\frac{|S_r|}{b_r}\bigg)^{ub_rp_1/2}= \bigg(\frac{\Delta u-1}{b_r}\bigg)^{ub_rp_1/2}.
\end{equation}
Inserting the latter expression in~\eqref{eq:union_bound0} and using standard approximations imply an upper bound of 
\begin{equation}\label{eq:u0}
\begin{split}
\bigg(\frac{\e n}{u}\bigg)^u \sum_{r\in\{1,2\}} \bigg(\frac{\e b_r}{\Delta u-1}\bigg)^{\Delta u} \bigg(\frac{\Delta u-1}{b_r}\bigg)^{ub_rp_1/2}
&\le \bigg(\frac{\e^{\Delta+1} n}{u}\bigg)^u \sum_{r\in\{1,2\}} \bigg(\frac{\Delta u}{b_r}\bigg)^{ub_rp_1/3}\\
&= \sum_{r\in\{1,2\}}  \bigg(\frac{\e^{\Delta+1} n}{u}\cdot \bigg(\frac{\Delta u}{b_r}\bigg)^{b_rp_1/3}\bigg)^u.
\end{split}
\end{equation}
In the first inequality we use the fact that $b_1,b_2$ are of order $\alpha$ so that $b_rp_1=\omega(1)$. This follows from the choices of sizes of $L_{1,i}'$ before Proposition~\ref{prop:props L'}, and the consequence for the sizes of $L_{1,i}''$ in part (i) of the proposition. Each of the summands of the latter expression can be rewritten as 
\begin{equation}\label{eq:exponential}
\exp\bigg(\bigg(\frac{b_rp_1}{3}-1\bigg)u\ln(u) - \frac{b_rp_1}{3}u\ln\bigg(\frac{b_r}{\Delta}\bigg) + u\ln(\e^{\Delta+1} n)\bigg).
\end{equation}
A computation of the second derivative shows that the function in \eqref{eq:exponential} is $\log$-convex (in $u$). Consequently, its maximum over the interval $u\in [1,\eta]$ is attained at $u=1$ or at $u=\eta$.
In the former case,~\eqref{eq:exponential} is of order $n^{-\omega(1)}$. 
To analyse the latter case, note that $b_r\ge 9\alpha \ge 9\Delta \eta $ for both $r\in\{1,2\}$. Hence, when $u=\eta$, each of the summands in \eqref{eq:u0} is at most
\[\bigg(\frac{\e^{\Delta+1} n}{u}\cdot \bigg(\frac{1}{9}\bigg)^{3\alpha p_1}\bigg)^u = n^{-\omega(1)},\]
where we used that $3\alpha p_1\ge \ln\ln n$ for $C\ge K^4$ and $u=\eta=\Theta(\alpha)$. 
\end{proof}
\begin{claim}
Property \ref{item:prop5.15b} is satisfied whp.
\end{claim}

\begin{proof}
Fix $u \in [ \eta,\min\{k,t\}]$, set $b_1\coloneqq |L_{1,i}''|$ and $b_2\coloneqq |L_{2,i}''|$.
The probability that \ref{item:prop5.15b} fails for some set $U\subseteq A$ of size $u$ is at most
\begin{equation}\label{eq:union_bound}
\begin{split}
\binom{n}{u} \sum_{s_1,s_2} \prod_{r\in \{1,2\}} \bigg(\binom{b_r}{s_r} \mathbb P(N^-_D(U)\cap L_{r,i}''\subseteq {S}_r\mid \cE)\bigg),
\end{split}
\end{equation}
where the sum is over $s_1\in [0,b_1]$ and $s_2\in [0,b_2]$ such that $s_1+s_2=\Delta u + \eta -1$,
 and $S_1\subseteq L_{1,i}''$ and $S_2\subseteq L_{2,i}''$ are arbitrary fixed sets of size $s_1,s_2$, respectively. Also, for each $r\in \{1,2\}$, similarly to \eqref{eq:bounds-for-uniform}, the corresponding probability in \eqref{eq:union_bound} is at most
$({s_r}/{b_r})^{ub_rp_1/2}.$
Inserting this bound in~\eqref{eq:union_bound} and using that $b_rp_1 = \omega(1)$ and $s_1+s_2\leq(\Delta+1)u$, as well as standard approximations, we get the following upper bound: 
\begin{align}
\bigg(\frac{\e n}{u}\bigg)^u \sum_{s_1,s_2} \prod_{r\in\{1,2\}} 
&\bigg(\frac{\e b_r}{s_r}\bigg)^{s_r} \bigg(\frac{s_r}{b_r}\bigg)^{ub_rp_1/2}\le \bigg(\frac{\e^{\Delta+2} n}{u}\bigg)^u \sum_{s_1,s_2} \prod_{r} \bigg(\frac{s_r}{b_r}\bigg)^{ub_rp_1/3}\nonumber\\
&= \sum_{s_1,s_2}  \bigg(\frac{\e^{\Delta+2} n}{u} \bigg(\frac{s_1}{b_1}\bigg)^{b_1p_1/3} \bigg(\frac{\Delta u+\eta-1-s_1}{b_2}\bigg)^{b_2p_1/3}\bigg)^u.\label{eq:u}
\end{align}

To estimate~\eqref{eq:u}, a simple derivative computation shows that the largest term in~\eqref{eq:u} corresponds to $s_r = (\Delta u +\eta-1) b_r/(b_1+b_2)$ for $r\in \{1,2\}$,
and thus the maximal term is at most \begin{equation}\label{eq:maxu}
\bigg(\frac{\e^{\Delta+2} n}{u}\bigg(\frac{\Delta u +\eta}{b_1+b_2}\bigg)^{(b_1+b_2)p_1/3}\bigg)^u.
\end{equation}
Next, we show that~\eqref{eq:maxu} is of order $n^{-\omega(1)}$, which combined with a union bound over the $O(n)$ values for $(s_1,s_2)$ in~\eqref{eq:u} concludes the proof of the claim.
To do so, we distinguish between two cases depending on whether or not $k\ge t$. 
Before dealing with these cases, we note the following inequality used in both of them:
\begin{equation}\label{eq:b1b2}
    b_1+b_2= b-|R_i'| \ge b-2\Delta^{-2}\alpha \ge \max\{3b/4,\Delta u+\eta\}.
\end{equation}
The last inequality can be justified as follows. For the first term in the maximum, note that $b\ge a\ge 19\alpha$ by definition, and thus $b\ge 3b/4+2\Delta^{-2}\alpha$.
For the second term in the maximum, we have
\[\Delta u+\eta+2\Delta^{-2}\alpha\le \Delta \min\{k,t\}+\eta+2\Delta^{-2}\alpha \le (\Delta-1) k + t+\eta+2\Delta^{-2}\alpha\le (\Delta-1) k+a=b,\]
where the last inequality holds provided that $C\geq K^4$ and $K$ sufficiently large.

\vspace{1em}
\noindent
\textbf{Case 1.} Assume that $k\ge t$ or equivalently that $b\ge a+t(\Delta-1)$. Using \eqref{eq:b1b2}, we can bound \eqref{eq:maxu} by
\begin{equation}\label{eq:maxunew}
    \bigg(\frac{\e^{\Delta+2} n}{u}\bigg(\frac{\Delta u +\eta}{b-2\Delta^{-2}\alpha}\bigg)^{bp_1/4}\bigg)^u.
\end{equation}
In turn, using the assumption of this case, and provided that $C\geq K^4$ and $K$ is large enough, we have
\begin{equation}\label{eq:inequalities}
a\ge 19\alpha\ge 2(t+\eta+2\Delta^{-2}\alpha)\quad\text{and}\quad \frac{\Delta u+\eta}{b-2\Delta^{-2}\alpha}\le \frac{\Delta t+\eta}{b-2\Delta^{-2}\alpha} \le \frac{b+\eta-(a-t)}{b-2\Delta^{-2}\alpha}\le 1-\frac{a/2}{b}. 
\end{equation}
Finally, by combining the inequalities \eqref{eq:maxunew} and \eqref{eq:inequalities}, using the standard bound $1-x\le \e^{-x}$ for $x\ge 0$ and letting $C\geq K^{3}$, we can dominate \eqref{eq:maxu} by
\[\bigg(\frac{\e^{\Delta+2} n}{u}\cdot \bigg(1-\frac{a}{2b}\bigg)^{bp_1/4}\bigg)^u \le \bigg(\frac{\e^{\Delta+2} n}{\eta} \cdot \exp\bigg(-\frac{a}{2b}\cdot\frac{bp_1}{4}\bigg)\bigg)^u = \bigg(\frac{\e^{\Delta+2} n}{\eta}\cdot  \e^{-ap_1/8}\bigg)^u = n^{-\omega(1)}.\]

\vspace{1em}
\noindent
\textbf{Case 2.}  Assume that $k\le t$ or equivalently that $b\le a+t(\Delta-1)$. Using \eqref{eq:b1b2}, we can bound the expression in \eqref{eq:maxu} by
\begin{equation}\label{eq:maxu2}
\bigg(\frac{\e^{\Delta+2} n}{u}\bigg(\frac{\Delta u +\eta}{b-2\Delta^{-2}\alpha}\bigg)^{bp_1/4}\bigg)^u.
\end{equation}
When the range $[\eta,k]$ is nonempty, the above expression is a $\log$-convex function of $u$, so its maximum is attained at $u=\eta$ or at $u=k$.
In both cases, using the inequalities $a\ge 2(k+\eta)$ and $\Delta k + \eta\le b-a/2$ (which holds provided that $C\geq K^{3}$), a similar analysis to the previous case shows that \eqref{eq:maxu2} is of order $n^{-\omega(1)}$, concluding the proof of this claim.
\end{proof}

\begin{claim}
Property \ref{item:prop5.15c} is satisfied whp.
\end{claim}

\begin{proof}
    Note that if $k\geq t$, there is nothing to prove, and thus we may assume that $k < t$, implying that $b \le a+t(\Delta-1)$.
    Fix $u\in (\max\{\eta, k\},t]$ and, for every pair of non-negative integers $s_1,s_2$ such that $s_1+s_2 = \Delta k+(u-k)-1+\eta\le u+(\Delta-1)k+\eta$, fix a pair of sets $S_1\subseteq L_{1,i}''$ and $S_2\subseteq L_{2,i}''$ with sizes $s_1$ and $s_2$, respectively.
    Then, the probability that \ref{item:prop5.15c} fails for some set $U\subseteq A$ of size $u$ is at most
\begin{equation}\label{eq:union_bound:i2}
\binom{n}{u} \sum_{s_1,s_2}\prod_{r\in\{1,2\}} \bigg(\binom{b_r}{s_r} \mathbb P(N^-_D(U)\cap L_{r,i}''\subseteq S_r\mid \cE)\bigg)
\end{equation}
where the sums here and below are over $s_1\in [0,b_1]$ and $s_2\in [0,b_2]$ such that $s_1+s_2=\Delta k+(u-k) + \eta -1$.
Similar computations to the previous claims show that the probabilities in~\eqref{eq:union_bound:i2} are at most $({s_r}/{b_r})^{ub_rp_1/2}$. Thus, similarly to \eqref{eq:u}, the above and standard approximations imply that~\eqref{eq:union_bound:i2} is at most
\[\sum_{s_1,s_2} \bigg(\frac{\e^{\Delta+2} n}{u}\bigg(\frac{s_1}{b_1}\bigg)^{b_1p_1/3} \bigg(\frac{\Delta k + (u-k)-1+\eta-s_1}{b_2}\bigg)^{b_2p_1/3}\bigg)^u.\]
Moreover, using an analogous derivative computation to the one leading to~\eqref{eq:maxu} and the inequalities~\eqref{eq:b1b2}, each summand in the preceding sum can be further bounded by
\[\bigg(\frac{\e^{\Delta+2} n}{u}\bigg(\frac{\Delta k + (u-k)+\eta}{b_1+b_2}\bigg)^{(b_1+b_2)p_1/3}\bigg)^u\le \bigg(\frac{\e^{\Delta+2} n}{u}\bigg(\frac{\Delta k + (u-k)+\eta}{b-2\Delta^{-2}\alpha}\bigg)^{bp_1/4}\bigg)^u.\] 
 Assuming that $C\geq K^4$, we may replace the second inequality in~\eqref{eq:inequalities} with 
\[\frac{u+(\Delta-1)k+\eta}{b-2\Delta^{-2}\alpha} = \frac{(b-2\Delta^{-2}\alpha)-(a-u-\eta-2\Delta^{-2}\alpha)}{b-2\Delta^{-2}\alpha}\le 1-\frac{a/2}{b}.\]
Finally, repeating the analysis as in the previous claims, we find that each summand is of order $n^{-\omega(1)}$. 
Moreover, there are at most $n+1$ choices of $(s_1,s_2)$ for each fixed $u$, so~\eqref{eq:union_bound:i2} is also of order $n^{-\omega(1)}$ for every $u\in (\max\{\eta, k\},t]$. 
Hence, a union bound over the $O(n^2)$ choices of $(u,s_1,s_2)$ concludes the proof of the claim. 
\end{proof}

Finally, we derive property \ref{item:extended-matching4} in Definition~\ref{def:f-match} and thus conclude the proof. It suffices to show that $B$ expands in $G[L_{2,i+1}'',B]$ with probability $1-n^{-\omega(1)}$: this is enough since $A_1$ is disjoint from $L_{2,i+1}''\subseteq A$ (the latter inclusion holds since $L_{i+1}\setminus L^*_{i+1}\subset P_i$ embeds into $A$ whereas $L^*_{i+1}$ embeds into $L''_{1,i+1}$).
Assume first that $i\in[\imax-1]$; the case $i=\imax$ will be treated slightly differently at the end of the proof. 
When $u=1$, the expansion follows directly from the positive out-degree guarantee towards $L_{2,i+1}''$ on $\cE$, so fix an integer $u\in [2,t]$.
We bound from above the probability that there exists a set $U\subseteq B$ of size $u$ which does not expand by a factor of 1 towards $L_{2,i+1}''$.
Namely, setting $a_2\coloneqq |L_{2,i+1}''|$ and fixing a set $S\subseteq L_{2,i+1}''$ of size $u-1$, by the union bound, the expansion property \ref{item:extended-matching4} is violated with probability at most
\begin{equation}\label{eq:union_bound_B0}
\binom{a_2}{u-1} \binom{n}{u} \mathbb P(N^-_D(U)\cap L_{2,i+1}''\subseteq S\mid \cE).
\end{equation}
The probability of the latter event is at most $((u-1)/a_2)^{ua_2p_1/2}$. Assuming $C\geq K^4$, we get that the expression in \eqref{eq:union_bound_B0} is bounded from above by 
\begin{align*}
\bigg(\frac{\e a_2}{u-1}\bigg)^{u-1} \bigg(\frac{\e n}{u}\bigg)^u \bigg(\frac{u-1}{a_2}\bigg)^{ua_2p_1/2}
&= \bigg(\frac{\e a_2}{u-1}\frac{u-1}{a_2}\bigg)^{u-1} \bigg(\frac{\e n}{u}\frac{u-1}{a_2}\bigg)^{u}\bigg(\frac{u-1}{a_2}\bigg)^{ua_2p_1/2-2u+1}\\
&\le \bigg(\e\cdot \frac{\e n}{a_2}\cdot \bigg(\frac{u}{a_2}\bigg)^{a_2p_1/3}\bigg)^u.
\end{align*}
Note that, as $C\geq K^4$, we have $a_2 \geq 9\alpha\geq 10t$ and thus $n/a_2 = O(\ln n)$ and $a_2p_1\ge 9\alpha p_1\ge 4\ln\ln n$.
We conclude that \eqref{eq:union_bound_B0} is of order $n^{-\omega(1)}$ for every $u\in [2,t]$. The union bound over $u\in [2,t]$ and $i\in [\imax-1]$ concludes the argument for all layers but the last one.

Finally, when $i=\imax$, we recall that every vertex $v\in R_{\imax}''$ either sends $\alpha p_1/3$ edges in $D$ towards $P_{2,\imax}'$ (and the neighbourhood of $v$ is distributed uniformly at random conditionally on its size), or has $\alpha p_1/2$ out-neighbours associated exclusively to it in $P_{2,\imax}''$.
Denote the set of all $v$ that do not have an exclusive set of $\alpha p_1/2$ out-neighbours in $P''_{2,\Gamma}$ by $R'''_{\imax}$ and note that it suffices to show that all subsets of $R'''_{\imax}$ of size at most $t$ expand by a factor of 1 towards $P'_{2,\Gamma}$, which is done analogously to the previous case by applying the union bound similarly to~\eqref{eq:union_bound_B0}.
\end{proof}

The typical $\cT_2'$-universality of $G(n,C\ln n/n)$ for a suitably large absolute constant $C$ now follows from $\imax$ iterations of Proposition~\ref{prop:expand}.

\begin{remark}\label{rem:verbatim}
We note that the proof of
Proposition~\ref{prop:expand} extends verbatim
to rooted forests with any fixed number of layers,
provided that the corresponding layer-size,
parent-localisation, and expansion properties
used in its proof are satisfied.
In particular, the argument does not depend
on the precise value of $\Gamma$.
\end{remark}

\subsection{\texorpdfstring{Case 2: rollback, absorption, and linking systems}{First case: matchings rooted at a linking system}}

In this section, we are concerned with the universality of $\cT_2''$. For the convenience of the reader, we recall some notation and definitions. The family $\cT_2$ contains all trees with at most $\exp(\Delta^5)\alpha$ leaves and no $K\alpha$ vertex-disjoint bare paths~of~length~$n/(K^2\alpha)$.
The family $\cT_2''$ is then defined as the collection of trees $T$ for which Algorithm \ref{algo:leaf-cut} applied to $T_1$ and $T_2$ yields $\min\{\mathbf{i_1},\mathbf{i_2}\} < \imax+1$. 
The goal of this section is to prove the following theorem.

\begin{theorem}\label{thm:univ_super-adapted}
There is a constant $C>1$ such that, for every $\Delta\ge 2$ and $p=p(n)\coloneq C\ln n/n$, whp $G(n,p)$ is $\cT_2''$-universal.  
\end{theorem}

\paragraph{First step: a structural decomposition}
We begin by exhibiting a convenient structural decomposition for every $T\in \cT_2''$.
By Theorem~\ref{thm:univ_adapted}, trees containing $\alpha$ vertex-disjoint bare paths of length $n/(K^2\alpha)$ can already be embedded simultaneously. We henceforth restrict to the remaining trees, which contain no such family.
To this end, fix such a tree $T\in \cT_2''$, recall the tree $\hat T$ obtained as output of Algorithm~\ref{algo:leaf-cut}, and choose $j\in\{1,2\}$ with $\mathbf{i_j}=\min\{\mathbf{i_1},\mathbf{i_2}\}\le\imax$.
Then, the eligible leaf set of $\hat T_j = \hat T\cap T_j$ contains no $24$-separated set of $(10+\delta_1)\alpha$ leaves at distance at least 2 from the root.
Since a maximal such set covers all eligible leaves in $\hat T_j$ by balls of radius $23$, missing at most $2\Delta$ leaves at distance 1 from the (at most) two roots, the number of leaves in $\hat T_j$ is at most
\begin{equation}\label{eq:leaves Tjhat}
(10+\delta_1)\alpha(1+\Delta+\ldots+\Delta^{23})+2\Delta+2
\le20\Delta^{24}\alpha.
\end{equation}

Set $\mathbf{i}=\max\{\mathbf{i_1},\mathbf{i_2}\}$. Leaf deletion does not increase the number of leaves, and each $T_j$ has at most $\exp(\Delta^5)\alpha$ leaves. Consequently,
\[
|V(T_j)\setminus V(\hat T_j)|
=\sum_{r=1}^{\mathbf{i_j}-1}|D_r^j|
\le \imax\cdot \exp(\Delta^5)\alpha=o(n).
\]
Hence, for all sufficiently large $n$,
\[
|V(\hat T_j)|\ge \min\{|V(T_1)|,|V(T_2)|\}-o(n)
\ge n/9-o(n)\ge n/10.
\]

Throughout this subsection, we require the following notation from Section~\ref{sec:long bare paths} (see, in particular,~\eqref{def:ell}): 
\begin{equation}\label{eq:k-ell-from-previous-sections}
    k\coloneqq \bigg\lceil \frac{25\ln n}{\ln\ln n}\bigg\rceil\quad\text{and}\quad \ell\coloneqq  \frac{n}{2K^2\alpha}+6k+4\leq \frac{n}{1.9K^2\alpha}.
\end{equation}

Since $\hat T_j$ has $o(n)$ leaves by~\eqref{eq:leaves Tjhat},
we may apply Corollary~\ref{cor:separate}
to $\hat T_j$ with $d=12$, $h=2\ell$,
and $S$ being its set of roots.
The leaf alternative in the proof of Corollary~\ref{cor:separate} is excluded by the failed test in
Algorithm~\ref{algo:leaf-cut}. Thus, $\hat T_j$ contains at least
\[
\frac{|V(\hat T_j)|}{100\ell}
\ge \frac{K^2\alpha}{800}
\]
vertex-disjoint bare paths of length $2\ell$,
all at distance at least $25$ from its roots,
where the last inequality holds for
sufficiently large $K$ and $n$.
We decompose the tree $T$ into three parts:
\begin{itemize}
\item $T_{\rm paths}$ is the union of $K^2\alpha/800$ vertex-disjoint bare paths among the ones exhibited above,
\item $T_{\rm small}$ is the rooted forest obtained as union of all components in $T_j\setminus V(\hat T_j)$ adjacent to $T_{\rm paths}$ and their edges towards $T_{\rm paths}$, and 
\item $T_{\rm rem} \coloneqq T\setminus (T_{\rm small}\cup T_{\rm paths})$.
\end{itemize}
We note that all nonroot vertices of the hanging trees in $T_{\rm small}$ were removed before round $\mathbf{i}_j$, so their heights are at most
$\mathbf{i}_j - 1 \le \imax - 1$.

Our strategy to embed $T$ combines elements from the approaches used for the embedding of the trees in $\cT_1$ and $\cT_2'$ with several important adjustments.
First of all, if $T_{\rm small}$ were disjoint from at least $\alpha$ of the paths in $T_{\rm paths}$,
there would exist $\alpha$ vertex-disjoint bare paths of length $n/(K^2\alpha)$ in $T$. 
This contradicts
our assumption that $T$ contains no such family of $\alpha$ vertex-disjoint bare paths.
Hence, we have that $T_{\rm small}$ is disjoint from less than $\alpha$ of the paths in $T_{\rm paths}$ and, in particular, contains at least $K^2\alpha/800 - \alpha\ge K\alpha$ trees. 

\paragraph{Absorbing $T_{\rm small}$ using many-to-one matchings: the first obstruction}
A natural first step towards proving Theorem~\ref{thm:univ_super-adapted} consists of reserving a set of vertices to embed $T_{\rm small}$; however, this approach is too na\"ive, as we now explain.
One natural way to perform this task is to iterate a leaf-cutting procedure to divide $T_{\rm small}$ into several layers in a way ensuring that:
\begin{itemize}
    \item leaves of $T_{\rm small}$ belong to the first layer and roots of $T_{\rm small}$ belong to the last layer,
    \item every layer has size $\Omega(\alpha)$ (with the implicit constant independent of $\Delta$),
    \item every non-root vertex sends exactly one edge to (the vertices in) the next layers.
\end{itemize}
Note that such a layering is used for embedding parts of the trees in $\cT_2'$ via many-to-one matchings (constructed in Section~\ref{subsec:embedL}).
Sadly, such a layering does not always exist in this case: for example, suppose that the forest $T_{\rm small}$ consists of $\alpha/((i+1)\ln(i+1))$ paths of length $i$ rooted at one of their endpoints, for every $i\in [\Delta]$, and that it admits the required layering. Note that, assuming $\Delta$ suitably large, this construction satisfies the requirements imposed up to now: it contains about $\alpha\ln\ln \Delta\ge K^2\alpha/800$ rooted paths (thus ensuring that $T_{\rm small}$ contains sufficiently many trees, as assumed in the previous paragraph).
On the one hand, the paths in $T_{\rm small}$ contain a total~of
\[\sum_{i=1}^{\Delta} (i+1)\cdot \frac{\alpha}{(i+1)\ln (i+1)} = O\bigg(\frac{\Delta \alpha}{\ln \Delta}\bigg)\]
vertices, thus implying that the number of layers is $O(\Delta/\ln \Delta)$.
On the other hand, the number of layers is necessarily greater than or equal to the size of the longest path: indeed, different vertices on the same path in $T_{\rm small}$ need to belong to different layers, thus leading to a contradiction when $\Delta$ is large.
While one could try to slightly reduce the number of paths in $T_{\rm paths}$ or slightly shorten some of these, this is unfortunately insufficient to conclude in the latter example: indeed, the number of roots there is of order $\Theta(\alpha \ln\ln \Delta)$, thus implying that each path in our family could contain $\Omega(\ln\ln \Delta)$ roots of long paths in $T_{\rm small}$ close to each other.

\paragraph{Height equalisation to address the first obstruction.}

The problem in the previous example is the few very long paths responsible for a `peak' in the distribution of the path lengths.
Our solution to this problem is to reroot these paths in a lower vertex (or, for general trees, shift the root down a descending path similarly to Algorithm~\ref{algo:layers}) and embed the part between the old and the new root inside $G[V_1]$: this would then `regularise' the distribution of the lengths and allow us to construct a valid layering for the remaining paths (where the roots of the paths which were partially constructed are shifted to their last vertex in $V_1$).
We start with a definition which uses the notion of a descending tree of a vertex presented after Corollary \ref{cor:separate}.

\begin{definition}[Equalising heights via shifts]\label{def:layers-small}
Fix a rooted forest $F$, write $(F_1,v_1),\ldots,(F_t,v_t)$ for the rooted trees of $F$ ordered from highest to lowest, and assume that $s\leq t$.
Define $h=h(F,s)$ to be the height of $F_{s}$ 
and consider vertices $u_1\in F_1,\ldots,u_{s}\in F_{s}$ such that for every $i\in [s]$, the trees $H_i\coloneqq (F_i)_{u_i}$ have height $h$, and $u_i$ is an arbitrary
 vertex with the above property.
Then, define the \emph{$s$-equalised forest} $F^s$ by replacing $(F_i)_{i=1}^{s}$ by $(H_i)_{i=1}^{s}$ in $F$.
\end{definition}

\paragraph{Absorbing $T_{\rm small}$ using many-to-one matchings: the second obstruction} Equalising the heights of the highest $s = \Omega(\alpha)$ trees in the forest $T_{\rm small}$ as described above resolves the problem of not being able to split $T_{\rm small}^{s}$ into layers of size $\Omega(\alpha)$, as each set among $(D_i\cap V(T_{\rm small}^s))_{i\in [h+1]}$ now has size $\Omega(\alpha)$.
Unfortunately, embedding $T_{\rm small}^s$ layer-by-layer is still unfeasible: another problem is that parts of the trees in $T_{\rm small}$ intersect the long bare paths in $T_{\rm paths}$ and thus their roots may potentially have to be embedded within the linking system.
However, our construction of the linking system does not allow to identify these roots until all vertices in $T_{\rm paths}$ have been embedded, while the vertices of $G$ hosting the forest $T_{\rm small}^s$ need to be decided in advance.
Then, for a vertex $v\in V(G)$, if all neighbours of $v$ in $G$ are occupied by images of non-root vertices in $T_{\rm paths}$, $v$ cannot be the image of a child of such a root.

A solution to the above obstruction is to embed some of the paths in $T_{\rm paths}$ and parts of the trees hanging from them into $G[V_1]$ using the rollback technique from Section~\ref{sec:rollback}.
This way, the images of some roots of trees in $T_{\rm small}^s$ will cover a conveniently chosen subset of vertices of $G$ where every vertex in the layering expands. 

\paragraph{Towards the proof of Theorem~\ref{thm:univ_super-adapted}}
We turn to the details. The next lemma provides a convenient structural decomposition of $T$.

\begin{lemma}\label{lem:22-separated roots}
Fix a tree $T\in \cT_2''$ with no $\alpha$ vertex-disjoint bare paths of length $n/(K^2 \alpha)$.
The tree $T$ contains a family of $\alpha$ vertex-disjoint paths $T_{\rm paths}^-$, each of length $\ell$, a rooted forest $T_{\rm small}^-$, and forests $F_1,F_2,$ and $F_3$ such that each of the following properties holds.
    \begin{enumerate}[label=\normalfont{(\arabic*)}]
        \item\label{pt(1)} There are vertices $r_{1,2}$ and $r_{2,3}$ such that $F_1\cap F_2 = \{r_{1,2}\}$ and $F_2\cap F_3 = \{r_{2,3}\}$ while $F_1$ intersects $F_3$ only if $r_{1,2}=r_{2,3}$. Moreover, each of $F_1,F_2,F_3$ has at most $\alpha+1$ components.
        \item\label{pt(2)} The forest $T_{\rm small}^-$ can be divided into layers $L_1,\ldots,L_{h+1}$ for some $h\in [\mathbf{i}]$ with $L:=\bigcup_{i=1}^h L_i$ such that:
        \begin{itemize}
            \item The layer $L_{h+1}$ contains all roots of trees in $T_{\rm small}^-$. 
            \item None of the layers contains $r_{1,2}$ or $r_{2,3}$.
            \item There exist subsets $L_i^*\subseteq L_i$ for all $i\in[h+1]$, with $L_1^*=\varnothing$, and a constant $M_\Delta$ depending only on $\Delta$ such that
\[
|L_i|-20\alpha\in
\bigcup_{m=0}^{M_\Delta}
[m\delta_1^i\alpha,
m\delta_1^i\alpha+\Delta^\imax]
\qquad \text{ for all } i\in [h+1],
\]
and
\[
\left|\bigcup_{i=1}^{h+1}L_i^*\right|
\le\delta_2\alpha.
\]
In particular, the number of possible
layer-size vectors is bounded as a
function of $\Delta$.
            \item Every vertex of $L_i$ for some $i\in [h]$, has exactly one neighbour in $T\setminus\bigcup_{j=1}^i L_j$, its parent, which belongs to $L_{i+1}\cup\cdots\cup L_{h+1}$. Writing $P_i$ for these parents, we have $|P_i|\ge19\alpha$ and
            \[
            L_{i+1}\setminus L_{i+1}^*\subseteq P_i
            \qquad \text{ for all $i\in [h-1]$.}
            \]
            For all $i\in[h]$, the parents of
vertices in $L_i^*$ belong to
$\bigcup_{j=i+1}^{h+1}L_j^*$. Also,
$L_{h+1}^*\subseteq V(F_1)$.
        \end{itemize}
        \item\label{pt(3)} Writing $\operatorname{Int}(T_{\rm paths}^-)$ for the set of internal vertices of the selected paths, we have
        \[
        V(T)=V(F_1\cup F_2\cup F_3)\sqcup
        \operatorname{Int}(T_{\rm paths}^-)\sqcup L,
        \qquad
        T=F_1\cup F_2\cup F_3\cup T_{\rm paths}^-\cup T_{\rm small}^-.
        \]
        The paths in $T_{\rm paths}^-$ are bare in $T\setminus L$, and every root of $T_{\rm small}^-$ is in $V(F_1\cup F_2\cup F_3\cup T_{\rm paths}^-)$.
        \item\label{pt(4)} For each $r\in\{1,2\}$, there exists a
$22$-separated set
\[
U_r\subseteq
(P_h\cap V(F_r))\setminus L_{h+1}^*,
\qquad |U_r|=(10+\delta_1)\alpha.
\]
Moreover, $F_3$ contains a $22$-separated
set $U_3$ of $30\beta$ vertices disjoint
from $\bigcup_{i=1}^h P_i
\cup\{r_{1,2},r_{2,3}\}$.
    \end{enumerate}
\end{lemma}

\begin{proof}[Proof of Lemma~\ref{lem:22-separated roots}]
Our proof is algorithmic. Consider a dynamic path forest $S\leftarrow T_{\rm paths}$ where all paths have length $\ell_S\leftarrow 2\ell$.
We say that a path forest $S$ is \emph{good} if $\ell_S\ge \ell$, it contains at least $K\alpha$ trees, and at least $K\alpha/2$ of its $K\alpha$ highest trees in $T_{\rm small}$ rooted in $S$ (ties broken arbitrarily) have their roots at distance at least $22$ from the set of terminal vertices of $S$.
 Now, at each step, check if $S$ is a good path forest or if $S$ contains strictly less than $K\alpha$ roots of trees in $T_{\rm small}$. If not, update $S$ by replacing every path $P$ therein with a subpath of $P$ where the first 22 and the last 22 vertices have been deleted, and set $\ell_S\leftarrow \ell_S-44$.

We claim that the described algorithm produces a good output $S$ after at most $\exp(\Delta^5)$ steps with at least $K\alpha$ roots of trees in $T_{\rm small}$. 
 First, note that every step shortens the paths in $S$ by removing 44 vertices from each of them, so any number of $o(\ln n/\ln\ln n)=o(\ell)$ steps would preserve the requirement that $\ell_S\ge \ell$. 
 Moreover, at every step, at least $K\alpha/2$ roots  
 of trees in $T_{\rm small}$ are evacuated from $S$, and since $T$ contains at most $\exp(\Delta^5)\alpha$ leaves (and thus at most that many roots in $T_{\rm small}$), the number of steps is at most $\exp(\Delta^5)$. 
Finally, if the set obtained by the algorithm is not good, there must be strictly fewer than $K\alpha$ trees in $T_{\rm small}$ with roots in $S$. 
On the other hand, since throughout the process only at most $44\exp(\Delta^5)$ vertices were removed from each path in $S$ and $T_{\rm paths}$ contains $K^2\alpha/800\ge (K+1)\alpha$ paths, the resulting path forest has at least $\alpha$ disjoint paths, each of length $2\ell-O_{\Delta}(1)\ge n/(K^2\alpha)$ (since $\ell=\omega(1)$), which are bare in $T$.
Recalling our assumption that $T$ contains no such paths
yields a contradiction, and so $S$ is indeed good. 

We conclude that, in addition to the set output by the above process being good, it contains at least $K\alpha$ roots of trees in $T_{\rm small}$. Note that every root at distance at least $22$ from the terminal sets in $S$ can have up to $44$ different roots at distance at most $22$ from itself.
As a result, one can greedily extract a set $R$ of $(K\alpha/2)/45\ge 100\alpha$ roots of some of the $K\alpha$ highest trees in $T_{\rm small}$ which is 22-separated in $T$.
Let $\mathcal H$ be the family of the $K\alpha$ highest trees in $T_{\rm small}$ rooted in $S$, breaking ties arbitrarily, and let $h$ be their minimum height, where $1\le h\le \mathbf{i}_j-1\le \imax-1$.
We define $T_{\rm paths}^-$ to be a family consisting of one path of length $\ell$ from each of $\alpha$ distinct components of $S$ containing no root of $\mathcal H$; such a choice is possible since at most $K\alpha$ components of $S$ contain these roots.
Also, denote by $V^*$ the set of vertices which are either internal for some path in $T_{\rm paths}^-$ or belong to a tree in $T_{\rm small}$ rooted in an internal vertex for some path in $T_{\rm paths}^-$.
By construction, $R\cap V^*=\varnothing$. In particular, every hanging tree rooted on the selected paths has height at most $h$.

Next, we construct the forests $F_1,F_2$ and $F_3$, as well as the rooted forest $T_{\rm small}^-$.
Each component of $T[V^*]$ consists of one selected path interior and all branches hanging from it, and has exactly two boundary edges. Thus, $T\setminus V^*$ has exactly $\alpha+1$ components. 
To apply Remark~\ref{rem:split} twice with $Q=R$, add the $\alpha$ edges joining the retained endpoints of the selected paths, apply Lemma~\ref{lem:tree-part} twice to the resulting tree, and then remove these auxiliary edges, retaining all isolated vertices.
Denote the resulting forests by
$\tilde F_1,\tilde F_2,\tilde F_3$ and their intersection
vertices by $r_{1,2}$ and $r_{2,3}$, as in~\ref{pt(1)}.
Since $|R|/9\ge 11\alpha+2$ for sufficiently large $n$
and $\beta=o(\alpha)$, by possibly ignoring trees containing $r_{1,2}$ and $r_{2,3}$, we ensure the existence of the set $U_3$ in~\ref{pt(4)} and find sets $A_1\subseteq V(\tilde F_1),A_2\subseteq V(\tilde F_2)$, each containing exactly $11\alpha$ vertices in $R$ (so $A_1,A_2$ are therefore also $22$-separated in $T$).

For every $v\in A_1\cup A_2$, choose a descendant $u_v$ in the tree $T_{\rm small}[v]$ descending from $v$ in $T_{\rm small}$ such that $H_v:=T_{\rm small}[u_v]$ has height exactly $h$. 
Let $B$ be the union of these $H_v$ and all the trees in $T_{\rm small}$ rooted in internal vertices of $T_{\rm paths}^-$. 
Its components have height at most $h$, and each of the two designated families $(H_v)_{v\in A_r}$ consists of $11\alpha$ trees of height exactly $h$.

Let $B_1,\ldots,B_h$ be the leaf-cutting layers consisting of the nonroot vertices of $B$, and let $B_{h+1}$ be its root set. 
We run the layer-size discretisation  Algorithm~\ref{algo:layers} with $h$ instead of $\imax$ on the layers $B_1,\ldots,B_{h+1}$, and $S_0=\{u_v:v\in A_1\}$. 
We note that we permit shifts only in trees from $B$ with roots in $S_0$: in particular, no tree of depth less than $h$ is ever shifted up by the algorithm.

At the beginning of each round, every designated component still meets the layer being processed, so that layer has size at least $22\alpha\ge 2(10+3\delta_1)\alpha$ and its threshold is well-defined. Each shift removes between $1$ and $\Delta^h$ vertices from that layer. The proof of Lemma~\ref{lem:layers}, with $h$ in place of $\imax$, therefore gives
\[
\delta_1^i\alpha/\Delta^h-1\le |S_i|\le2\delta_1^i\alpha
\qquad \text{ for every }i\in [h+1],
\]
and ensures termination for $\delta_1<1/(4\Delta^\imax)$ and sufficiently large $n$. In particular, at most $2\delta_1\alpha$ components are ever shifted, containing in total at most
\[
2\delta_1\alpha\,\Delta^{h+1}\le\delta_2\alpha
\]
vertices. Also, $|V(B)|\le\exp(\Delta^5)\Delta^{\imax+1}\alpha$, so the discretisation bound in \ref{pt(2)} holds, for example, with
\[
M_\Delta:=\left\lceil
\exp(\Delta^5)\Delta^{\imax+1}\delta_1^{-(\imax+1)}
\right\rceil.
\]
Let $T_{\rm small}^-$ be the retained forest, and let $L_1,\ldots,L_{h+1}$ be its final layers. Define $L_i^*$ to consist of the vertices of $L_i$ lying in components shifted at least once. (Since no such vertex lies in $L_1$, we have $L_1^*=\varnothing$.) 
Also, as only components originating from $A_1$ are shifted, $L_{h+1}^*\subseteq V(\tilde F_1)$.

Finally, set $L=\bigcup_{i=1}^h L_i$ and
\[
F_r:=\tilde F_r\setminus L
\qquad \text{ for every } r\in[3].
\]
Each of $(\tilde F_r)_{r=1}^3$ is obtained from a tree by deleting at most $\alpha$ auxiliary edges, so it has at most $\alpha+1$ components. Also,
Since deleting $L$ preserves the number of components (as only leaves are deleted iteratively), $(F_r)_{r=1}^3$ have at most $\alpha+1$ components each as well. 
Thus, the display in~\ref{pt(3)} follows readily, and all trees rooted at internal vertices of the paths in $T_{\rm paths}^-$ belong to $T_{\rm small}^-$. 
In particular, these paths are bare in $T\setminus L$, and furthermore $L_{h+1}^*\subseteq V(F_1)$ by definition, proving~\ref{pt(3)}.

Now, at least $11\alpha-\delta_2\alpha\ge (10+\delta_1)\alpha$ of the designated trees descending from $A_1$ were never shifted, and none in the trees descending from $A_2$ was shifted. 
For $r=1,2$, choose $U_r$ to consist of $(10+\delta_1)\alpha$ vertices $u_v$ with $v\in A_r$ for which $H_v$ was never shifted; in particular, $U_1,U_2$ stay disjoint from $L_{h+1}^*$. 
Since these trees have height $h$, their roots $u_v$ belong to $P_h\cap V(F_r)$.
Also, moving a marked root down its pendant tree cannot decrease its distance to another marked root, 
so the sets $U_r$ satisfy the separation requirement in~\ref{pt(4)}. 
Moreover, since $U_3$ consists of roots whose hanging trees are not included in $T_{\rm small}^-$, it is disjoint from every $P_i$; it also avoids the vertices $r_{1,2}$ and $r_{2,3}$ by construction. Thus,~\ref{pt(4)} holds.

It remains to verify the last point in~\ref{pt(2)}. 
Every nonroot vertex of a layer $L_i$ has exactly one neighbour outside $\bigcup_{j=1}^i L_j$, namely its parent, which lies in a higher layer.
For every $i\in [h-1]$, since each unshifted nonroot vertex of $L_{i+1}$ was exposed by leaf-cutting, it has a child in $L_i$. 
Hence, $L_{i+1}\setminus L_{i+1}^*\subseteq P_i$. If a vertex of $L_i^*$ has parent $v$, then $v$ belongs to the same retained shifted component, and consequently $v\in\bigcup_{j=i+1}^{h+1}L_j^*$. 
All children in $L_i$ of such a parent belong to that same component, and thus the total number of children in $L_i$ of the parents of $L_i^*$ is $|L_i^*|\le\delta_2\alpha$.
Finally, $|P_i|\ge |L_{i+1}|-|L_{i+1}^*|\ge 19\alpha$ for all $i\in [h-1]$ since $|L_{i+1}^*|\le \delta_2\alpha\le \alpha$ while $|P_h|\ge |U_1|+|U_2|\ge 20\alpha$. 
This completes the proof of~\ref{pt(2)} and the lemma.
\end{proof}

We are ready to show Theorem~\ref{thm:univ_super-adapted} by combining the previous lemma with tools already developed in previous sections.

\paragraph{Proof of Theorem~\ref{thm:univ_super-adapted}.}
Fix $h$ as in Lemma~\ref{lem:22-separated roots} and a vector $(x_1,\ldots,x_{h+1})$ which is one of the bounded number of possible sizes of layers in $T_{\rm small}^-$.
We begin by constructing the sets $(L_i')_{i=1}^{h},(L_i'')_{i=1}^{h}$, $(P_{r,h}',P_{r,h}'')_{r=1}^2$ and $Q_L',Q_P'$ similarly to Section~\ref{sec:absorb}. To avoid repetitions, we only briefly remind the reader their definitions, and clarify the necessary changes.

Let $p_1$ be such that $(1-p_1)^{3}=1-p$, and independently sample:
\begin{itemize}
    \item A random directed graph $D$ where all directed edges appear independently with probability $p_1$,
    \item and a random graph $G_1\sim G(n,p_1)$.
\end{itemize}
Write $D'$ to denote the underlying undirected graph of $D$, and recall that $G=G_1\cup D'\sim G(n,p)$.

Next, we construct sets following fixed deterministic rules using only information from $D$, in particular independently of $G_1$.
Choose $(L_i')_{i=1}^{h}$ to be an arbitrary sequence of disjoint vertex sets with sizes $(x_i-\Delta^{-2}\alpha)_{i=1}^{h}$, respectively.
For every $i\in [h]$, let $L_{1,i}'$ be the first $10\alpha$ vertices in $L_i'$, and set $L_{2,i}'\coloneqq L_{i}'\setminus L_{1,i}'$.
Furthermore, the sets $(L''_{1,i})_{i=1}^{h}$ and $(L''_{2,i})_{i=1}^h$ are subsets of $(L_{1,i}')_{i=1}^{h}$ and $(L_{2,i}')_{i=1}^{h}$, respectively, which expand well into each other in the directed graph $D$, and, respectively, each occupies most of the vertices of $L_{1,i}'$ or $L_{2,i}'$. 
The exact definitions of $L_{1,i}''$ and $L_{2,i}''$ are analogous to the ones in Proposition \ref{prop:props L'} (see also Remark~\ref{rem:QLprime}), where we recall that $\beta = n/(\ln n)^2$ and that $L'$ denotes the union of $(L_i')_{i=1}^{h}$.
We note that, for every $i\in \{2,\ldots,h\}$, when embedding a tree with the above layer sizes, its $L_i^*$-sets (consisting of the vertices of the shifted trees ending up on layer $L_i$) will be embedded into $L_{1,i}''$.
The set $Q_L'$ is defined to be the collection of vertices in $[n]\setminus L'$ which turned out to expand insufficiently towards  
some of the layers $(L_{1,i}'')_{i=1}^{h},(L_{2,i}'')_{i=1}^{h}$, that is, all vertices in $[n]\setminus L'$ with fewer than $|L_{r,i}'|p_1/2$ out-neighbours in $L_{r,i}''$ for some $(r,i)\in\{1,2\}\times [h]$. 
Then, Lemma \ref{lem:QL2} asserts that whp $|Q_{L}'|\leq \beta$. Set $L'':=(L''_{1,1}\cup L''_{2,1})\cup\ldots\cup (L''_{1,h}\cup L''_{2,h})$.

Next, we prepare the ground for the embedding of the parent set of $L_h$.
Following the discussion before Lemma \ref{lem:Q'}, the sets $P_{1,h}',P_{2,h}'$ are arbitrary disjoint subsets of $[n]\setminus (L'\cup Q_{L}')$, each of size $\alpha$.
In addition, $Q_P'$ is the collection of all $v\in [n]\setminus (L'\cup P_{1,h}'\cup P_{2,h}'\cup Q_L')$ with at most $\alpha p_1/2$ out-neighbours in $P_{1,h}'$ or in $P_{2,h}'$. 
By Lemma~\ref{lem:Q'}, whp $|Q'_P|\leq\beta$.
Then, the sets $P_{1,h}''$ and $P_{2,h}''$ are constructed similarly to the discussion before Lemma~\ref{lem:W}, which asserts that the construction is whp successful. 
More precisely, for all $i\in \{1,2\}$, expose the out-neighbourhood sizes $\{|N_D^-(v)\cap P_{i,h}'|: v\in L''\}$, arbitrarily order the vertices in $L''$ with less than $\alpha p_1/2$ out-neighbours in $P_{i,h}'$ and call this set $W'$.
For each vertex $v\in W'$, consecutively select $\alpha p_1/2$ out-neighbours of $v$ in 
$V' := [n]\setminus (L'\cup Q_L'\cup Q_P'\cup P_{1,h}'\cup P_{2,h}')$ outside the current union $P_{1,h}''\cup P_{2,h}''$, and add them to $P_{i,h}''$ (while keeping $P_{1,h}''$ and $P_{2,h}''$ disjoint). Then, complete $P_{i,h}''$ to be of size $\beta$ by arbitrarily adding new vertices from $V'\setminus(P_{1,h}''\cup P_{2,h}'')$. Finally, let $V'' \coloneqq V' \setminus (P_{1,h}''\cup  P_{2,h}'')$ and complete $(L_{1,i}''\cup L_{2,i}'')_{i=1}^{h}$ to pairwise disjoint sets $(R_i'')_{i=1}^{h}$ with sizes $(x_i)_{i=1}^{h}$, respectively, as defined in the beginning of Section~\ref{subsec:embedL}, by adding arbitrary vertices in $V''$.

We continue by designating sets of vertices where different parts of the trees will be embedded.
 First, set $R''\coloneqq R''_1\cup \ldots\cup R_{h}''$ and let $V_1\subseteq [n]\setminus R''$ be an arbitrary set of size $|V(F_1\cup F_2\cup F_3)|+n/K^{3}$ that contains $(P_{r,h}',P_{r,h}'')^{2}_{r=1}$ 
 as well as $(L'\setminus L'')\cup Q_L'\cup Q_P'$.
By Lemma~\ref{lem:22-separated roots}\ref{pt(3)},
\[
|V(F_1\cup F_2\cup F_3)|=n-\alpha(\ell-1)-\sum_{i=1}^{h} x_i,
\]
so the order of $F_1\cup F_2\cup F_3$ is fixed by the layer-size vector. 
Now, fix $V_1$ and also fix a partition of $[n]\setminus (V_1\cup R'')$ into $V_2,V_3$ with
\[
|V_2|=(1/(2K^2)-1/K^3)n,\qquad |V_3|=3(2k+1)\alpha.
\]
These outer sets are chosen separately for the fixed vector $(x_1,\ldots,x_{h+1})$ and do not depend on the orders of $F_1,F_2,F_3$. 
Write $Z=(L'\setminus L'')\cup Q_L'\cup Q_P'$ and choose $w_{1,2}$ in $V_1\setminus\bigl(Z\cup\bigcup_{r=1}^2(P_{r,h}'\cup P_{r,h}'')\bigr)$.
Fix $V_1^{1},V_{1}^{2},$ and $V_1^{3}$ with $V_{1}^{1}\cap V_1^{2}=\{w_{1,2}\}$ and $V_1^{1}\cap V_1^{3}=V_1^{2}\cap V_{1}^{3} = \varnothing$ covering $V_1$, that is, $V_1 = V_1^{1}\cup V_1^{2}\cup V_1^{3}$. 
In addition, require that $P_{r,h}'\cup P_{r,h}''\subseteq V_1^r$ for each $r\in [2]$, that $|V_1^r|\geq |V(F_r)|+n/K^4$ for each $r\in [3]$, and that $(L'\setminus L'')\cup Q_L'\cup Q_P'\subseteq V_1^3$.
Note that the subsets $V_1^1,V_1^2,V_1^3$ are chosen independently of $G_1$.
Since each $F_r$ has at most $\alpha+1$ components
and its designated set is $22$-separated, we may
apply Remark~\ref{rem:embed forest}.
If $|Z|<\beta$, add vertices inside $V_1^3$ to $Z$ until it reaches size $\beta$; this ensures that $|Z|\le3\beta$ and $|U_3|=30\beta\ge9|Z|$.
Indeed, the additional $15\alpha$ vertices fit into
the available $n/K^4$ buffer, since
\[
15\alpha+\alpha+3\beta+10dt+\ln n<n/K^4
\]
for sufficiently large $n$ and $C$.
Then, by the results in Section~\ref{sec:Gamma+1}, each of the following statements holds with probability $1-o(n^{-3})$ (for fixed orders of $F_1,F_2,F_3$ and for all choices of the forests with these orders):
\begin{itemize}
    \item There is an embedding of $F_1$ in $G_1[V_1^1]$ such that $P_{1,h}'\cup P_{1,h}''$ is covered by the image of $U_1\subseteq P_h\cap V(F_1)$ from Lemma~\ref{lem:22-separated roots}, and the vertex $r_{1,2} = F_1\cap F_2$ is mapped to~$w_{1,2}$.
    \item There is an embedding of $F_2$ in $G_1[V_1^2]$ such that $P_{2,h}'\cup P_{2,h}''$ is covered by the image of $U_2\subseteq P_h\cap V(F_2)$ from the same lemma, the vertex $r_{1,2} = F_1\cap F_2$ is mapped to $w_{1,2}$, and the image of $r_{2,3}=F_2\cap F_3$ is some vertex $w_{2,3}\in V_1^2$.
    \item There is an embedding of $F_3$ in $G_1[V_1^3\cup \{w_{2,3}\}]$ where the vertex $r_{2,3}$ is mapped to $w_{2,3}$ and the set $(L'\setminus L'')\cup Q_L'\cup Q_P'$ is covered by the image of $U_3$.
\end{itemize}
Crucially, note that the error probability $o(n^{-3})$ suffices to ensure that the above point holds uniformly for the $O(n^2)$ choices of sizes of $F_1$ and $F_2$.
In particular, with probability $1-o(n^{-1})$, for every choice of $F_1,F_2,F_3$, the forests $F_1,F_2,F_3$ admit compatible embeddings into $V_1^1,V_1^2,V_1^3\cup\{w_{2,3}\}$, respectively.
Furthermore, note that each embedding as above leaves uncovered a `buffer' set $W$, consisting of all vertices of $V_1$ outside the image of the embedding. While the set $W$ clearly depends on the embedding, its size is always $n/K^{3}$.

Our next step is to embed the paths in $T_{\rm paths}^{-}$ using the results from Section~\ref{sec:long bare paths}.
First, note that
\[
    |[n]\setminus (V_1\cup R'')|=|\{\text{internal vertices in }T_{\rm paths}^{-}\}|-n/K^{3} = (1/2K^{2}-1/K^{3})n+3(2k+1)\alpha,
\]
where the latter equality holds as each maximal path in $T_{\rm paths}^{-}$ has length $\ell$ (see \eqref{eq:k-ell-from-previous-sections}) and there are $\alpha$ many maximal paths in $T_{\rm paths}^{-}$. 
Conditionally on the previously fixed sets and the edges in $D$ outside $D[V_3]$ (which are sufficient for the construction of the sets), the estimates in the proof of Lemma~\ref{lem:random-graphs-adapted} with $\psi=1$ and $C$ sufficiently large show that whp $G_1[[n]\setminus R'']$ satisfies properties~\ref{def:iii}--\ref{def:vi} of Definition~\ref{def:adapted} for this fixed partition, with vertex quantifiers restricted to $[n]\setminus R''$.
Moreover, $V_3\cap L'=\varnothing$, since $L''\subseteq R''$ and $L'\setminus L''\subseteq V_1^3$. 
Recalling that no internal edges of $D[V_3]$ has been exposed, conditionally on the prepared sets and $G_1$, these edges remain independent Bernoulli variables with parameter $p_1$, and
\[
D'[V_3]\sim G(3(2k+1)\alpha,2p_1-p_1^2).
\]
Using only $G_1$ and the constructed partition, fix the terminal sets $A_1,A_2\subseteq V_3$ supplied by the matching construction of Lemma~\ref{lm:paths}, once for all forest orders and embeddings. 
Then, by Lemma~\ref{cor:linking-systems}, whp the graph $D'[V_3]$ contains an $(A_1,A_2;2k)$-linking system.
Then, repeating the proof of Theorem~\ref{thm:univ_adapted} with very minor modifications, presented at the end of Section \ref{sec:long bare paths}, on the existence of the said linking system, we may
use $D'[V_3]$ to extend every embedding of $F_1\cup F_2\cup F_3$ in $V_1$ as described above, to an embedding of $F_1\cup F_2\cup F_3\cup T_{\rm paths}^{-}$ in $[n]\setminus R''$.
We note that Lemma~\ref{lm:paths} (which is an important part of the proof of Theorem~\ref{thm:univ_adapted})
uses only properties~\ref{def:iii}--\ref{def:vi}
of Definition~\ref{def:adapted}, and therefore
the same matching construction applies to
the partition
$[n]\setminus R''=V_1\sqcup V_2\sqcup V_3$,
provided these properties hold for the restricted graph.
This is automatic for \ref{def:iii}, \ref{def:v} and \ref{def:vi} and holds whp for \ref{def:iv} when $C$ is sufficiently large.
Note that the existence of a linking system serves for all orders of forest, all forest embeddings and all endpoint pairings; only the bounded number of choices of heights and layer-size vectors requires a further union bound.

It now remains to embed the layers $(L_i)_{i=1}^{h}$ into $(R_i'')_{i=1}^{h}$. 
Let $\varphi$ denote the embedding constructed above, and set
\[
L_{r,h+1}''=\varphi(L_{h+1}\cap V(F_r))
\quad\text{for each $r=1,2$, and}\qquad
R_{h+1}''=\varphi(P_h).
\]
By construction,
\[
P_{r,h}'\cup P_{r,h}''
\subseteq \varphi(U_r)
\subseteq R_{h+1}''\cap L_{r,h+1}''
\quad\text{for each $r=1,2$, and}\qquad
\varphi(L_{h+1}^*)\subseteq L_{1,h+1}''.
\]
Moreover, $\varphi(P_h)$ avoids
$(L'\setminus L'')\cup Q_L'\cup Q_P'$,
since these vertices are covered by $\varphi(U_3)$.
 
Therefore, by applying the extension of Proposition~\ref{prop:expand} for $h$ layers described in Remark~\ref{rem:verbatim} iteratively, for every $s\in [h]$, 
whp we may extend the embedding so that $(L_i)^{h}_{i=s}$ are embedded into $(R_i'')_{i=s}^{h}$, respectively. This concludes the proof of Theorem~\ref{thm:univ_super-adapted}.

\subsection{\texorpdfstring{Universality for $\cT_3$}{Universality for trees with many many leaves}}\label{subsec:last_family}
It remains to ensure that typically all trees in $\cT_3$ can be simultaneously embedded in $G\sim G(n,C\ln n/n)$ for suitably large absolute constant $C$.
The proof for this family already essentially follows from considerations in~\cite{Mon19}: the argument of Case A presented there actually hides no dependence on $\Delta$ when the trees have a sufficiently large $20$-separated set of leaves of size $\Omega(\Delta\alpha)$.
For convenience and completeness, we present a sketch following closely Section~\ref{sec:Gamma+1} with significant simplifications of the approach presented there, 
which slightly differs from Case A in~\cite{Mon19}.

Set $\nu\coloneqq \exp(\Delta^5)\alpha$.
As a first step, using Lemma~\ref{lem:tree-part}, we divide the tree $T$ into subtrees $T_1,T_2,T_3$ such that:
\begin{itemize}
    \item $V(T_1)\cap V(T_2)=\{r_{1,2}\}$ and $V(T_2)\cap V(T_3)=\{r_{2,3}\}$ while $V(T_1)$ intersects $V(T_3)$ only if $r_{1,2}=r_{2,3}$,
    \item $\min\{|V(T_1)|,|V(T_2)|,|V(T_3)|\}\ge n/9$.
\end{itemize}
 Also, by choosing appropriately the tree to which Lemma~\ref{lem:tree-part} is applied for the second time, we can (and do) assume that  $T_1$ contains at least $\nu/3$ leaves.
The first main difference with Section~\ref{sec:Gamma+1} is that we consider $\Gamma=1$ and only two layers $L_1,L_2\subseteq V(T_1)$ of equal size such that: 
\begin{itemize}
    \item $L_2$ consists of $10\Delta\alpha$ vertices forming a $22$-separated set of parents of leaves in $T_1$,
    \item for every vertex $v$ in $L_2$, the set $L_1$ contains a single leaf of $T_1$ attached to $v$.
\end{itemize}
Indeed, having a number of vertices above the cutoff $\exp(\Delta^5)\alpha$ allows us to find $10\Delta\alpha$ leaves that are $24$-separated by a greedy procedure, while avoiding vertices near the roots $r_{1,2}$ and $r_{2,3}$. 
The parents of these leaves give the required $22$-separated set $L_2$. Set $\tilde T_1:=T_1\setminus L_1$.

Now, let $L_1'$ be the set of the first $5\Delta\alpha$ vertices in $[n]$. Define the set $Q_L'$ of vertices in $[n]\setminus L_1'$ which send less than $|L_1'|p_1/2$ edges in $D$ towards $L_1'$, where $D$ is the random directed graph defined in Section~\ref{sec:Gamma+1} and in the beginning of the proof of Theorem \ref{thm:univ_super-adapted}.
By Chernoff's bound applied similarly to Lemma~\ref{lem:QL2}, whp $|Q_L'|\le \beta$.
Next, define $P_1'$ to be an arbitrary set of size $\alpha$ which is disjoint from $L_1'\cup Q_L'$, 
and let $Q_P'$ be the set of vertices in $[n]\setminus(L_1'\cup Q_L'\cup P_1')$ which send fewer than $|P_1'|p_1/2$ edges towards $P_1'$ in $D$.
By Chernoff's bound applied similarly to Lemma~\ref{lem:Q'}, we have that whp $|Q_P'|\le \beta$.
Now, construct $P_1''$ as follows. For each vertex $v\in L_1'$ with fewer than $\alpha p_1/2$ out-neighbours in $P_1'$, select $\alpha p_1/2$ fresh out-neighbours outside $L_1'\cup Q_L'\cup Q_P'\cup P_1'\cup P_1''$ and add them to $P_1''$. Complete $P_1''$ to size $\beta$ outside the same forbidden sets. This is possible with high probability by the consecutive exposure argument from the proof of Lemma~\ref{lem:W}.
Finally, consider vertex subsets $V_1,V_2$ intersecting in one vertex with 
\begin{itemize}
    \item $V_1$ containing $P_1'\cup P_1''$ and the first $|V(\tilde T_1)|+n/100-|P_1'\cup P_1''|$ vertices in 
    \[[n]\setminus (L'_1\cup Q_L'\cup Q_P'\cup P_1'\cup P_1''),\]
    \item $V_2$ containing the first vertex $w_{1,2}$ in $V_1\setminus (P_1'\cup P_1'')$, the set $Q_L'\cup Q_P'$ as well as the following (in the fixed default ordering) $|V(T_2)|+n/100-1-|Q_L'\cup Q_P'|$ vertices in $[n]\setminus (L_1'\cup V_1\cup Q_L'\cup Q_P')$.
\end{itemize}

Recall the decomposition of $G$ into $G_1$ and $D'$. Then, by Lemma~\ref{cor:embedparentsfinal)3}, for any valid choice of sizes of $T_1$, $T_2$ and $T_3$, 
each of the following holds with probability $1-o(n^{-3})$: for all choices of trees $\tilde T_1,T_2,T_3$ with the required sizes,
\begin{itemize}
    \item there is an embedding of $\tilde T_1$ in $G_1[V_1]$ such that $P_1'\cup P_1''$ is covered by the image of $L_2$ and the vertex $T_1\cap T_2$ is mapped to $w_{1,2}$,
    \item there is an embedding of $T_2$ in $G_1[V_2]$ covering $Q_L'\cup Q_P'$ where the vertex $T_1\cap T_2$ is mapped to $w_{1,2}$ and $T_2\cap T_3$ is mapped to some vertex $w_{2,3}$,
    \item there exists an embedding of $T_3$ in the subgraph of $G_1$ induced by the complement to the union of $L_1'$ with the images of $(\tilde T_1\cup T_2)\setminus \{r_{2,3}\}$ where the vertex $r_{2,3}$ is mapped to $w_{2,3}$.
\end{itemize}

Note that the error probability $o(n^{-3})$ suffices to ensure that whp the above points hold uniformly for the $O(n^2)$ choices of sizes of $T_1,T_2,T_3$.
 Finally, by adding all vertices outside the images of $\tilde T_1,T_2,T_3$ to $L_1'$, thus forming a set $R_1''$ of size $10\Delta \alpha$, one can verify (similarly to Proposition~\ref{prop:expand}) that each of the following holds whp:
\begin{itemize}
    \item $G$ is a $t$-joined graph (where $t=4(\ln\ln n)/p_1$),
    \item every subset $U\subseteq R_1''$ with $|U|\leq t$ has at least $|U|$ neighbours in $P_1'\cup P_1''$,
    \item every set $U$ in the image of $L_2$ with $|U|\leq t$ has at least $|U|$ neighbours in $L_1'$.
\end{itemize}
By Lemma~\ref{lem:match-easy}, this implies that whp, for every tree $T$, the embedding of $\tilde T_1\cup T_2\cup T_3$ can always be extended to $L_1$ (and thus completed to a full embedding of $T$), thus implying that whp $G(n,p)$ is $\cT_3$-universal and concluding the proof of Theorem~\ref{thm:univ_trees}.

\section*{AI declaration} 
An earlier version of this manuscript was uploaded to arXiv on 8 August 2026 and was prepared without the use of AI tools.

In preparing the present version, the authors used ChatGPT-6 Astra to assist with language polishing and to review the mathematical content for potential errors. The tool provided feedback and suggested revisions, which were independently assessed and, where appropriate, implemented by the authors themselves. The resulting changes materially affected only a few short passages scattered throughout the manuscript.

The authors have verified the correctness and originality of all content, including the references, and take full responsibility for the manuscript.

\section*{Acknowledgements} The authors would like to thank Sahar Diskin for many useful discussions in the initial stages of the project, and Jeff Kahn for several comments and suggestions. 
Part of this work was carried out during research visits to the University of Sheffield and TU Wien. The authors are grateful to these institutions for their hospitality.

Cohen Antonir was supported in part by the Israel Science Foundation grant 2110/22 and 1028/16, by ERC Consolidator Grant 101044123 (RandomHypGra), and by ERC Starting Grant 633509. Lichev was supported by the Austrian Science Fund (FWF) grant No. 10.55776/ESP624. 
For open access purposes, the authors have applied a CC BY public copyright license to any author-accepted manuscript version arising from this submission.

\bibliographystyle{alpha}
\bibliography{Bib}

\newcommand{\etalchar}[1]{$^{#1}$}
\begin{thebibliography}{DMMC{\etalchar{+}}24}

\bibitem[ABS09]{ABS}
N.~Alon, B.~Bukh, and B.~Sudakov.
\newblock Discrete {K}akeya-type problems and small bases.
\newblock {\em Isr. J. Math.}, 174:285--301, 2009.

\bibitem[ABZZ25]{ABZZ}
N.~Alon, I.~Benjamini, G.~Zakharov, and M.~Zhukovskii.
\newblock Sums along the edges of bounded degree graphs.
\newblock {\em arXiv preprint \emph{arXiv:2507.01138}}, 2025.

\bibitem[AC07]{AlonCapalbo}
N.~Alon and M.~Capalbo.
\newblock Sparse universal graphs for bounded-degree graphs.
\newblock {\em Random Struct Algorithms}, 31:123--133, 2007.

\bibitem[ADI{\etalchar{+}}25]{ADILS25}
M.~Anastos, S.~Diskin, D.~Ignasiak, L.~Lichev, and Y.~Sha.
\newblock Spanning trees of bounded degree in random geometric graphs.
\newblock {\em arXiv preprint arXiv:2505.16818}, 2025.

\bibitem[ADK17]{ADK}
Stephen Alstrup, S\o{}ren Dahlgaard, and Mathias B\ae{}k~Tejs Knudsen.
\newblock Optimal induced universal graphs and adjacency labeling for trees.
\newblock {\em Journal of the ACM}, 64(4):1--22, 2017.

\bibitem[AKS83]{AKS}
M.~Ajtai, J.~Koml{\'o}s, and E.~Szemer{\'e}di.
\newblock An ${O(n \log n)}$ sorting network.
\newblock In {\em Proceedings of the fifteenth annual ACM symposium on Theory of computing}, pages 1--9, 1983.

\bibitem[AKS85]{Ham1}
M.~Ajtai, J.~Koml\'{o}s, and E.~Szemer\'{e}di.
\newblock First occurrence of {H}amilton cycles in random graphs.
\newblock {\em Annals of Discrete Mathematics}, 27:173--178, 1985.

\bibitem[AKS07]{AKS07}
N.~Alon, M.~Krivelevich, and B.~Sudakov.
\newblock Embedding nearly-spanning bounded degree trees.
\newblock {\em Combinatorica}, 27(6):629--644, 2007.

\bibitem[Alo10]{Alon-survey}
N.~Alon.
\newblock Universality, tolerance, chaos and order.
\newblock In Imre B\'{a}r\'{a}ny, J\'{o}zsef Solymosi, and G\'{a}bor S\'{a}gi, editors, {\em An irregular mind. Szemerédi is 70. Dedicated to Endre Szemerédi on the occasion of his seventieth birthday}, pages 21--37. Springer, Berlin, 2010.

\bibitem[Alo24]{Alon-implicit}
Noga Alon.
\newblock Implicit representation of sparse hereditary families.
\newblock {\em Discrete Comput. Geom.}, 72(2):476--482, 2024.

\bibitem[AS00]{AS2000}
Noga Alon and Joel~H. Spencer.
\newblock {\em The probabilistic method}.
\newblock Wiley-Interscience Series in Discrete Mathematics and Optimization. Wiley-Interscience [John Wiley \& Sons], New York, second edition, 2000.
\newblock With an appendix on the life and work of Paul Erd\H os.

\bibitem[AY93]{Alon-Yuster}
N.~Alon and R.~Yuster.
\newblock Threshold functions for ${H}$-factors.
\newblock {\em Combin. Probab. Comput.}, 2(2):137--144, 1993.

\bibitem[BCLR89]{BCLR}
Sandeep~N. Bhatt, F.~R.~K. Chung, F.~T. Leighton, and Arnold~L. Rosenberg.
\newblock Universal graphs for bounded-degree trees and planar graphs.
\newblock {\em SIAM Journal on Discrete Mathematics}, 2(2):145--155, 1989.

\bibitem[BCPS10]{BCPS10}
J.~Balogh, B.~Csaba, M.~Pei, and W.~Samotij.
\newblock Large bounded degree trees in expanding graphs.
\newblock {\em Electronic Journal of Combinatorics}, 17(1):R6, 2010.

\bibitem[BDS{\etalchar{+}}24]{BDSZZ}
\'{E}. Bonnet, J.~Duron, J.~Sylvester, V.~Zamaraev, and M.~Zhukovskii.
\newblock Small but unwieldy: A lower bound on adjacency labels for small classes.
\newblock {\em SIAM Journal on Computing}, 53(5):1578--1601, 2024.

\bibitem[Bel23]{Bel22}
T.~Bell.
\newblock The {P}ark-{P}ham theorem with optimal convergence rate.
\newblock {\em Electronic Journal of Combinatorics}, 30(2):P2.25, 2023.

\bibitem[BHK{\etalchar{+}}19]{BHKMPP19}
J.~B{\"o}ttcher, J.~Han, Y.~Kohayakawa, R.~Montgomery, O.~Parczyk, and Y.~Person.
\newblock Universality for bounded degree spanning trees in randomly perturbed graphs.
\newblock {\em Random Structures \& Algorithms}, 55(4):854--864, 2019.

\bibitem[BHK{\etalchar{+}}24]{BHKMP}
F.~Burghart, A.~Heckel, M.~Kaufmann, Noela M\"{u}ller, and Matija Pasch.
\newblock Sharp thresholds for factors in random graphs.
\newblock {\em arXiv preprint \emph{arXiv:2411.14138}}, 2024.

\bibitem[BL84]{chips}
S.~N. Bhatt and C.~E. Leiserson.
\newblock How to assemble tree machines.
\newblock In {\em Advances in Computing Research}, volume~2, pages 95--114. JAI Press, 1984.

\bibitem[Bol84]{Ham2}
B.~Bollob\'{a}s.
\newblock The evolution of sparse graphs.
\newblock In {\em Graph theory and combinatorics}, pages 35--57. Academic Press, London, 1984.

\bibitem[Bol98]{Bol98}
B{\'e}la Bollob{\'a}s.
\newblock {\em Modern graph theory}, volume 184.
\newblock Springer Science \& Business Media, 1998.

\bibitem[Cap99]{Capalbo}
M.~R. Capalbo.
\newblock A small universal graph for bounded-degree planar graphs.
\newblock In {\em Proceedings of the 10th Annual Symposium On Disc. Algorithms}, SODA '99, pages 156--160, 1999.

\bibitem[CG78]{ChungG}
F.~R.~K. Chung and R.~L. Graham.
\newblock On graphs which contain all small trees.
\newblock {\em J. Combin. Theory Ser. B}, 24:14--23, 1978.

\bibitem[CG79]{ChungG2}
F.~R.~K. Chung and R.~L. Graham.
\newblock On universal graphs.
\newblock {\em Ann. New York Acad. Sci.}, 319:136--140, 1979.

\bibitem[CG83]{ChungG3}
F.~R.~K. Chung and R.~L. Graham.
\newblock On universal graphs for spanning trees.
\newblock {\em Journal of the London Mathematical Society, Second Series}, 27(2):203--211, 1983.

\bibitem[CG98]{hitting-question}
F.~Chung and R.~Graham.
\newblock {\em Erd{\H{o}}s on graphs: His legacy of unsolved problems}.
\newblock A. K. Peters/CRC Press, 1998.

\bibitem[CK99]{CapalboR}
M.~Capalbo and S.~R. Kosaraju.
\newblock Small universal graphs.
\newblock In {\em Proceedings of the thirty-first annual ACM symposium on Theory of Computing}, STOC '99, pages 741--749, 1999.

\bibitem[CRS83]{circuits}
F.~R.~K. Chung, A.~L. Rosenberg, and L.~Snyder.
\newblock Perfect storage representations for families of data structures.
\newblock {\em SIAM J. Alg. Disc. Methods}, 4:548--565, 1983.

\bibitem[DK08]{DK08}
D.~Dellamonica{, Jr.} and Y.~Kohayakawa.
\newblock An algorithmic {F}riedman--{P}ippenger theorem on tree embeddings and applications.
\newblock {\em Electronic Journal of Combinatorics}, 15(1):R127, 2008.
\newblock 14 pp.

\bibitem[DKN22]{DKN22}
N.~Dragani{\'c}, M.~Krivelevich, and R.~Nenadov.
\newblock Rolling backwards can move you forward: on embedding problems in sparse expanders.
\newblock {\em Transactions of the American Mathematical Society}, 375(7):5195--5216, 2022.

\bibitem[DKP26]{DKP}
M.~Delcourt, T.~Kelly, and L.~Postle.
\newblock Thresholds for $(n,q,2)$-{S}teiner systems via refined absorption.
\newblock {\em Mathematical Proceedings of the Cambridge Philosophical Society}, pages 1--20, 2026.

\bibitem[DKRR08]{DKRRSODA}
D.~Dellamonica{, Jr.}, Y.~Kohayakawa, V.~R\"{o}dl, and A.~Ruci\'{n}ski.
\newblock Universality of random graphs.
\newblock In {\em Proceedings of the 19th Annual ACM-SIAM Symposium on Discrete Algorithms}, SODA '08, pages 782--788, 2008.

\bibitem[DKRR15]{DKRR}
D.~Dellamonica{, Jr.}, Y.~Kohayakawa, V.~R\"{o}dl, and A.~Ruci\'{n}ski.
\newblock An improved upper bound on the density of universal random graphs.
\newblock {\em Random Structures \& Algorithms}, 46(2):274--299, 2015.

\bibitem[DMMC{\etalchar{+}}24]{DMMPS24}
N.~Dragani{\'c}, R.~Montgomery, D.~Munh{\'a}~Correia, A.~Pokrovskiy, and B.~Sudakov.
\newblock Hamiltonicity of expanders: optimal bounds and applications.
\newblock {\em arXiv preprint arXiv:2402.06603}, 2024.

\bibitem[EP23]{DiazPerson}
A.~{Espuny D\'{i}az} and Y.~Person.
\newblock Spanning {$F$}-cycles in random graphs.
\newblock {\em Combinatorics, Probability and Computing}, 32(5):833--850, 2023.

\bibitem[ER60]{ER}
P.~Erd\H{o}s and A.~R\'{e}nyi.
\newblock On the evolution of random graphs.
\newblock {\em Math. Inst. Hung. Acad. Sci}, 5(1):17--61, 1960.

\bibitem[FKL19]{FKL19}
Asaf Ferber, Gal Kronenberg, and Kyle Luh.
\newblock Optimal threshold for a random graph to be 2-universal.
\newblock {\em Transactions of the American Mathematical Society}, 372(6):4239--4262, 2019.

\bibitem[FKNP21]{FraKahNarPar2021}
Keith Frankston, Jeff Kahn, Bhargav Narayanan, and Jinyoung Park.
\newblock Thresholds versus fractional expectation-thresholds.
\newblock {\em Ann. of Math. (2)}, 194(2):475--495, 2021.

\bibitem[FN18]{FerberNenadov}
A.~Ferber and R.~Nenadov.
\newblock Spanning universality in random graphs.
\newblock {\em Random Struct Alg}, 53:604--637, 2018.

\bibitem[FNP16]{FNP}
A.~Ferber, R.~Nenadov, and U.~Peter.
\newblock Universality of random graphs and rainbow embedding.
\newblock {\em Random Struct Alg}, 48(3):546--564, 2016.

\bibitem[FP87]{FP87}
J.~Friedman and N.~Pippenger.
\newblock Expanding graphs contain all small trees.
\newblock {\em Combinatorica}, 7(1):71--76, 1987.

\bibitem[Fri99]{Fri99}
Ehud Friedgut.
\newblock Sharp thresholds of graph properties, and the {$k$}-{SAT} problem.
\newblock {\em J. Amer. Math. Soc.}, 12(4):1017--1054, 1999.
\newblock With an appendix by Jean Bourgain.

\bibitem[GJK]{GJK-unpublished}
R.~Glebov, D.~Johannsen, and M.~Krivelevich.
\newblock Hitting time appearance of certain spanning trees in the random graph process.
\newblock Unpublished manuscript.

\bibitem[Gle13]{GJK13}
R.~Glebov.
\newblock {\em On {H}amilton cycles and other spanning structures}.
\newblock PhD thesis, Freie Universit{\"a}t Berlin, 2013.

\bibitem[Hax01]{Hax01}
P.~E. Haxell.
\newblock Tree embeddings.
\newblock {\em Journal of Graph Theory}, 36(3):121--130, 2001.

\bibitem[Hec21]{Heckel}
A.~Heckel.
\newblock Random triangles in random graphs.
\newblock {\em Random Structures \& Algorithms}, 59(4):616--621, 2021.

\bibitem[HKMP24]{HKMP24}
Annika Heckel, Marc Kaufmann, Noela M{\"u}ller, and Matija Pasch.
\newblock The hitting time of clique factors.
\newblock {\em Random Structures \& Algorithms}, 65(2):275--312, 2024.

\bibitem[HKS12]{HKS12}
D.~Hefetz, M.~Krivelevich, and T.~Szab{\'o}.
\newblock Sharp threshold for the appearance of certain spanning trees in random graphs.
\newblock {\em Random Structures \& Algorithms}, 41(4):391--412, 2012.

\bibitem[HMMPS25]{HMMP-S25}
J.~Hyde, N.~Morrison, A.~M{\"u}yesser, and M.~Pavez-Sign{\'e}.
\newblock Spanning trees in pseudorandom graphs via sorting networks.
\newblock {\em Proceedings of the American Mathematical Society}, 153(06):2353--2367, 2025.

\bibitem[JKS13]{JKS13}
Daniel Johannsen, Michael Krivelevich, and Wojciech Samotij.
\newblock Expanders are universal for the class of all spanning trees.
\newblock {\em Combinatorics, Probability and Computing}, 22(2):253--281, 2013.

\bibitem[JKV08]{JKV}
A.~Johansson, J.~Kahn, and V.~Vu.
\newblock Factors in random graphs.
\newblock {\em Random Structures \& Algorithms}, 33(1):1--28, 2008.

\bibitem[J{\L}R00]{JLR00}
S.~Janson, T.~{\L}uczak, and A.~Ruci{\'n}ski.
\newblock {\em Random graphs}.
\newblock Wiley-Interscience Series in Discrete Mathematics and Optimization. Wiley-Interscience, New York, 2000.

\bibitem[JP24]{Pham}
V.~Jain and H.~T. Pham.
\newblock Optimal thresholds for {L}atin squares, {S}teiner triple systems, and edge colorings.
\newblock In {\em Proceedings of the 2024 Annual ACM-SIAM Symposium on Discrete Algorithms}, SODA '24, pages 1425--1436. Society for Industrial and Applied Mathematics, 2024.

\bibitem[Kah22]{K22}
Jeff Kahn.
\newblock Hitting times for {S}hamir’s problem.
\newblock {\em Transactions of the American Mathematical Society}, 375(1):627--668, 2022.

\bibitem[Kah23]{Kahn-Shamir}
Jeff Kahn.
\newblock Asymptotics for {S}hamir's problem.
\newblock {\em Advances in Mathematics}, 422:109019, 2023.

\bibitem[Kee22]{Keevash}
P.~Keevash.
\newblock The optimal edge-colouring threshold.
\newblock {\em preprint arXiv:2212.04397}, 2022.

\bibitem[KKKW25]{KimKim}
Neel Kaul, Jaehoon Kim, Minseo Kim, and David Wood.
\newblock On universal graphs for trees and tree-like graphs.
\newblock {\em arXiv preprint \emph{arXiv:2511.22358}}, 2025.

\bibitem[KL14]{KimLee}
J.~H. Kim and S.~J. Lee.
\newblock Universality of random graphs for graphs of maximum degree two.
\newblock {\em SIAM Journal on Discrete Mathematics}, 28(3):1467--1478, 2014.

\bibitem[KLW16]{KLW16}
J.~Kahn, E.~Lubetzky, and N.~Wormald.
\newblock The threshold for combs in random graphs.
\newblock {\em Random Structures \& Algorithms}, 48(4):794--802, 2016.

\bibitem[KLW17]{KLW17}
J.~Kahn, E.~Lubetzky, and N.~Wormald.
\newblock Cycle factors and renewal theory.
\newblock {\em Communications on Pure and Applied Mathematics}, 70(2):289--339, 2017.

\bibitem[KN26]{KriNen}
M.~Krivelevich and R.~Nenadov.
\newblock Minors in small-set expanders.
\newblock {\em Proceedings of the American Mathematical Society}, 154:1407--1420, 2026.

\bibitem[KNP21]{KNP}
J.~Kahn, B.~Narayanan, and J.~Park.
\newblock The threshold for the square of a {H}amilton cycle.
\newblock {\em Proceedings of the American Mathematical Society}, 149(8):3201--3208, 2021.

\bibitem[KNR88]{KNR}
Sampath Kannan, Moni Naor, and Steven Rudich.
\newblock Implicit representation of graphs.
\newblock In {\em Proceedings of the Twentieth Annual ACM Symposium on Theory of Computing}, STOC '88, pages 334--343, New York, NY, USA, 1988. Association for Computing Machinery.

\bibitem[Knu98]{Knuth}
D.~E. Knuth.
\newblock {\em The art of computer programming, volume 3: sorting and searching}.
\newblock Addison Wesley Longman Publishing Co., Inc., second edition, 1998.

\bibitem[Kor76]{Korshunov}
A.~D. Korshunov.
\newblock Solution of a problem of {E}rd{\H{o}}s and {R}{\'{e}}nyi on {H}amiltonian cycles in nonoriented graphs.
\newblock {\em Soviet Math. Doklady}, 17:760--764, 1976.

\bibitem[Kri97]{Krivelevich-factors}
M.~Krivelevich.
\newblock Triangle factors in random graphs.
\newblock {\em Combinatorics, Probability and Computing}, 6:337--347, 1997.

\bibitem[Kri10]{Kri10}
M.~Krivelevich.
\newblock Embedding spanning trees in random graphs.
\newblock {\em SIAM Journal on Discrete Mathematics}, 24(4):1495--1500, 2010.

\bibitem[KS83]{Ham3}
J.~Koml\'{o}s and E.~Szemer\'{e}di.
\newblock Limit distribution for the existence of {H}amiltonian cycles in a random graph.
\newblock {\em Discrete Mathematics}, 43:55--63, 1983.

\bibitem[Mon14]{Montgomery-old}
R.~Montgomery.
\newblock Embedding bounded degree spanning trees in random graphs.
\newblock {\em arXiv preprint \emph{arXiv:1405.6559}}, 2014.

\bibitem[Mon19]{Mon19}
R.~Montgomery.
\newblock Spanning trees in random graphs.
\newblock {\em Advances in Mathematics}, 356:106793, 2019.

\bibitem[P\'76]{Posa}
L.~P\'{o}sa.
\newblock Hamiltonian circuits in random graphs.
\newblock {\em Discrete Math.}, 14:359--364, 1976.

\bibitem[Per25]{Per2025}
Will Perkins.
\newblock Searching for (sharp) thresholds in random structures: where are we now?
\newblock {\em Bull. Amer. Math. Soc. (N.S.)}, 62(1):113--143, 2025.

\bibitem[PP24]{PP24}
Jinyoung Park and Huy Pham.
\newblock A proof of the {K}ahn--{K}alai conjecture.
\newblock {\em Journal of the American Mathematical Society}, 37(1):235--243, 2024.

\bibitem[Rob55]{Robbins}
H.~Robbins.
\newblock A remark on {S}tirling's formula.
\newblock {\em The American Mathematical Monthly}, 62(1):26--29, 1955.

\bibitem[Ruc92]{Rucinski-factors}
A.~Ruci\'nski.
\newblock Matching and covering the vertices of a random graph by copies of a given graph.
\newblock {\em Discrete Math.}, 105(1-3):185--197, 1992.

\bibitem[VZ26]{VZ}
S.~Vakhrushev and M.~Zhukovskii.
\newblock Spanning triangulations in random graphs.
\newblock {\em preprint arXiv:2605.20361}, 2026.

\bibitem[Zhu25]{Z}
M.~Zhukovskii.
\newblock Sharp thresholds for spanning regular subgraphs.
\newblock {\em preprint arXiv:2502.14794}, 2025.

\end{thebibliography}

\appendix

\section{Proof of Theorem~\ref{thm:linking_system} and Lemma~\ref{cor:linking-systems}}
 \label{sec:proofthm2}

We start with a corollary which follows by a simple union bound from Theorem \ref{prop:path-systems-new}. For this, we recall that for a function $\psi$ with $\psi(n)\in \left[1,\frac{\ln n}{4\ln\ln n}\right]$ we write $k^*=k^*(n,\psi)=\lceil C^*\ln n/(\psi\ln \ln n)\rceil$.

\begin{corollary}\label{cor:many-functions}
     For every $\xi>0$ there is $n_0$ such that for every $n\geq n_0$ the following holds. For every $i\in [\ln\ln n]$, fix real numbers $1\leq \psi_{i}\le \ln n/(4\ln\ln n)$, set $p_i \coloneqq C_0 (\ln n)^{\psi_i}/n$, and sample $G_i\sim G(n,p_i)$ (not necessarily independently).
     In addition, let $k_i\geq k^*(n,\psi_i)$ be some integers and assume that $m_i\coloneqq n/(k_i+1)$ are integers. Then for every disjoint $m_i$-tuples $\mathbf{a}_i,\mathbf{c}_i$ of different vertices from $[n]$, 
        \[
            \bP\left(\text{$G_i$ contains an $(\mathbf{a}_i,\mathbf{c}_i;k_i)$-path system for every $i\in [\ln\ln n]$}\right) \ge 1-\xi.
        \]
\end{corollary}

We now show Lemma~\ref{cor:linking-systems},  whose proof also yields Theorem~\ref{thm:linking_system} by taking $\delta=1$ and using two path systems of lengths $\lceil k/2\rceil$ and $\lfloor k/2\rfloor$ for a desired depth $k$.
We restate the lemma for the convenience of the reader.

\Lemmathm*

\begin{proof}
Consider the sequence of functions $\psi_i \colon [3,\infty)\to \mathbb{R}$ defined for every non-negative integer $i$:
\[
        \psi_i (x)\coloneqq \max\left\{1,\frac{\ln x}{4\e^{i}\ln\ln x}\right\}.
    \]
    For every integer $n\geq 3$, let $s=s(n)$ be the first index $i$ such that $\psi_{i}(n)=1$. Note that for every~$n$ we have $s(n)\leq\lceil \ln\ln n-\ln\ln\ln n - \ln 4\rceil$. 
    From here onwards, fix an arbitrarily small $\xi>0$ and a large integer $n=n(\delta,\xi)$, with the exact dependencies to be specified throughout the proof.
    Further, define the function $\psi_*\equiv \psi_i$ 
    where $i$ is the smallest index $i\in[s]$ such that $\psi(n)\geq\psi_i(n)$. Then we get $\psi(n)\in[\psi_i(n),\e\psi_{i}(n)]$. We note that, for any $n'\neq n$, $\psi_*(n')=\psi_i(n')$, where $i$ is fixed and does not depend on $n'$. 
    Lastly, set $n_1\coloneqq (k+1) m \in [\delta n/2,2\delta n/3]$, 
    $S_1 \coloneqq [n_1]$, and $S_2\coloneqq [\delta n]\setminus [n_1-m]$.
    Let $\mathbf{c}$ be an ordered collection of all vertices of $S_1\cap S_2$, and without loss of generality assume $\mathbf{a}\subseteq S_1\setminus S_2$ and $\mathbf{b}\subseteq S_2\setminus S_1$.

    Let $n_0(\xi)$ be as given by Corollary \ref{cor:many-functions}, assume that $n\geq 2n_0/\delta$ and that $n$ is large enough such that for every $x>0$ we have
    \begin{align} \label{eq:ln-lnln-monotone1}
	2\le \frac{\ln (\delta n/2)}{\ln\ln (\delta n/2)} &\le \frac{\ln n_1}{\ln\ln n_1}< \frac{\ln(n_1+x)}{\ln\ln (n_1+x)},
    \end{align}
    and
    \begin{align}\label{eq:ln-lnln-monotone2}
    s+1\le \ln\ln n_1 \quad \text{and}\quad \frac{\ln (n)}{\ln\ln (n)}&\leq \frac{2\ln (\delta n/2)}{\ln\ln (\delta n/2)}.
\end{align}
We then apply Corollary \ref{cor:many-functions} to $\psi_0,\ldots,\psi_s$, with $n$ replaced by $n_1$, which is allowed due to the former inequality in \eqref{eq:ln-lnln-monotone2} as well as $1\le \psi_i(n_1) \le \frac{\ln n_1}{4\ln\ln n_1}$ for every $i$.
Within the notation of the corollary, we get that for every $k_i\geq k^*(n_1,\psi_i)$, $p_i=p_i(n_1)$, and every pair of disjoint $m_i$-tuples $\mathbf{a}_i,\mathbf{c}_i$ of different vertices from $[n_1]$, we have
\begin{equation}
\label{eq:path-systems-s}
\bP\left(\text{$G(n_1,p_i)$ contains an $(\mathbf{a}_i,\mathbf{c}_i;k_i)$-path system for every $i\leq s$}\right) \ge 1-\xi.
\end{equation}
Since $k\geq 4\e k^*(n,\psi)$ and since $\psi(n)\le \e \psi_*(n)$, we have
\[
        k\geq   \frac{4\e C^*\ln n}{\psi(n) \ln\ln n}  \stackrel{\eqref{eq:ln-lnln-monotone1}}\geq  \frac{4 C^* \ln n_1}{\psi_*(n) \ln\ln n_1} .
\]
Moreover, as $\psi_*(x) = \max\left\{1, \frac{\ln x}{4\e^{i}\ln\ln x}\right\}$ for some $i$, by the assumptions on $n$ we have
$$
\psi_*(n)\stackrel{\eqref{eq:ln-lnln-monotone2}}\leq 2\psi_*\left({\delta n}/{2}\right) \stackrel{\eqref{eq:ln-lnln-monotone1}}\leq 2\psi_*(n_1),
$$ 
which in turn implies that $k\geq  \frac{2C^*\ln n_1}{ \psi_*(n_1) \ln\ln n_1}$. Therefore, 
\[
	  k\geq\left\lceil \frac{C^*\ln n_1}{\psi_* (n_1) \ln\ln n_1}\right \rceil = k^*(n_1,\psi_*).
   \]

For every $i\in [s]$, write $\lambda_i\coloneqq 4^{-1}\e^{-i}\in (0,1]$ and let $\varphi_i(x)\coloneqq   \frac{(\ln x)^{\psi_i(x)}}{x}$. In addition, for the unique $i$ for which $\psi_i = \psi_*$, let $\varphi_* \coloneqq \varphi_i$, $\lambda_*\coloneqq \lambda_i$, and $p_*=p_i$ from Corollary \ref{cor:many-functions}. We now claim that $p\geq p_*(n_1)$.\
First, if $\psi_*(n_1) > 1$ then as $n\geq n_1$ we also have $\psi_*(n)>1$ by assumption \eqref{eq:ln-lnln-monotone1}. Therefore, $\varphi_*(n)= n^{-1+\lambda_*}$ and $\varphi_*(n_1)= n_1^{-1+\lambda_*}$. As $p\geq 2C_0\varphi_*(n)/\delta$ we have
\[
\frac{p}{p_*(n_1)} \geq \frac{2C_0 \varphi_*(n)/\delta}{C_0 \varphi_*(n_1)} = \frac{2}{\delta}\cdot \left(\frac{n_1}{n}\right)^{1-\lambda_*} \geq  \left(\frac{\delta}{2}\right)^{-\lambda_*} >1,
\]
where the penultimate inequality holds as $n_1/ n \ge \delta/2$ and as $\lambda_i\le 1$.
On the other hand, if $\psi_*(n_1) = 1$, then, as $p=2C_0(\ln n)^{\psi(n)}/\delta n\geq 2C_0(\ln n)/\delta n$, we have
\[
\frac{p}{p_*(n_1)} =\frac{pn_1}{C_0\ln n_1}\ge \frac{2 n_1}{\delta  n} \cdot \frac{\ln n}{\ln n_1} \geq 1,
\]
where the last inequality holds as $n_1\in [\delta n/2, n]$. 
Therefore, for every pair of disjoint $m$-tuples $\mathbf{x},\mathbf{y}$ of different vertices from $[n_1]$, we have
     \[
        \bP\left(\text{$G(n_1,p)$ contains an $(\mathbf{x},\mathbf{y};k)$-path system}\right) \ge 1-\xi,
    \]    
    where we applied~\eqref{eq:path-systems-s} with $\psi_i=\psi_*$.
    In particular, letting $G\sim G(\delta n,p)$ and noting that $G[S_1],G[S_2]\sim G(n_1,p)$, we have
    \[
         \bP\left(\text{$G[S_1]$ contains an $(\mathbf{a},\mathbf{c};k)$-path system}\right) \ge 1-\xi,
    \]
    and 
    \[
        \bP\left(\text{$G[S_2]$ contains a $(\mathbf{c},\mathbf{b};k)$-path system}\right)\geq 1-\xi.
    \]
	Finally, as $\xi$ was chosen arbitrarily, by applying Lemma \ref{lem:combaining path systems}, the proof is concluded.
\end{proof}

\begin{remark}
    We wish to emphasise that Remark \ref{remark:uniform bound} follows from the above proof. Indeed, all lower bounds on $n$ in the proof can be chosen in terms of $\xi$ and a positive lower bound on $\delta$
    only. Moreover, the argument in Lemma~\ref{lem:combaining path systems} gives a linking-system failure probability of at most $2\xi$. Hence, the conclusion of Lemma \ref{cor:linking-systems} holds uniformly for $\delta$ in
    any fixed compact subset of $(0,1]$.
\end{remark}

\end{document}